\documentclass[a4paper]{amsart}

\usepackage{verbatim}
\usepackage[notcite, notref]{}
\usepackage[dvips]{color}
\usepackage{amssymb}
\usepackage[active]{srcltx}
\usepackage[colorlinks,pdfpagelabels,pdfstartview = FitH,bookmarksopen = true,bookmarksnumbered = true,linkcolor = blue,plainpages = false,hypertexnames = false,citecolor = red] {hyperref}
\usepackage{comment}
\usepackage{todonotes}

\allowdisplaybreaks

\title[Nonlocal equations with measure data]{Nonlocal equations with measure data}

\author[Kuusi]{Tuomo Kuusi}
\address{Tuomo Kuusi\\Aalto University
Institute of Mathematics
\\ P.O. Box 11100
FI-00076 Aalto,
Finland}
\email{tuomo.kuusi@aalto.fi}

\author[Mingione]{Giuseppe Mingione}
\address{Giuseppe Mingione\\Dipartimento di Matematica e Informatica, Universit\`a di Parma\\
Parco Area delle Scienze 53/a, Campus, 43100 Parma, Italy}
\email{giuseppe.mingione@unipr.it.}

\author[Sire]{Yannick Sire}
\address{Yannick Sire\\Universit\'e Aix-Marseille, I2M\\
Technopole de Chateaux-Gombert, Marseille, France}
\email{sire@cmi.univ-mrs.fr}

\newtheorem{theorem}{Theorem}[section]

\newtheorem{lemma}{Lemma}[section]
\newtheorem{cor}{Corollary}[section]
\newtheorem{corollary}{Corollary}[section]
\theoremstyle{definition}
\newtheorem{definition}{Definition}
\newtheorem{remark}{Remark}[section]

\numberwithin{equation}{section}

\def\eqn#1$$#2$${\begin{equation}\label#1#2\end{equation}}
\def\charfn_#1{{\raise1.2pt\hbox{$\chi_{\kern-1pt\lower3pt\hbox{{$\scriptstyle#1$}}}$}}}
\newcommand{\rif}[1]{(\ref{#1})}
\newcommand{\trif}[1] {\textnormal{\rif{#1}}}

\newcommand{\eps}{\varepsilon}
\newcommand{\tail}{{\rm Tail}}

 \DeclareMathOperator*{\osc}{osc}

\def\diam{\operatorname{diam}}
\def\dist{\operatorname{dist}}

\newcommand{\divo}{\textnormal{div}}

 \DeclareMathOperator*{\esssup}{ess\,sup}
\DeclareMathOperator*{\essinf}{ess\,inf}

\def\en{\mathbb N}
\def\er{\mathbb R}
\newcommand{\ern}{\mathbb{R}^n}

\def\loc{\operatorname{loc}}

\newcommand{\lp}{\mathcal{L}_{\Phi} }

\newcommand{\ratio}{\nu, L}

\def\mean#1{\mathchoice%
          {\mathop{\kern 0.2em\vrule width 0.6em height 0.69678ex depth -0.58065ex
                  \kern -0.8em \intop}\nolimits_{\kern -0.4em#1}}%
          {\mathop{\kern 0.1em\vrule width 0.5em height 0.69678ex depth -0.60387ex
                  \kern -0.6em \intop}\nolimits_{#1}}%
          {\mathop{\kern 0.1em\vrule width 0.5em height 0.69678ex
              depth -0.60387ex
                  \kern -0.6em \intop}\nolimits_{#1}}%
          {\mathop{\kern 0.1em\vrule width 0.5em height 0.69678ex depth -0.60387ex
                  \kern -0.6em \intop}\nolimits_{#1}}}

\def\vintslides_#1{\mathchoice%
          {\mathop{\kern 0.1em\vrule width 0.5em height 0.697ex depth -0.581ex
                  \kern -0.6em \intop}\nolimits_{\kern -0.4em#1}}%
          {\mathop{\kern 0.1em\vrule width 0.3em height 0.697ex depth -0.604ex
                  \kern -0.4em \intop}\nolimits_{#1}}%
          {\mathop{\kern 0.1em\vrule width 0.3em height 0.697ex depth -0.604ex
                  \kern -0.4em \intop}\nolimits_{#1}}%
          {\mathop{\kern 0.1em\vrule width 0.3em height 0.697ex depth -0.604ex
                  \kern -0.4em \intop}\nolimits_{#1}}}

\newcommand{\aveint}[2]{\mathchoice%
          {\mathop{\kern 0.2em\vrule width 0.6em height 0.69678ex depth -0.58065ex
                  \kern -0.8em \intop}\nolimits_{\kern -0.45em#1}^{#2}}%
          {\mathop{\kern 0.1em\vrule width 0.5em height 0.69678ex depth -0.60387ex
                  \kern -0.6em \intop}\nolimits_{#1}^{#2}}%
          {\mathop{\kern 0.1em\vrule width 0.5em height 0.69678ex depth -0.60387ex
                  \kern -0.6em \intop}\nolimits_{#1}^{#2}}%
          {\mathop{\kern 0.1em\vrule width 0.5em height 0.69678ex depth -0.60387ex
                  \kern -0.6em \intop}\nolimits_{#1}^{#2}}}
                  
\newtoks\by
\newtoks\paper
\newtoks\book
\newtoks\jour
\newtoks\yr
\newtoks\pages
\newtoks\vol
\newtoks\publ
\def\et{ \& }
\def\name[#1, #2]{#1 #2}
\def\ota{{\hbox{\bf ???}}}
\def\cLear{\by=\ota\paper=\ota\book=\ota\jour=\ota\yr=\ota
\pages=\ota\vol=\ota\publ=\ota}
\def\endpaper{\the\by, \textit{\the\paper},
{\the\jour} \textbf{\the\vol} (\the\yr), \the\pages.\cLear}
\def\endbook{\the\by, \textit{\the\book},
\the\publ, \the\yr.\cLear}
\def\endpap{\the\by, \textit{\the\paper}, \the\jour.\cLear}
\def\endproc{\the\by, \textit{\the\paper}, \the\book, \the\publ,
\the\yr, \the\pages.\cLear}

\begin{document}
\begin{abstract}
We develop an existence, regularity and potential theory for nonlinear integrodifferential equations involving measure data. The nonlocal elliptic operators considered are possibly degenerate and cover the case of the fractional $p$-Laplacean operator with measurable coefficients. We introduce a natural function class where we solve the Dirichlet problem, and prove basic and optimal nonlinear Wolff potential estimates for solutions. These are the exact analogs of the results valid in the case of local quasilinear degenerate equations established by Boccardo \& Gallou\"et \cite{BG1, BG2} and Kilpel\"ainen \& Mal\'y \cite{KM1, KM2}. As a consequence, we establish a number of results which can be considered as basic building blocks for a nonlocal, nonlinear potential theory: fine properties of solutions, Calder\'on-Zygmund estimates, continuity and boundedness criteria are established via Wolff potentials. 
A main tool is the introduction of a global excess functional that allows to prove a nonlocal analog of the classical theory due to Campanato \cite{camp}. Our results cover the case of linear nonlocal equations with measurable coefficients, and the one of the fractional Laplacean, and are new already in such cases.
\end{abstract}
\maketitle
\tableofcontents 

\section{Introduction and results}
Nonlocal operators have attracted increasing attention over the last years. Apart from their theoretical interest, and the new mathematical phenomena they allow to observe, they intervene in a quantity of applications and models since they allow to catch more efficiently certain peculiar aspects of the modelled situations. For instance, we mention their use in quasi-geostrophic dynamics \cite{CVa}, nonlocal diffusion and modified porous medium equations \cite{BV, V, V2}, dislocation problems \cite{CG}, phase transition models \cite{AB, CSm}, image reconstruction problems \cite{GO}. For this reason it is particularly important to study situations when such nonlocal operators are involved in equations featuring singular or very irregular data, as for instance those modelling source terms which are concentrated at points. This leads to study nonlocal equations having measures as data. In this paper we establish basic facts about existence and regularity properties of solutions to possibly nonlinear, nonlocal equations having a right hand side which is a measure; the equations considered here are allowed to be degenerate. More precisely we shall consider nonlocal elliptic equations, and related Dirichlet problems, that read as   
\eqn{base}
$$-\mathcal{L}_{\Phi} u= \mu \qquad \mbox{in}\  \Omega \subset \er^n\;,$$
where $\Omega$ is a bounded open subset for $n\geq 2$, $-\lp$ is a nonlocal operator defined by
\eqn{duality}
$$
\langle -\mathcal{L}_{\Phi} u, \varphi\rangle  := \int_{\ern}\int_{\ern} \Phi(u(x){-}u(y))(\varphi(x)-\varphi(y)) K(x,y) \, dx \, dy\,,
$$
for every smooth function $\varphi$ with compact support. In \rif{base} it is assumed that $\mu$ belongs to $\mathcal M(\ern)$, 
that is the space of Borel measures with finite total mass on $\ern$. The function $\Phi : \er \mapsto \er$ is assumed to be continuous, satisfying $\Phi(0) = 0$ together with the monotonicity property
\eqn{monophi}
$$
\Lambda^{-1}|t|^{p} \leq  \Phi(t) t \leq \Lambda|t|^{p}\,, \qquad \forall\  t \in \er \,.
$$
Finally, the kernel $K\colon \ern \times \ern \to \er$ is assumed to be measurable, and satisfying the following ellipticity/coercivity properties: 
\eqn{thekernel}
$$ \frac{1}{\Lambda |x{-}y|^{n+sp}} \leq K(x,y)  \leq \frac{\Lambda}{|x{-}y|^{n+sp}} \qquad \forall\  x,y \in \ern, \ x \not= y\,$$
where $\Lambda \geq 1$ and
\eqn{basicbounds}
$$
s \in (0,1)\,, \qquad p > 2- \frac{s}{n}=:p_* \,.
$$
The lower bound $p>p_*$ comes from the fact that we are considering elliptic problems involving a measure, and for $s=1$ it recovers the necessary lower bound required in the classical local case to get Sobolev solutions; see also Section \ref{brief} below. Assumptions \rif{monophi}-\rif{basicbounds} make indeed $-\lp$ an elliptic operator. Moreover, \rif{thekernel} tells us the natural domain of definition of $-\lp$ is the fractional Sobolev space $W^{s,p}(\er^n)$ in the sense that this is the largest space to which $\varphi$ has to belong to in order to make the duality in \rif{duality} finite when $u \in W^{s,p}(\ern)$. We recall that the definition of $W^{s,p}(\mathcal O)$ for $s \in (0,1)$ and $p \geq 1$ and any open subset $\mathcal O \subset \ern$ is 
\[
W^{s,p}(\mathcal O) := \left\{ v \in L^{p}(\mathcal O) \; : \;[v]_{s,p;\mathcal O} < \infty \right\} 
\]
with
$$
[v]_{s,p;\mathcal O}^p:=\int_{\mathcal O} \int_{\mathcal O} \frac{|v(x){-}v(y)|^p}{|x{-}y|^{n+sp}} \, dx \, dy
\quad \mbox{and} \quad \|v\|_{W^{s,p}(\mathcal O)}:= \|v\|_{L^{p}(\mathcal O)}+[v]_{s,p;\mathcal O}\;. 
$$
Notice that, upon taking the special case $\Phi(t)= |t|^{p-2}t$, we recover the fractional $p$-Laplacean operator with measurable coefficients (see \cite{BCF, DKP1, DKP2}). On the other hand, in the case $\Phi(t)= t$ we cover the special case of linear fractional operators with measurable coefficients $\mathcal L$ defined by
\eqn{linearn}
$$
\langle -\mathcal{L} u, \varphi\rangle  := \int_{\ern}\int_{\ern} (u(x){-}u(y))(\varphi(x)-\varphi(y)) K(x,y) \, dx \, dy\,. 
$$ 
For these equations see also \cite{BCI, CS, CS2}.  
Finally, when $K(x,y)= |x{-}y|^{-(n+ps)}$ in \rif{linearn} we recover the case of the classical fractional Laplacean $\mathcal L = (-\triangle)^s$. The forthcoming results are new already in such cases.  
In connection to the equation \rif{base} we shall consider the related Dirichlet problems, that is those of the form 
\begin{equation} \label{eq}
\left\{
\begin{array}{rl}
 - \mathcal{L}_{\Phi} u =  \mu & \mbox{in } \; \Omega    \\
  u =  g &  \mbox{in } \, \ern \setminus \Omega \,,
\end{array}
\right.
\end{equation}
where in general the ``boundary datum" $g \in W^{s,p}(\er^n)$ must be prescribed on the whole complement of $\Omega$. In this case, and when $\Phi(t)= |t|^{p-2}t$ and $\mu=0$, we are essentially considering the Euler-Lagrange equation of the functional 
$$
v \mapsto \int_{\ern}\int_{\ern} |v(x){-}v(y)|^p K(x,y) \, dx \, dy
$$
minimised in the class of functions such that $v =  g$ outside $\Omega$. See Remark \ref{solvability} below. In this paper we shall introduce a natural function class allowing for solvability of the Dirichlet problem \rif{eq}; in particular, we shall introduce a suitable notion of solution to equations of the type \rif{base}. Such solutions, called SOLA (Solutions Obtained as Limits of Approximations) and constructed via an approximation procedure with problems involving more regular data, do not in general lie in the natural energy space associated to the operator $-\lp$, that is $W^{s,p}$, but exhibit a lower degree of integrability and differentiability. See Definition \ref{soladef} below. This is in perfect analogy with what happens in the case of classical, local measure data problems - see next Section \ref{brief}. For such reasons these solutions should be considered as the analog of the {\em very weak solutions} usually considered in the classical case. The degree of regularity of solutions will in anyway coincide with that of the fundamental solution of the fractional Laplacean equation
\eqn{fraclap}
$$
(- \triangle)^s u =\mu
$$
with $\mu\equiv \delta$ being the Dirac measure charging the origin. In this case $u(x) \approx |x|^{2s-n}$. Note that it will be indeed part of the game to determine a reasonable way of prescribing the boundary datum $g$ from \rif{eq} when defining the notion of approximate solution we are going to work with. We recall that \rif{base} must be primarily intended in the distributional sense, so that the first requirement that SOLA are bound to satisfy is being solutions in the sense of distributions. We will expand on this at the beginning of Section \ref{SOLAsec} below. 

After having proved the existence and basic regularity results for solutions we then pass to examine their pointwise and fine behaviour. In a next step we proceed to the main results of the paper, that is we produce local potential estimates for solutions; see Section \ref{pot sec} below. These allow to describe the pointwise behaviour of solutions - finiteness, possibility of defining the precise representative, continuity etc - via natural and optimal Wolff potentials. These results are  the sharp counterpart of those available in the classical potential theory, and in the more recently developed nonlinear one \cite{HKM, KM1, KM2}. We note that a few interesting existence and regularity results for the specific equation \rif{fraclap} in the case $s>1/2$ have been obtained in \cite{KPU} and, in a different setting in \cite{KR}. More recent work in \cite{CV} deals again with fractional equations involving measures, this time for any $s>0$, while the operator is given by the fractional Laplacean, and the analysis is carried out by means of fundamental solutions. A notion of renormalised solution for semilinear equations is proposed in \cite{ABB}. 

In order to better describe how far the results presented in this paper extend to the fractional setting those available in the classical case, we shall give in the next section a short review of some of the key point results of the classical local theory. For general notation we refer to Section \ref{notationsec} below, that we recommend to at this point. 

\subsection{Brief review of the classical local theory}\label{brief} An existence and regularity theory for general quasilinear equations involving measures has been established by Boccardo \& Gall\"ouet in a series of papers \cite{BG1, BG2, BG3}. The authors deal with Dirichlet problems of the type 
\eqn{Dir1}
$$
\left\{
    \begin{array}{cc}
    -\divo \ a(x,Du)=\mu & \qquad \mbox{in $\Omega$}\\[2 pt]
        u= 0&\qquad \mbox{on $\partial\Omega$}\,,
\end{array}\right.
$$
where the vector field $a(x,Du)$ has $p$-growth and coercivity with respect to the gradient, and exhibits a measurable dependence on $x$. The main model case here is given by the $p$-Laplacean operator with measurable coefficients $c(x)$
\eqn{caso}
$$
a(x,Du) = c(x)|Du|^{p-2}Du\,, \qquad 0 < 1/\Lambda \leq c(x) \leq \Lambda\,.
$$
For the sake of simplicity, from now on all the results for \rif{Dir1} will be stated in the case \rif{caso} holds. Under the optimal assumption $p > 2-1/n$, Boccardo \& Gall\"ouet introduce the notion of SOLA, that are defined as distributional solutions which have been obtained as limits (a.e. and in $L^{p-1}$) of 
a sequence of $W^{1,p}$-solutions $\{u_j\}$ of problems 
$$
\left\{
    \begin{array}{cc}
    -\divo \ a(x,Du_j)=\mu_j & \qquad \mbox{in $\Omega$}\\[2 pt]
        u_j= 0&\qquad \mbox{on $\partial\Omega$}\;,
\end{array}\right.
$$
where the sequence $\{\mu_j\}\subset C^{\infty}(\Omega)$ converges to $\mu$ weakly in the sense of measures. 
The final outcome is the existence of a distributional solution $u$ to the original problem satisfying
\eqn{bg} 
$$
u \in W^{1, q}(\Omega)\,, \qquad \forall \ q < \frac{n(p-1)}{n-1}\,.
$$
Solutions are not in general energy solutions, as shown by the nonlinear Green's function $G_P(x) \approx |x|^{-(p-1)/(n-p)}$, which solves $-\divo (|Du|^{p-2}Du)=\delta$ for $n\not=p$, and that does not belong to $W^{1,p}_{\loc}$. It instead precisely satisfy \rif{bg}.  For this reason such solutions are often called {\em very weak solutions}; for more information about degenerate problems with measure data we refer the reader to \cite{milan}. We remark that different notions of solutions have been proposed \cite{elenco, BGO, DMOP, Lind}, also in order to prove unique solvability (which is still an open problem). Such definitions are all equivalent in the case the measure $\mu$ is nonnegative, as eventually shown in \cite{KKT}. Following the classical approach settled in the linear potential theory, pointwise and fine properties of solutions can be studied by using suitable nonlinear potentials, as first shown in the pioneering work of Kilpel\"ainen \& Mal\'y \cite{KM1, KM2}. More precisely, with $\mu$ being a Borel measure, we define the (truncated) Wolff potential ${\bf W}^{\mu}_{\beta, p}$ of the measure $\mu$ as 
\eqn{wolff} 
$$
{\bf W}^{\mu}_{\beta, p}(x_0,r):= \int_0^r \left(\frac{|\mu|(B_\varrho(x_0))}{\varrho^{n-\beta p}}\right)^{1/(p-1)}\, \frac{d\varrho}{\varrho}\,, \qquad  \beta >0
$$
whenever $x \in \er^n$ and $0< r \leq \infty$. Summarizing the various contributions given in \cite{KM1, KM2, TW, KK, KMb} we have that the pointwise estimate
\eqn{pot1}
$$
|u(x_0)| \leq  c{\bf W}_{1,p}^\mu(x_0,r)+ c\left(\mean{B_r(x_0)}|u|^{p-1}\, d x\right)^{1/(p-1)}
$$
holds whenever $B_r(x_0)\subset \Omega$ and the right hand side is finite; the constant $c$ depends only on $n, p,\ratio$. Moreover, when the right hand side measure and the solution itself is nonnegative, the following lower bound holds too:
$$
c^{-1}{\bf W}_{1,p}^\mu(x_0,r) \leq u(x_0) \,.
$$   
Wolff potentials, originally introduced and studied by Havin \& Maz'ya in \cite{MH}, naturally come into the play when studying the local pointwise behaviour of solutions and play in nonlinear potential theory the same fundamental role of that Riesz potentials play in the classical linear one. Indeed, it can be proved that if
 $$
 \lim_{r\to 0}{\bf W}_{1,p}^{\mu}(x, r)=0 \ \ \mbox{locally uniformly in $\Omega$ w.r.t.~$x$}\,,
 $$
 then $u$ is continuous in $\Omega$. Moreover, any point $x$ for which the Wolff potential ${\bf W}_{1,p}^\mu(x,R)$ is finite for some $R >0$ is a Lebesgue point of $u$. The validity of Wiener criterion for non-linear equations relies on the local behaviour of suitable Wolff potentials, as shown in \cite{M, KM2} while estimate \rif{pot1} allows to deduce important existence and solvability theorems for non-homogeneous equations, as shown in \cite{JV, PV1, PV2}.  For more on potential estimates and more results we refer to \cite{KMl, KMb}.

\subsection{The basic existence theorem and SOLA}\label{SOLAsec} As mentioned above, in order to formulate the existence result for the Dirichlet problem \rif{eq} we first need to specify the suitable notion of solution, and to do it in a way which is such that basic properties of solutions in the classical local case are preserved. 
SOLA are therefore defined following the approximation scheme 
settled in the local case, with an additional approximation for the boundary values, which is here allowed to be different than zero. Before giving the definition we need to specify a few nonlocal objects that will be crucial in the subsequent analysis. Since problems of the type \rif{eq} are defined on the whole $\er^n$, the analysis of solutions necessarily involve a quantification of the long-range interactions of the function $u$. A suitable quantity to control the interaction is the following {\em Tail}, which is initially defined whenever $v \in L^{p-1}_{\loc}(\ern)$:
\begin{equation}\label{tail}
{\rm Tail}(v; x_0, r):= \left[ r^{sp} \int_{\ern\setminus B_r(x_0)}\frac{ |v(x)|^{p-1}}{ |x{-}x_0|^{n+sp}}\,d x \right]^{1/(p-1)}\,.
\end{equation} 
This quantity is already known to play an important problems in the regularity theory of energy solutions \cite{DKP1, DKP2}. 
In the following, when the point $x_0$ will be clear from the context, we shall also denote 
${\rm Tail}(v; r)\equiv {\rm Tail}(v; x_0, r)$. 
We accordingly define 
\begin{equation} \label{tail L}
L^{p-1}_{sp}(\ern) := \big\{ v \in L_{\rm loc}^{p-1}(\ern) \; : \;   {\rm Tail}(v;z,r)< \infty \; \;  \forall \, z \in \ern\,, \forall\, r \in (0,\infty)\big\}\,.
\end{equation}
It is easy to see that $W^{s,p} (\ern)\subset L^{p-1}_{sp}(\ern)$. 
Both the Tail and the space defined in \rif{tail L} play an important role when considering both classical (energy) distributional solutions 
to \rif{base} and those that we eventually define as SOLA. The definition is then as follows: 
\begin{definition}\label{weak}
Let $\mu \in (W^{s,p}(\Omega))'$ and $g \in W^{s,p} (\ern)$.  A weak (energy) solution to the problem
\begin{equation} \label{eqvw}
\left\{
\begin{array}{rl}
 - \mathcal{L}_{\Phi} u =  \mu & \mbox{in } \; \Omega    \\
  u =  g &  \mbox{in } \, \ern \setminus \Omega 
\end{array}
\right.
\end{equation}
is a function 
$u\in W^{s,p}(\ern)$ such that 
$$
\int_{\ern} \int_{\ern} \Phi(u(x){-}u(y))  (\varphi(x) - \varphi(y)) K(x,y)\, dx \, dy = 
\langle\mu,  \varphi\rangle
$$
holds for any $\varphi \in C^\infty_0(\Omega)$ and such that $u = g$ a.e. in $\ern \setminus \Omega$. Accordingly, we say that $u$ is a weak subsolution (supersolution) to \trif{eqvw} if and only if 
\eqn{subsuper}
$$\int_{\ern} \int_{\ern} \Phi(u(x){-}u(y))  (\varphi(x) - \varphi(y)) K(x,y)\, dx \, dy \leq (\geq) 
\langle \mu,  \varphi\rangle $$
holds for every non-negative $\varphi \in C^{\infty}_0(\Omega)$.
\end{definition}
Let us observe that the space $L^{p-1}_{sp}(\ern)$ is an essential tool in order to give the definition above. In the following we shall very often refer to a weak solution (sub/supersolution) to \rif{eqvw} saying that is a weak solutions to $-\mathcal{L}_{\Phi} u = \mu$ in a domain $\Omega$, thereby omitting to specify the boundary value $g$; this will be clear from the context or will not intervene in the estimates. 
\begin{remark} \label{costanti} With $k \in \er \setminus \{0\}$ being a constant,  
Definition \ref{weak} is such that $u+k$ is not a weak solution anymore, essentially because constants do not belong to $W^{s,p}(\ern)$ (otherwise \rif{subsuper} is still satisfied by $u+k$). On the other hand, what we really need in all regularity estimates (that is in all the proofs but those for Theorem \ref{existence} below) is that \rif{subsuper} is satisfied, that $u,g \in W^{s,p}_{\loc}(\er^n)$ and that $g \in L^{p-1}_{sp}(\ern)$. At this point we note that $g \in L^{p-1}_{sp}(\ern)$ implies that $g+k \in L^{p-1}_{sp}(\ern)$ and therefore given a weak solution $u$, we can apply all the local estimates valid for $u$ to $u +k$, too. This will be done several times in the rest of the paper. 
\end{remark}
We are now ready for the following:
\begin{definition}[{\rm SOLA for the Dirichlet problem}]\label{soladef}
Let $\mu \in \mathcal M(\ern)$, $g \in  W^{s,p}_{\loc} (\ern) \cap L^{p-1}_{sp}(\ern)$ and let $-\lp$ be 
 defined in \trif{duality} under assumptions \trif{monophi}-\trif{basicbounds}.
We say that  a function $u\in W^{h,q}(\Omega)$ for 
\eqn{condizioni}
$$ h \in (0,s)\,, \qquad \max\{1,p-1\} =:  q_* \leq  q < \bar q:=  \min \left\{ \frac{n(p-1)}{n-s} \,, p \right\}\;,$$
is a SOLA to \trif{eq} 
if it is a distributional solution to $ - \mathcal{L}_{\Phi} u =  \mu$ in $\Omega$, that is 
\begin{equation}\label{veryweak}
\int_{\ern} \int_{\ern} \Phi(u(x){-}u(y))  (\varphi(x) - \varphi(y)) K(x,y)\, dx \, dy = \int_{\ern} \varphi \, d\mu
\end{equation}
holds whenever $\varphi \in C^\infty_0(\Omega)$, if $u = g$ a.e. in $\ern \setminus \Omega$. Moreover it has to satisfy the following approximation property: There exists a sequence of functions $\{u_j\} \subset W^{s,p}(\ern)$ weakly solving the approximate Dirichlet problems 
\eqn{approssimazione}
$$
\left\{
\begin{array}{rl}
  -\mathcal{L}_{\Phi} u_j = \mu_j & \mbox{in } \; \Omega    \\
  u_j =  g_j &  \mbox{on } \, \ern \setminus \Omega \,,
\end{array}
\right.
$$
in the sense of Definition \ref{weak}, such that 
$u_j$ converges to $u$ a.e. in $\ern$ and locally in $L^q(\ern)$. Here the sequence $\{\mu_j\} \subset C_0^\infty(\ern)$ converges to $\mu$ weakly in the sense of measures in $\Omega$ and moreover satisfies
\begin{equation} \label{eq:meas conv cond}
\limsup_{j \to \infty} |\mu_j|(B) \leq |\mu|(\overline B)
\end{equation}
whenever $B$ is a ball. The sequence $ \{g_j\} \subset C_0^\infty(\ern)$ converges to $g$ in the following sense: For all balls $B_r \equiv B_r(z)$ with center 
in $z$ and radius $r>0$, it holds that
\begin{equation} \label{eq:g_j conv}
g_j \to g \quad \mbox{in } \;  W^{s,p}(B_r)\,,\qquad  \mbox{and} \qquad \lim_{j}\, {\rm Tail}(g_j-g; z, r)=0\,.
\end{equation}
\end{definition}
Condition~\eqref{eq:meas conv cond} can be easily seen to be satisfied if, for example, the sequence $\{\mu_j\}$ is obtained via convolutions with a family of standard mollifiers; as a matter of fact, this is a canonical way to construct the approximating sequence $\{\mu_j\}$ when showing the existence of SOLA. 

A SOLA to \trif{eq} always exists, as stated in the next theorem.
\begin{theorem}[Solvability]\label{existence}
Let $\mu \in \mathcal M(\ern)$, $g \in  W^{s,p}_{\loc} (\ern) \cap L^{p-1}_{sp}(\ern)$ and let $-\lp$ be 
defined in \trif{duality} under assumptions \trif{monophi}-\trif{basicbounds}. Then there exists a {\rm SOLA} $u$ to \trif{eq} in the sense of Definition \ref{soladef}, such that $u \in W^{h,q}(\Omega)$ for every $h$ and $q$ as described in \trif{condizioni}. 
\end{theorem}
\subsection{Potential bounds and fine properties of solutions}\label{pot sec} Our first main result concerning nonlinear potential estimates is the following one, involving the Wolff potentials defined in \rif{wolff}:\begin{theorem}[Potential upper bound]\label{pot0}
Let $\mu \in \mathcal M(\ern)$, $g \in  W^{s,p}_{\loc} (\ern) \cap L^{p-1}_{sp}(\ern)$ and let $-\lp$ be 
defined in \trif{duality} under assumptions \trif{monophi}-\trif{basicbounds}. Let $u$ be a {\rm SOLA} to 
\trif{eq} and assume that for a ball $B_r(x_0) \subset \Omega$ the Wolff potential ${\bf W}_{s,p}^\mu(x_0,r)$ is finite. Then $x_0$ is a Lebesgue point of $u$ in the sense that there exists the precise representative of $u$ at $x_0$
\eqn{precise}
$$u(x_0) := \lim_{\varrho \to 0} (u)_{B_\varrho(x_0)}=\lim_{\varrho \to 0} \mean{B_\varrho(x_0)}u\, d x$$
 and the following estimate holds for a constant $c$ depending only on 
$n,s,p,\Lambda$:
\eqn{stimawolff1}
$$
|u(x_0)|  \leq c {\bf  W}_{s,p}^\mu(x_0,r) + c\left(\mean{B_r(x_0)} |u|^{q_*} \, dx \right)^{1/q_*} +c{\rm Tail}(u ; x_0, r)\,,
$$
where $q_* := \max\{1,p-1\}$.  
\end{theorem}
Estimate \rif{stimawolff1} formally gives the one in \rif{pot1} when setting $s=1$, apart from the additional tail term 
${\rm Tail}(u ; x_0, r)$ appearing on the right-hand side, whose role is to encode 
the long-range interactions appearing when dealing with nonlocal problems. The proof is anyway considerably 
different from the ones offered in \cite{KM2, TW}, and is based on a different comparison scheme, which is particularly delicate in the nonlocal case. 
Theorem \ref{pot0} is sharp in describing the pointwise behaviour of the SOLA in the sense that the Wolff potential appearing on the right hand side of \rif{stimawolff1} cannot be replaced by any other potential (as it is on the other hand easy to guess by 
basic dimension analysis). Indeed, if the measure $\mu$ is nonnegative, then we also have the following potential lower bound, that in turn implies an optimal description, in terms of Lebesgue points :
\begin{theorem}[Potential lower bound and fine properties]\label{thm:lower}
Let $\mu \in \mathcal M(\ern)$ be a nonnegative measure, $g \in  W^{s,p}_{\loc} (\ern) \cap L^{p-1}_{sp}(\ern)$ and let $-\lp$ be 
 defined in \trif{duality} under assumptions \trif{monophi}-\trif{basicbounds} with $p< n/s$. Let $u$ be a 
{\rm SOLA} to \trif{eq} which is non-negative in the ball $B_r(x_0)\subset \Omega$ and 
such that the approximating sequence $\{\mu_j\}$ for $\mu$ as described in Definition \ref{soladef} is made of nonnegative functions. Then the estimate 
\eqn{stimawolff2}
$$
{\bf W}_{s,p}^\mu(x_0,r/8) \leq c u(x_0) + c {\rm Tail}(u_- ; x_0, r/2)
$$
holds for a constant $c\equiv c(n,s,p,\Lambda)$, as soon as ${\bf W}_{s,p}^\mu(x_0,r/8)$ is finite, where 
$u_-:=\max\{-u,0\}$. In this case, according to Theorem \ref{pot0}, $u(x_0)$ is defined as the precise representative of $u$ at $x_0$ as in \trif{precise}. Moreover, when ${\bf W}_{s,p}^\mu(x_0,r/8)$ is infinite, we have that 
\eqn{essinf} 
$$
\lim_{t \to 0} \,(u)_{B_t(x_0)} = \infty\,. 
$$
Consequently, every point is a Lebesgue point for $u$. 
\end{theorem} 
Pointwise estimates via potentials imply local Calder\'on-Zygmund type estimates; this is a consequence of the fact that the behaviour of Wolff potentials in rearrangement invariant function spaces, and in particular in Lebesgue spaces, is  
known \cite{C}. As an example of applications, we confine ourselves to report the following: 
\begin{corollary}[Calder\'on-Zygmund estimates]\label{cztheory} Let $\mu \in \mathcal M(\ern)$, $g \in  W^{s,p}_{\loc} (\ern) \cap L^{p-1}_{sp}(\ern)$ and let $-\lp$ be 
 defined in \trif{duality} under assumptions \trif{monophi}-\trif{basicbounds}. Let $u$ be a {\rm SOLA} to 
\trif{eq}. Then
\begin{itemize}
\item If $sp < n$, then $u$ belongs to the following Marcinkiwiecz space:
$$
 u \in \mathcal M^{\frac{n(p-1)}{n-sp}}_{\loc}(\Omega) \Longleftrightarrow \sup_{\lambda \geq 0} \lambda^{\frac{n(p-1)}{n-sp}} |\{x \in \Omega'\, : \, |u(x)|> \lambda \}| < \infty
$$
for every open subset $\Omega' \Subset \Omega$. 
\item If $spq<n$ and $q>1$, then
$$
\mu \in L^q_{\loc}(\Omega) \Longrightarrow u \in L^{\frac{nq(p-1)}{n-spq}}_{\loc}(\Omega)\;.
$$
\end{itemize}

\end{corollary}
The first result appearing in Corollary \ref{cztheory} is the sharp fractional counterpart of the Marcinkiewicz estimate in \cite{elenco}. Theorem \ref{pot0} is actually itself a corollary of a more general result that we report below, and that in a sense quantifies the oscillations of the gradient averages around the considered point. For this we need to introduce another quantity that we shall extensively use throughout the paper. This is the following {\em global excess} functional, which is defined for functions $f \in L_{\rm loc}^{q_*}(\ern) \cap L_{sp}^{p-1}(\ern)$:
\eqn{eccesso} 
$$
E(f;x_0,r) := \left(\mean{B_r(x_0)} |f{-}(f)_{B_r(x_0)}|^{q_*} \, dx \right)^{1/q_*} + {\rm Tail}(f{-}(f)_{B_r(x_0)};x_0,r)\,,
$$
where, as above, $q_* = \max\{1,p-1\}$. When the role of the point $x_0$ will be clear from the context we shall often denote $E(f;r)\equiv E(f;x_0,r) $. The global excess functional is the right tool to quantify, in an integral way, the oscillations of functions and here we show that it plays for nonlocal problems the same role that the traditional local excess functional defined by
$$
A(f;r)\equiv A(f; x_0,r) :=  \left(\mean{B_r(x_0)} |f - (f)_{B_r(x_0)}|^{q_*} \, dx \right)^{1/q_*} $$ 
plays for local ones; see Section \ref{intermediatesec} below. A careful study of the structure of the Tail leads to establish basic properties of the global excess making it a good replacement of the local one; 
these are established in Lemma \ref{lemma:basic E} below. Then Theorem \ref{pot0} follows from a stronger regularity/decay property of the global excess:
\begin{theorem}[Global excess decay]\label{pot}
Under the assumptions of Theorem \ref{pot0} there exists a constant $c \equiv c(n,s,p,\Lambda)$ such that 
the following estimate: 
\eqn{estipot0}
$$
 \int_{0}^r E(u;x_0,t) \, \frac{dt}{t}+|(u)_{B_r(x_0)}-u(x_0)|  \leq c{\bf W}_{s,p}^\mu(x_0,r) + c E(u;x_0,r)\,,
$$
holds whenever ${\bf W}_{s,p}^\mu(x_0,r)$ is finite. 
\end{theorem}
Estimate \rif{estipot0} tells that finiteness of Wolff potentials at a point $x_0$ allows for a pointwise control 
on the oscillations of the solution averages and eventually implies that $x_0$ is a Lebesgue point for $u$. By reinforcing this condition in a uniform decay to zero we obtain the continuity of $u$. Indeed,  the following  holds:
\begin{theorem}[Continuity criterion]\label{thm:cont}
Let $\mu \in \mathcal M(\ern)$, $g \in  W^{s,p}_{\loc} (\ern) \cap L^{p-1}_{sp}(\ern)$ and let $-\lp$ be 
 defined in \trif{duality} under assumptions \trif{monophi}-\trif{basicbounds}. Let $u$ be a {\rm SOLA} to 
\trif{eq} and $\Omega '\Subset \Omega$ be an open subset. If 
\eqn{decay}
$$
\lim_{t \to 0}\sup_{x \in \Omega'} {\bf W}_{s,p}^\mu(x,t) = 0\,,
$$
then $u$ is continuous in $\Omega'$. 
\end{theorem}
The previous general theorem allows to deduce a few corollaries asserting the continuity of solutions in borderline cases and when certain density conditions are satisfied by the measure $\mu$; these are listed below. 
\begin{corollary}[Lorentz regularity]\label{lorentz} Let $\mu \in \mathcal M(\ern)$, $g \in  W^{s,p}_{\loc} (\ern) \cap L^{p-1}_{sp}(\ern)$ and let $-\lp$ be 
 defined in \trif{duality} under assumptions \trif{monophi}-\trif{basicbounds} with $sp < n$. Let $u$ be a {\rm SOLA} to 
\trif{eq} and assume also that $\mu$ satisfies the Lorentz space condition 
$$\mu \in L(n/(sp), 1/(p-1))\;,$$
locally in $\Omega$. 
Then $u$ is continuous in $\Omega$.
\end{corollary}
Corollary \ref{lorentz} is immediate via a standard linkage between Lorentz spaces and Wolff potentials; we refer for instance to \cite{KMb} for more details and to \cite{steinweiss} for basic definitions and properties of Lorentz spaces. The next corollary is instead completely immediate and provides a sharp and borderline extension to nonlocal operators of classical local density results due to Kilpel\"ainen \cite{kilp} and Liebermann \cite{liebe}.
\begin{corollary}[Measure density conditions]\label{densitym}Let $\mu \in \mathcal M(\ern)$ be a measure, $g \in  W^{s,p}_{\loc} (\ern) \cap L^{p-1}_{sp}(\ern)$ and let $-\lp$ be 
 defined in \trif{duality} under assumptions \trif{monophi}-\trif{basicbounds} with $sp < n$. Let $u$ be a {\rm SOLA} to 
\trif{eq} and assume that the following density condition is satisfied for every ball $B_r\subset \ern$ and for a function $h \colon [0, \infty]\to [0, \infty]$:
$$
|\mu| (B_r) \leq h(r)r^{n-sp} \qquad \mbox{where} \qquad \int_0 [h(r)]^{1/(p-1)}\, \frac{dr}{r}< \infty\;.
$$
Then $u$ is continuous in $\Omega$.
\end{corollary}

\subsection{Nonlocal Campanato's theory}\label{intermediatesec} The role of the global excess functional introduced in \rif{eccesso} is to encode the information about the oscillations of solutions, taking simultaneously into account both the behaviour in small balls and the long-range 
interactions. 
Indeed, via this functional, we are able to prove a neat integral decay estimate that incorporates the regularity 
of solutions to homogeneous Dirichlet problems of the type
\eqn{homeqn}
$$
\left\{
\begin{array}{rl}
  -\mathcal{L}_{\Phi} v = 0 & \mbox{in } \; \Omega    \\
  v =  g &  \mbox{on } \, \ern \setminus \Omega \,,
\end{array}
\right.
$$
for some open subset $\Omega \subset \er^n$, 
where $v\in W^{s,p}(\er^n)$ is a distributional solution in the sense of Definition \ref{weak}, for some boundary value $g\in W^{s,p} (\ern)$. The following decay excess result plays in this setting a role similar to the one that in the local  setting is played by the classical estimates
of Campanato \cite{camp}:
\begin{theorem}[Global excess decay] \label{lemma:osc red}
With $p \in (1,\infty)$, $s \in (0,1)$ let $v\in W^{s,p}(\ern)$ be a weak solution of \trif{homeqn} under assumptions \trif{monophi}-\trif{thekernel} in the sense of Definition \ref{weak}. Let $B_R(x_0)\subset \Omega$ and let the global excess functional $E(\cdot)$ be defined as 
in \trif{eccesso} with $q_* = \max\{1,p-1\}$. Then there exist positive constants $\alpha \in (0,  sp/q_*)$, and $c$, both depending only on $n,s,p,\Lambda$, such that the following inequality holds whenever $0<\varrho \leq r \leq R$:
\begin{eqnarray} \label{eq:decay compi}
\nonumber
E(v;x_0,\varrho) & \leq & 
  c  \left(\frac{\varrho}{r}\right)^\alpha \left( \frac{r}{R} \right)^{sp/q_*} E(v;x_0,R) \\  &  &\qquad  +c \left(\frac{\varrho}{r}\right)^\alpha \int_r^{R}  \left(  \frac{r}{t}\right)^{sp/q_*} E(v;x_0,t)  \, \frac{dt}{t} 
\,.
\end{eqnarray}
\end{theorem}

\subsection{Reductions, solvability}\label{reduction} 
In several of the following proofs, we can reduce our analysis to the case in which
\eqn{specialp}
$$
\Phi(t)\equiv \Phi_p(t):= |t|^{p-2}t \qquad \mbox{and}\qquad K(x,y)= K(y,x)\,, \qquad \forall \ x,y \in \er^n\,.
$$
Let us explain how to make these reductions. As for the first reduction in \rif{specialp}, 
we define the new kernel $ \overline{K}_{u,K}(x,y)$ by
$$
\overline{K}_{u,\Phi,K}(x,y)  := 
\left\{
\begin{array}{rl}
\displaystyle\frac{\Phi(u(x){-}u(y)) K(x,y) }{|u(x){-}u(y)|^{p-2}(u(x)-u(y))} &  \mbox{if}\  x \not= y 	\ \mbox{and}\  u(x)\not =u(y)\\ [10 pt]
 |x-y|^{-n-ps} &  \mbox{if}\  x \not= y 	\ \mbox{and}\  u(x) =u(y)  \,.
\end{array}
\right.
$$
We also set
\eqn{newkernelu}
$$
\widetilde{K}(x,y) \equiv \widetilde{K}_{u,\Phi,K}(x,y) := \frac12 \left(\overline{K}_{u,\Phi,K}(x,y) + \overline{K}_{u,\Phi,K}(y,x) \right)\,.
$$
Clearly $\widetilde{K}$ is symmetric and by~\rif{thekernel} and~\rif{monophi} it satisfies the growth and ellipticity bounds $ \Lambda^{-2} \leq \widetilde{K}(x,y) |x{-}y|^{n+ps} \leq \Lambda^2$ for $x \not = y$. 
We then observe  that we may rewrite \rif{duality} as
\[
\langle -\mathcal{L}_{\Phi} u, \varphi\rangle  := \int_{\ern}\int_{\ern}   
|u(x){-}u(y)|^{p-2}(u(x){-}u(y)) (\varphi(x)-\varphi(y)) \overline{K}_{u,\Phi,K}(x,y) \, dx \, dy 
\]
for all $\varphi \in C_0^{\infty}(\Omega)$.
Renaming variables we also deduce that 
$$
 \int_{\ern}\int_{\ern} |u(x){-}u(y)|^{p-2}(u(x){-}u(y)) (\varphi(x)-\varphi(y)) \widetilde{K}(x,y) \, dx \, dy = \int_{\ern} \varphi \, d\mu\,,
$$
 and this tells that we may reduce to \rif{specialp} once we already have a solution $u$ to deal with. The only difference is that the new Kernel depends on the solution itself, but this is not a problem, since the only thing we are going to use about the kernels considered in this paper is that they satisfy the growth and ellipticity conditions in \rif{thekernel}. We will use the following abbreviated notation (see \rif{specialp}) 
  \begin{eqnarray}
\nonumber\langle - \widetilde{\mathcal L}_u w, \varphi\rangle &:=  &
\int_{\ern}\int_{\ern}   |w(x){-}w(y)|^{p-2}(w(x){-}w(y)) (\varphi(x)-\varphi(y)) \widetilde{K}(x,y) \, dx \, dy \\
&:=  &
\int_{\ern}\int_{\ern}   \Phi_p(w(x){-}w(y)) (\varphi(x)-\varphi(y)) \widetilde{K}_{u,\Phi,K}(x,y) \, dx \, dy
\,.\label{abbreviated}
\end{eqnarray}
The advantage in using the additional condition \rif{specialp} is for instance that we can then use the estimates developed in~\cite{DKP1,DKP2}.
\begin{remark}[Solvability]\label{solvability} Dirichlet problems as 
$$
\left\{
\begin{array}{rl}
  -\mathcal{L}_{\Phi} v = \mu \in C^{\infty}_0(\Omega) & \mbox{in } \; \Omega    \\
  v =  g \in C^{\infty}_0(\er^n)&  \mbox{on } \, \ern \setminus \Omega \,,
\end{array}
\right.
$$
with $\Phi(\cdot)$ as in \rif{monophi}, 
can be solved by using the classical Direct Methods of the Calculus of Variations in fractional Sobolev spaces. Indeed, denoting $$\mathcal{P} \Phi(t): = \int_0^{|t|} \Phi(\tau) \, d\tau$$ 
the primitive of $\Phi(\cdot)$, we have that by \rif{monophi} a solution $v \in W^{s,p}(\ern)$ can be obtained by solving the minimization problem
$$
v \to \min_{w \in W_g^{s,p}(\ern)} \,  \int_{\ern}\int_{\ern} \mathcal{P} \Phi(w(x)-w(y))K(x,y)\, dx\, dy - \int_{\er^n}  w\mu \, dx 
$$
in the generalised Dirichlet class 
$ W_g^{s,p}(\ern)= \left\{w \in W^{s,p}(\ern) \, :\, w\equiv g \ \mbox{in}\ \ern\setminus \Omega \right\}$. Notice that \rif{monophi} allows to verify the coercivity of the functional in the above display. For more details on these facts we refer to \cite{DKP1}.  
\end{remark}

\subsection{General notation}\label{notationsec} In what follows we denote by $c$ a general positive constant, possibly varying from line to line; special occurrences will be denoted by $c_1, c_2, \bar c_1, \bar c_2$ or the like. All such constants will always be {\em larger or equal than one}; moreover relevant
dependencies on parameters will be emphasized using 
parentheses, i.e., ~$c_{1}\equiv c_1(n,p,s,\Lambda)$ means that $c_1$ depends only on 
$n,p,s,\Lambda$. We denote by $B_r(x_0):=\{x \in \er^n \, : \,  |x-x_0|< r\}$ the open ball with center $x_0$ and radius $r>0$; when not important, or clear from the context, we shall omit denoting the center as follows: $B_r \equiv B_r(x_0)$; moreover, with $B$ being a generic ball with radius $r$ we will denote by $\sigma B$ the ball concentric to $B$ having radius $\sigma r$, $\sigma>0$. Unless otherwise stated, different balls in the same context will have the same center. 
With $\mathcal O \subset \er^{k}$ being a measurable set with respect to a measure w, and with $h$ being a measurable map, we shall denote by $$
   (h)_{\mathcal O} \equiv \mean{\mathcal O}  h \, d{\rm w}  := \frac{1}{{\rm w}(\mathcal O)}\int_{\mathcal O}  h \, d{\rm w}
$$
its integral average. In the case w is the usual Lebesgue measure we denote, in a standard way, w$(\mathcal O)=|\mathcal O|$. As usual, we abbreviate $\essinf \equiv \inf$ and $\esssup \equiv \sup$. 

\section{The global excess and proof of Theorem \ref{lemma:osc red}}\label{eccessosec}
In this section we prove a few fundamental estimates for weak solutions to homogeneous Dirichlet 
problems of the type \rif{homeqn} (for some boundary value $g$) and some basic properties of the global excess functional in \rif{eccesso}.
With these tools available, we then prove Theorem \ref{lemma:osc red}. 
Following the content of Section \ref{reduction}, we shall always assume that the \rif{specialp}, and this will allows us to use the results from \cite{DKP1, DKP2}. The reduction to the case when the identities in \rif{specialp} are in force can be done re-writing  
$ - \mathcal{L}_{\Phi}= -\widetilde{\mathcal L}_v$
as described in \rif{abbreviated}. In particular, \rif{homeqn}$_1$ will be re-written as $\widetilde{\mathcal L}_vv =0$.

\subsection{Preliminary estimates} We start with a few basic results taken from \cite{DKP1}.
\begin{theorem}[Caccioppoli estimate \cite{DKP1}]\label{thm:cacc} Let $p \in (1,\infty)$, $s \in (0,1)$ and
let $v\in W^{s,p}(\ern)$ be a weak solution of \trif{homeqn} under assumptions \trif{monophi}-\trif{thekernel}. Then, for any $B_r\subset \Omega$ and any $\varphi\in C^\infty_0(B_r)$, the following estimate holds:
\begin{eqnarray*}
\nonumber && \int_{B_r}\int_{B_r} \frac{|v(x)\varphi(x){-}v(y)\varphi(y)|^p}{|x{-}y|^{n+sp}} \,dx\, dy\\
 && \nonumber \qquad  \leq c\int_{B_r}\int_{B_r} 
 (\max\{|v(x)|,|v(y)|\})^p \frac{|\varphi(x){-}\varphi(y)|^p}{|x{-}y|^{n+sp}}\, dx\, dy\\
&&\qquad  \quad+\,c \,\int_{B_r}|v(x)||\varphi(x)|^p\,dx \left(\sup_{y\,\in\, {\rm supp}\,\varphi}\int_{\ern\setminus B_r} 
\frac{|v(x)|^{p-1}}{ |x{-}y|^{n+sp}}\,dx \right)\!, 
\end{eqnarray*}
where $c$ depends only on~$p,\Lambda$.
\end{theorem}
\begin{theorem}[Local boundedness \cite{DKP1}]\label{thm:bnd}
Let $p \in (1,\infty)$, $s \in (0,1)$ and let $v\in W^{s,p}(\ern)$ be a weak solution of \trif{homeqn} under assumptions 
\trif{monophi}-\trif{thekernel}. Let $B_r(x_0)\subset \Omega$; then the following estimate holds:
\begin{eqnarray}\label{sup_estimate}
\sup_{B_{r/2}(x_0)} |v| \leq c  \left(\mean{B_r(x_0)} |v|^p\, dx\right)^{1/p}+{\rm Tail}(v;x_0,r/2)\,,
\end{eqnarray}
where ${\rm Tail}(v ;x_0,r/2)$ is defined in~\eqref{tail} and the constant $c$ depends only on $n,s,\Lambda$.
\end{theorem} 
The last result we are reporting is concerned with De Giorgi's theory about H\"older continuity of solutions to homogeneous fractional equations. 
\begin{theorem}[H\"older continuity \cite{DKP1}]\label{thm:holder}
Let $p \in (1,\infty)$, $s \in (0,1)$ and let $v\in W^{s,p}(\ern)$ be a weak solution of \trif{homeqn} under assumptions \trif{monophi}-\trif{thekernel}. Let $B_{2r}(x_0)\subset \Omega$; 
then there exists positive constants $\alpha \in (0,  sp/q_*)$ with $q_* = \max\{1,p-1\}$, and $c$, both depending only on $n,s,p,\Lambda$, such that 
\begin{equation*}
\osc_{B_\varrho(x_0)} v \leq c \left(\frac{\varrho}{r}\right)^\alpha \inf_{k \in \er}\left[ \mean{B_{2r}(x_0)} |v-k| \, dx+{\rm Tail}(v-k;x_0,r/2)  \right]
\end{equation*}
holds whenever $\varrho \in (0,r]$.
\end{theorem}
\begin{remark} \label{remark:k}
In the statement in~\cite{DKP1} the constant $k$ is not present, but it is easy to check that it is true since $\mathcal{L}_{\Phi} v = \mathcal{L}_{\Phi} (v-k)$ in the weak sense for all $k \in \er$. Naturally $v-k$ does not belong to $W^{s,p}(\ern)$, but nevertheless the proofs from~\cite{DKP1} go through. See also Remark \ref{costanti}. 
\end{remark}
\begin{cor} \label{cor:bnd}
Under the assumptions of Theorem~\ref{thm:bnd}, also 
$$
\sup_{B_{\sigma r}(x_0)} |v| \leq \frac{c}{(1-\sigma)^{\frac{np}{p-1}}} \left[   \mean{B_r(x_0)} |v|\, dx +{\rm Tail}(v;x_0,r/2) \right]
$$
holds whenever $\sigma \in (0,1)$, with $c\equiv c(n,s,p,\Lambda)$.
\end{cor}
\begin{proof} It is of course enough to consider the case when $\sigma\geq 1/2$. In this case, let us consider numbers $ \sigma \leq t < \gamma \leq  1$ and point $z \in B_{tr}(x_0)$. Applying Theorem \ref{thm:bnd} with the choice $B_r\equiv B_{(\gamma-t)r/100}(z)$, we gain 
\eqn{suppi1}
$$
|v(z)| \leq \frac{c}{(\gamma-t)^{n/p}}  \left(\mean{B_{sr}(x_0)} |v|^p\, dx\right)^{1/p}+{\rm Tail}(v;z,(\gamma-t)r/200)\,. 
$$
We have of course used that $B_{(\gamma-t)r/100}(z) \subset B_{\gamma r}(x_0)$ whenever $z \in B_{tr}(x_0)$. The Tail term in \rif{suppi1} can be estimated, by splitting the integration domain of the corresponding integral in the sets $B_{r/2}(x_0)\setminus B_{(\gamma-t)r/200}(z) $ 
and $\ern\setminus (B_{r/2}(x_0)\cup B_{(\gamma-t)r/200}(z))$, as follows: 
\begin{eqnarray}
\nonumber &&{\rm Tail}(v;z,(\gamma-t)r/200)\leq  \frac{c}{(\gamma-t)^{\frac{n}{p-1}}}\left(\mean{B_{r/2}(x_0)} |v|^{p-1}\, dx\right)^{1/(p-1)}\\ &&\hspace{2cm}  + \nonumber 
 \left[ \frac{r^{sp}}{(\gamma-t)^{n}} \int_{\ern\setminus (B_{r/2}(x_0)\cup B_{(\gamma-t)r/200}(z))} \frac{|v(x)|^{p-1}}{ |x{-}x_0|^{n+sp}}\,d x \right]^{1/(p-1)}\\
&& \qquad  \qquad \qquad  \quad \qquad\leq  \frac{c}{(\gamma-t)^{\frac{n}{p-1}}} \left[\left(\mean{B_{\gamma r}(x_0)} |v|^{p}\, dx\right)^{1/p}+ \tail(v;x_0,r/2)\right]\,.\label{stimatail}
\end{eqnarray}
Notice that we have used the elementary estimate 
$$\frac{|x{-}x_0|^{n+sp}}{|x-z|^{n+sp}}\leq \frac{c(n,p)}{(\gamma-t)^{n+sp}}$$ for $x \not \in B_{(\gamma-t)r/200}(z)$. 
Using \rif{suppi1} and \rif{stimatail} we get
\begin{equation} \  \nonumber
	\sup_{B_{tr}(x_0)}| v | \leq \frac{c}{(\gamma-t)^{\frac{n}{p-1}}} \left[  
	 \left(\mean{B_{\gamma r}(x_0)} |v|^p\, dx\right)^{1/p}+ {\rm Tail}(v;x_0,r/2) \right]\,,
\end{equation}
which is valid whenever $1/2 \leq t< \gamma \leq 1$. Extracting $\sup v$ from the last integral on the right hand side and appealing to Young's inequality, we arrive at
$$
\sup_{B_{ tr}(x_0)}|v|  \leq  \frac12 \sup_{B_{ \gamma r}(x_0)}|v|
+
\frac{c}{(\gamma-t)^{\frac{np}{p-1}}} \left[  \mean{B_{r}(x_0)} |v|\, dx +   {\rm Tail}(v;x_0, r/2) \right]\,.
$$
A standard iteration argument, see e.g.~\cite{giusti}, then finishes the proof.
\end{proof}
The same argument applies to subsolutions, and we report the statement for completeness. 
\begin{lemma} \label{lemma:sub bnd}
Let $p \in (1,\infty)$, $s \in (0,1)$ and let $w\in W^{s,p}(\ern)$ be a weak subsolution of \trif{homeqn} under assumptions \trif{monophi}-\trif{thekernel}. Let $B_r\equiv B_r(x_0)\subset \Omega$; then 
$$
\sup_{B_{\sigma r}(x_0)} w_+ \leq  \frac{c}{(1-\sigma)^{\frac{np}{p-1}}} \left[  \mean{B_r(x_0)} w_+ \, dx + {\rm Tail}(w_+;x_0,r/2) \right]
$$
holds whenever $\sigma \in (0,1)$, with $c\equiv c(n,s,p,\Lambda)$. 
\end{lemma}
Combining Theorem~\ref{thm:cacc} and Corollary~\ref{cor:bnd} yields the following improved Caccioppoli type estimate.
\begin{lemma} \label{lemma:cacc improved}
Under the assumptions of Theorem~\ref{thm:bnd}, for any $B_{r}\equiv B_{r}(x_0)\subset \Omega$ the following estimate:
$$
 \int_{B_{\sigma r}}\mean{B_{\sigma r}} \frac{|v(x){-}v(y)|^p}{|x{-}y|^{n+sp}} \,dx\, dy \leq \frac{c}{(1-\sigma)^\theta r^{sp}}\left[
  \mean{B_r} |v|\, dx + {\rm Tail}(v;x_0,r/2)\right]^{p}
$$
holds whenever $\sigma \in [1/2,1)$, with $c\equiv c(n,s,p,\Lambda)$ and $\theta \equiv 
\theta(n,p)$.
\end{lemma}
\begin{proof}
Let $\varphi \in C_0^\infty(B_{(1+\sigma)r/2})$ such that $0\leq \varphi \leq 1$, $\varphi \equiv  1$ on $B_{\sigma r}$ and $|D\varphi| \leq c/[r(1-\sigma)]$. Then Theorem~\ref{thm:cacc} implies, also by using Young's inequality, that
\begin{eqnarray}\label{eq:cacc2}
&&\int_{B_{\sigma r }}\mean{B_{\sigma r}} \frac{|v(x){-}v(y)|^p}{|x{-}y|^{n+sp}} \,dx\, dy
 \nonumber   \leq  \frac{c}{[(1-\sigma)r]^{sp}}\mean{B_{(1+\sigma)r/2}} |v|^p \, dx\\
&&\nonumber \qquad  \qquad \qquad+\, \frac{c}{(1-\sigma)^{n+sp}r^{sp}} \mean{B_{(1+\sigma)r/2}} |v| \,dx \left(r^{sp}\int_{\ern\setminus B_{r}} \frac{|v(x)|^{p-1}}{ |x{-}x_0|^{n+sp}}\,dx \right)\\
&&  \nonumber  \qquad  \qquad\leq  \frac{c}{(1-\sigma)^{n+p}r^{sp}}\left\{\sup_{B_{(1+\sigma)r/2}} |v|^p + 
 [{\rm Tail}(v;x_0,r/2)]^{p}\right\}\label{obvious}
\end{eqnarray}
holds. Observe that we have used the obvious inequality
$$
\frac{|x{-}x_0|}{|x{-}y|} \leq \frac{c}{(1-\sigma)}\;,
$$
which is valid for whenever $x \in \ern\setminus B_r(x_0)$ and $y \in B_{(1+\sigma)r/2}$. We continue using Corollary \ref{cor:bnd} that gives
$$
\sup_{B_{(1+\sigma)r/2}(x_0)} |v|^p \leq \frac{c}{(1-\sigma)^{\frac{np}{p-1}}} 
\left[   \mean{B_r(x_0)} |v|\, dx +{\rm Tail}(v;x_0,r/2) \right]^p\;.
$$
Connecting the last inequality with \rif{obvious} finally yields the assertion. 
\end{proof}
In the subsequent lemma we lower the exponents appearing in Lemma \ref{lemma:cacc improved}. 
\begin{lemma} \label{lemma:cacc improved 2}
Under the assumptions of Theorem~\ref{thm:bnd}, for any $B_{r}\equiv B_{r}(x_0)\subset \Omega$ the following estimate:
$$ \int_{B_{\sigma r}}\mean{B_{\sigma r}} \frac{|v(x){-}v(y)|^q}{|x{-}y|^{n+hq}} \,dx\, dy  \leq
\frac{ c}{(1-\sigma)^{\theta} r^{hq}} \left[  \mean{B_r} |v|\, dx+{\rm Tail}(v;x_0,r/2)  \right]^{q}
$$
holds whenever $q \in [1,p]$, $h \in (0,s)$, and $\sigma \in [1/2,1)$, with $c \equiv c(n,s,p,\Lambda,s-h)$ and with $\theta \equiv \theta(n,p)$. 
\end{lemma}
\begin{proof}
We can assume $1\leq q < p$, otherwise when $q=p$ the result is in Lemma \ref{lemma:cacc improved}, but see also the computations below. By H\"older's inequality and the fact that $h < s$ we get
\begin{eqnarray} \  
\nonumber &&  \int_{B_{\sigma r}}\mean{B_{\sigma r}} \frac{|v(x){-}v(y)|^q}{|x{-}y|^{n+hq}} \,dx\, dy 
=\int_{B_{\sigma r}}\mean{B_{\sigma r}} \frac{|v(x){-}v(y)|^q}{|x{-}y|^{s q}} |x{-}y|^{(s-h)q}\, \frac{dx\, dy}{|x{-}y|^n}
\\ \nonumber &&  \qquad \leq \left( \int_{B_{\sigma r}}\mean{B_{\sigma r}} \frac{|v(x){-}v(y)|^p}{|x{-}y|^{n+sp}} \,dx\, dy \right)^{q/p}
\\ \nonumber &&  \qquad \qquad \cdot \left( \int_{B_{\sigma r}}\mean{B_{\sigma r}} |x{-}y|^{ -n + (s-h)qp/(p-q)} \,dx\, dy \right)^{(p-q)/p}
\\ \nonumber &&  \qquad \leq \frac{c r^{(s-h)q}}{(s-h)^q} \left( \int_{B_{\sigma r}}\mean{B_{\sigma r}} \frac{|v(x){-}v(y)|^p}{|x{-}y|^{n+sp}} \,dx\, dy \right)^{q/p}
\end{eqnarray}
and the assertion follows by using Lemma~\ref{lemma:cacc improved}. 
\end{proof}
Next we collect a few results about basic properties of global and local excess functionals. 
\begin{lemma}[Structure properties of the global excess functional] \label{lemma:basic E}
Let $p \in (1,\infty)$, $s \in (0,1)$, $q_* := \max\{1,p-1\}$, $B_r \equiv B_r(x_0)\subset \er^n, \ r \in (0,R]$ and $f \in L_{\rm loc}^{q_*}(\ern) \cap L_{sp}^{p-1}(\ern)$. 
Then there exists a constant $c\equiv c(n,s,p)$ such that the following properties hold
\begin{eqnarray} \  
\nonumber {\rm (1)} & & E(f+g; x_0,r) \leq c E(f; x_0,r)+ c E(g; x_0,r)\,,\ \ \forall\,  r >0
\\ \nonumber {\rm (2)} & & A(f; x_0,r) \leq c A(f; x_0,t)\,, \quad E(f; x_0,r) \leq c E(f; x_0,t)\\ 
 \nonumber && \mbox{whenever }\  R/4 \leq r \leq t \leq R
\\ \nonumber {\rm (3)} & & {\rm Tail}(f-(f)_{B_{r}};x_0,r) \leq  c \left( \frac{r}{R} \right)^{sp/q_*} E(f; x_0,R) \\
\nonumber && \hspace{35mm}+ c\int_r^{R}  \left(  \frac{r}{t}\right)^{sp/q*}A(f;x_0,t)  \, \frac{dt}{t}\,, \ \mbox{for every}\  r \in (0, R]\\ \nonumber {\rm (4)} & & E(f; x_0,\sigma r)\leq c(\sigma) E(f; x_0,r)\,,\  \mbox{for every}\  r \in (0, R] \ \mbox{and}\ \sigma \in (0,1).
\end{eqnarray}
\end{lemma}

\begin{proof} In what follows all the balls considered will have $x_0$ as center. Property $(1)$ follows directly from the definitions and triangle inequality. 

As for (2), we give the proof in the case $t=R$, the general case $r \leq t \leq R$ follows just replacing $R$ by $t$ in the argument below. 
We start observing that
 \begin{eqnarray} 
\nonumber
|(f)_{B_{r}}-(f)_{B_{R}}| & \leq & \mean{B_r}| f - (f)_{B_{R}}| \, dx
\\ & \nonumber \leq & 4^{n/q} \left(\mean{B_R}| f - (f)_{B_{R}}|^{q_*} \, dx \right)^{1/q_*} \leq  c A(f;x_0,R)\label{tt1}
\end{eqnarray}
holds for $r \in [R/4,R]$ by H\"older's inequality. Moreover, a standard property of the local excess functional is that 
\eqn{tt2}
$$
A(f; x_0,r)= \left(\mean{B_r}| f - (f)_{B_r}|^{q_*} \, dx\right)^{1/q_*} \leq 2\left(\mean{B_r}| f - l|^{q_*} \, dx\right)^{1/q_*}
$$
for every $l \in \er$. Using the content of the last two displays (with $l = (f)_{B_R}$), it is then easy to obtain the first inequality from (2). For the remaining properties we shall use the splitting
\begin{eqnarray}  
\nonumber {\rm Tail}(f-(f)_{B_r};x_0, r) &   \leq &
c \left[ r^{sp} \int_{\ern \setminus B_{R}} \frac{|f(x)-l|^{p-1}}{ |x{-}x_0|^{n+sp}} \, dx \right]^{1/(p-1)} 
\\ \nonumber && 
 + c \left[ r^{sp} \int_{B_{R}\setminus B_{r} }  \frac{|f(x)-l|^{p-1}}{ |x{-}x_0|^{n+sp}}  \, dx \right]^{1/(p-1)}
\\  &=:  &T_{1}(l) + T_{2}(l)\label{split}
\end{eqnarray}
for $l \in \er$. 
The definition of Tail and \rif{tt1} now give
\begin{eqnarray*}{\rm Tail}(f-(f)_{B_{r}};x_0, r) 
& \leq &
c {\rm Tail}(f-(f)_{B_{R}};x_0, r) + c|(f)_{B_{r}}-(f)_{B_{R}}|\\
&\leq & cT_{1}((f)_{B_{R}}) + cT_{2}((f)_{B_{R}}) + cA(f;x_0,R)\\
&\leq & c{\rm Tail}(f-(f)_{B_{R}};x_0, R) + cA(f;x_0,R)\\
&= & c E(f;x_0,R)\,,
\end{eqnarray*}
completing the proof of $(2)$. 

We now begin the proof of (3), again starting by \rif{split} with $l = (f)_{B_r} $, that gives
\eqn{split again}
$$
{\rm Tail}(f-(f)_{B_r};x_0, r) \leq T_{1}((f)_{B_r})+T_{2}((f)_{B_r})\;.
$$
By $(1)$ we have
\begin{eqnarray} \  
\nonumber 
&&T_{1}((f)_{B_r}) = c  \left( \frac{r}{R} \right)^{\frac{sp}{p-1}}  {\rm Tail}(f-(f)_{B_{r}};x_0, R) 
\\ \nonumber & &   \qquad \quad\leq c \left( \frac{r}{R} \right)^{\frac{sp}{p-1}} \left[ {\rm Tail}(f-(f)_{B_{R}};x_0, R) 
+ {\rm Tail}((f)_{B_{R}}-(f)_{B_{r}};x_0, R) \right]
\\  &  &\qquad  \quad \leq c   \left( \frac{r}{R} \right)^{sp/q_*}  E(f; x_0, R) + c \left( \frac{r}{R} \right)^{sp/q_*} 
 |(f)_{B_{R}}-(f)_{B_{r}}|\,.\label{tt4}
\end{eqnarray}
In order to estimate the last quantity appearing on the right hand side of the above display, we first consider the case 
$r \in (0,R/4]$. Letting $\gamma \in (1/4,1/2]$ and $k\geq 2$, $k \in \en$, such that $r = \gamma^k R$, we have, by using also \rif{tt2} and property (2) (for obviously a different choice of $r,t,R$)
\begin{eqnarray} \nonumber 
 |(f)_{B_{r}}-(f)_{B_{\gamma R}}|& = &|(f)_{B_{\gamma^k R}}-(f)_{B_{\gamma R}}|\\  
 &\leq &   \sum_{j=1}^{k-1} | (f)_{B_{\gamma^{j}R }}- (f)_{B_{\gamma^{j+1}R }}| \nonumber \\ 
& \leq& \gamma^{-n/q_*} 
 \sum_{j=1}^{k-1} \left( \mean{ B_{\gamma^{j}R}  }|f - (f)_{B_{\gamma^{j}R}}  |^{q_*} \, dx  \right)^{1/q_*} \nonumber \\
& \leq& c  \sum_{j=1}^{k-1} \int_{\gamma^{j}R}^{\gamma^{j-1}R} 
 A(f;x_0, \gamma^{j}R)\, \frac{dt}{t} \nonumber \\
& \leq&  c  \sum_{j=1}^{k-1} \int_{\gamma^{j}R}^{\gamma^{j-1}R} 
 A(f;x_0,t)\, \frac{dt}{t} \nonumber \\
& \leq&c\int_{r}^{R} 
 A(f;x_0, t)\,  \nonumber \frac{dt}{t}\,.
\label{estlike}
\end{eqnarray} 
On the other hand, when $r \in (R/4, R]$, as in \rif{tt1} we trivially have that 
$$|(f)_{B_R}-(f)_{B_r}|  \leq  c  A(f;x_0, R)  \leq  c  E(f;x_0, R) \leq  c  \left( \frac{r}{R} \right)^{sp/q_*} E(f;x_0, R) $$ so that in any case $r \in (0, R]$ we find 
\begin{equation} \  \nonumber 
|(f)_{B_R}-(f)_{B_r}|   \leq  c   \left( \frac{r}{R} \right)^{sp/q_*}E(f;x_0, R) +  c \int_{r}^R A(f;x_0, t)\, \frac{dt}{t}\,.
\end{equation}
Hence, recalling \rif{tt4}, we can conclude with
\eqn{split2}
$$
T_{1}((f)_{B_r}) \leq  c \left( \frac{r}{R} \right)^{sp/q_*}  E(f;x_0, R)+c\int_{r}^R  \left( \frac{r}{t} \right)^{sp/q_*}A(f;x_0, t) \, \frac{dt}{t} 
$$
whenever $r \leq R$. 
Next, we estimate $T_2((f)_{B_r}) $ and we start by showing  that 
\eqn{rewritet2}
$$
T_2((f)_{B_r})\leq c \left[ r^{sp} \int_{B_{R}\setminus B_{r} }  \frac{|f(x)-(f)_{B_r}|^{q_*}}{ |y{-}x_0|^{n+sp}}  \, dx \right]^{1/q_*}
$$
holds for a constant $c \equiv c (n,s,p)$. In the case $p\geq 2$ there is nothing to prove since $q_* =p-1$. When $p< 2 $ we instead define 
$$
\lambda_{R, r}:= \int_{B_{R} \setminus B_{r}}  \frac{dx}{ |x{-}x_0|^{n+sp}} = \frac{c}{r^{sp}}-\frac{c}{R^{sp}}\,,
\qquad {\rm w} (x):=  \frac{1}{ |x{-}x_0|^{n+sp}}
$$
for $c \equiv c (n,s,p)$. 
Then we have, using Jensen's inequality
\begin{eqnarray*}
 T_{2}((f)_{B_r}) &\leq &
  \left[ r^{sp}\lambda_{R,r}\right]^{1/(p-1)} \left(\mean{B_{R} \setminus B_{r} } |f-(f)_{B_r}|^{p-1}  \, d{\rm w}(x) \right)^{1/(p-1)}
  \\& \leq &
  c\left[ r^{sp}\lambda_{R,r}\right]^{(2-p)/(p-1)} r^{sp}
  \int_{B_{R} \setminus B_{r}}|f-(f)_{B_r}| \, d{\rm w}(x)
    \\& \leq &c
 r^{sp} \int_{B_{R}\setminus B_{r}}\frac{|f(x)-(f)_{B_r}|}{ |x{-}x_0|^{n+sp}}  \, dx\end{eqnarray*}
so that \rif{rewritet2} follows also when $p<2$ as $q_*=1$ in this case. To proceed, we again first consider the case where it initially is $r \leq R/4$. 
As in the case of $ T_{1}((f)_{B_r})$, we write  $r = \gamma^kR$ and, using the concavity of the function $t \mapsto t^{1/q_*}$, we have
\begin{eqnarray} \  
\nonumber
T_2((f)_{B_r})  & \leq  &   c \left( \sum_{j=0}^{k} \gamma^{s p j} \mean{B_{\gamma^{-j} r}} |f-(f)_{B_{r}}|^{q_*} \, dx \right)^{1/q_*}
\\   & \leq &
\nonumber
c \sum_{j=0}^{k} \gamma^{s p j/q_*} \left( \mean{B_{\gamma^{-j} r}} |f-(f)_{B_{r}}|^{q_*} \, dx \right)^{1/q_*}
\end{eqnarray}
with $c \equiv c(n,s,p)$. Moreover, by triangle and H\"older's inequalities we also get
\begin{eqnarray*}  
\nonumber
&&
\left( \mean{B_{\gamma^{-j} r}} |f-(f)_{B_{r}}|^{q_*} \, dx \right)^{1/q_*}       
\\&&\qquad \qquad  \leq  c  \sum_{i=0}^{j} \left( \mean{B_{\gamma^{-i} r}} |f-(f)_{B_{\gamma^{-i} r}}|^{q_*} \, dx \right)^{1/{q_*}}
\leq   c\sum_{i=0}^{j}A(f;x_0,\gamma^{-i}r) 
\end{eqnarray*}
whenever $j \in \{0, \ldots, k\}$. 
By combining the two last displays above and applying Fubini's theorem on sums, we obtain using again the property in (3)
\begin{eqnarray*}  
T_2((f)_{B_r})  &\leq& c  \sum_{j=0}^{k} \gamma^{s p j/q_*}A(f;x_0,\gamma^{-j}r) \\
&\leq& c \sum_{j=0}^{k-1} \int_{\gamma^{-j}r}^{\gamma^{-j-1}r}\gamma^{s p j/q_*}A(f;x_0,\gamma^{-j}r) \, \frac{dt}{t}
+ c \gamma^{s p k/q_*}A(f;x_0,R)\\
&\leq& c \sum_{j=0}^{k-1} \int_{\gamma^{-j}r}^{\gamma^{-j-1}r}\left( \frac{r}{t} \right)^{sp/q_*}A(f;x_0,t) \, \frac{dt}{t}+ c\left(\frac{r}{R}\right)^{s p/q_*}A(f;x_0,R)\\
&\leq& c\int_{r}^{R}\left( \frac{r}{t} \right)^{sp/q_*}A(f;x_0,t) \, \frac{dt}{t}+ c\left(\frac{r}{R}\right)^{s p/q_*}E(f;x_0,R)\;.
\end{eqnarray*}
On the other hand, by the very definition of $T_2((f)_{B_r}) $ when $r \in (R/4, R]$, we simply have 
$$
T_2((f)_{B_r})    \leq A(f;x_0, R)\leq \left( \frac{r}{R} \right)^{sp/q_*}E(f;x_0, R)
$$
so that in any case $r \in (0,R)$ we conclude with $$
T_{2}((f)_{B_r}) \leq  c\left( \frac{r}{R} \right)^{sp/q_*}   E(f;x_0, R)+c\int_{r}^R  \left( \frac{r}{t} \right)^{sp/q_*}A(f;x_0, t) \, \frac{dt}{t} \;.
$$
Combining this last estimate with \rif{split} and \rif{split2} then leads to $(3)$. We finally prove (4), which is a straightforward consequence of the preceding properties. Indeed by using also \rif{tt2} it is easy to get that
$$
A(f; x_0,\sigma r)\leq 2\left(\mean{B_{\sigma r}}| f - (f)_{B_r}|^{q_*} \, dx\right)^{1/q_*}
\leq 2\sigma^{-n}A(f; x_0, r)\;.
$$ In a similar way, using (3) (with the choice $r \to \sigma r$ and $R \to r$) and the previous inequality, yields
$$
{\rm Tail}(f-(f)_{B_{\sigma r}};x_0,\sigma r)\leq c(\sigma) E(f; x_0,r), 
$$ so that (4) follows and the proof is complete. 
\end{proof}
In the next result we restate the oscillation decay estimate of Theorem \ref{thm:holder} in terms of the local excess.  
\begin{lemma} \label{locexcessdecay}
With the assumptions and the notation of Theorem~\ref{thm:holder}, also the inequality 
\eqn{eq:osc red pot 1}
$$
A(v;x_0,\varrho) \leq c \left(\frac{\varrho}{r}\right)^\alpha  E(v;x_0,r)
$$
holds with $c\equiv c(n,s,p,\Lambda)$ whenever $0 < \varrho \leq r \leq R  $.\end{lemma}
	\begin{proof}
	Theorem~\ref{thm:holder} implies that
\begin{eqnarray} \label{eq:osc red pot} 
\nonumber &&
A(v;x_0,\varrho) \leq  \osc_{B_{\varrho}} v  \leq  c \left(\frac{\varrho}{r}\right)^\alpha \left[  A(v;x_0,r)
+E(v;x_0,r/2)\right]
\end{eqnarray}
for all $\varrho \in (0,r]$ and with $\alpha \equiv \alpha(n,s,p,\Lambda)$, $\alpha \in (0,sp/q_*)$. Now \rif{eq:osc red pot 1} follows by using Lemma~\ref{lemma:basic E}(2) to estimate 
$E(v;x_0,r/2) \leq cE(v;x_0,r)$. \end{proof}
\subsection{Proof of Theorem \ref{lemma:osc red}}
Lemma~\ref{lemma:basic E}(3) implies that
$$
{\rm Tail}(v-(v)_{B_{\varrho}};x_0, \varrho) \leq  
c \left( \frac{\varrho}{R} \right)^{sp/q_*} E(v;x_0,R)  +c \int_\varrho^{R}  \left(  \frac{\varrho}{t}\right)^{sp/q_*} A(v;x_0,t)  \, \frac{dt}{t} 
$$
for all $\varrho \in (0,r]$. Furthermore, by~\eqref{eq:osc red pot 1}
we have 
\begin{eqnarray} \  
\nonumber
\int_\varrho^{r}  \left(  \frac{\varrho}{t}\right)^{sp/q_*} A(v;x_0,t)  \, \frac{dt}{t}  
& \leq & 
c \int_\varrho^{r}  \left(  \frac{\varrho}{t}\right)^{sp/q_*} \left(\frac{t}{r}\right)^\alpha  \, \frac{dt}{t} E(v;x_0,r)
\\ \nonumber & \leq & c \left(\frac{\varrho}{r}\right)^\alpha  E(v;x_0,r)\,,
\end{eqnarray}
since $\alpha< sp/q_*$.
Consequently
\eqn{semitri0}
$$
{\rm Tail}(v-(v)_{B_{\varrho}};x_0, \varrho) \leq  c \left( \frac{\varrho}{R} \right)^{sp/q_*} E(v;x_0,R)+ c \left(\frac{\varrho}{r}\right)^\alpha  E(v;x_0,r)   \,.
$$
We now prove the following estimate:
\eqn{semitri}
$$
E(v;x_0,r) \leq c \int_r^{R}  \left(  \frac{r}{t}\right)^{sp/q_*} E(v;x_0,t) \, \frac{dt}{t}+ c\left(\frac{r}{R}\right)^{sp/q_*}E(v;x_0,R).  
$$
For this we distinguish between two cases, the first being when $r < R/2$. Then we have, 
by Lemma~\ref{lemma:basic E}(2)
$$
E(v;x_0,r) = \frac{1}{\log 2}\int_r^{2r} E(v;x_0,r)\, \frac{dt}{t} \leq c\int_r^{R}  \left(  \frac{r}{t}\right)^{sp/q_*} E(v;x_0,t)\,.  
$$
On the other hand, when $r \in [R/2, R]$ again by Lemma~\ref{lemma:basic E}(2) we trivially have
 $$
E(v;x_0,r) \leq cE(v;x_0,R)\leq c\left(\frac{r}{R}\right)^{sp/q_*}E(v;x_0,R)
$$
so that \rif{semitri} follows. Combining \rif{semitri0} and \rif{semitri} we conclude with
\begin{eqnarray*}  
\nonumber {\rm Tail}(v-(v)_{B_{\varrho}};x_0, \varrho) & \leq & c\left( \frac{\varrho}{r} \right)^{\alpha} \left( \frac{r}{R} \right)^{sp/q_*}E(v;x_0,R)\\ && \qquad +  c \left(\frac{\varrho}{r}\right)^\alpha   \int_r^{R}  \left(  \frac{r}{t}\right)^{sp/q_*} E(v;x_0,t)  \, \frac{dt}{t}
 \,,
\end{eqnarray*}  
where we have again used that $\alpha < sp/q_*$. On the other hand we also have 
$$ A(v;x_0, \varrho) \leq   c \left( \frac{\varrho}{r} \right)^{\alpha} \left( \frac{r}{R} \right)^{sp/q_*}E(v;x_0,R)+c\left(\frac{\varrho}{r}\right)^\alpha   \int_r^{R} 
 \left(  \frac{r}{t}\right)^{sp/q_*} E(v;x_0,t)  \, \frac{dt}{t} \,,
$$
that follows by \rif{eq:osc red pot 1}  
and \rif{semitri}. The assertion of \rif{eq:decay compi} now follows combining the content of the last two displays. 
\section{Basic a priori estimates}
In this section we are going to consider operators of the type $\mathcal{L}_{\Phi}$ under the assumptions \rif{monophi}-\rif{basicbounds}. With $B_{2r} \equiv B_{2r}(x_0)\subset \er^n$ being a fixed ball, we shall consider weak solutions $u \in W^{s,p}(\ern)$ to the Dirichlet problem
\begin{equation}\label{unod}
\left \{
\begin{array}{ccc}
-\mathcal{L}_{\Phi} u= \mu  \in C_0^\infty(\ern)&\mbox{in}& B_{2r}\\
u=g \in W^{s,p}(\ern)& \mbox{in}&  \mathbb R^n \backslash B_{2r}\,.
\end{array} \right.
\end{equation}
Accordingly, following the notation introduced in Section \ref{reduction}, 
we see that $u$ solves 
\begin{equation}\label{unod2}
\left \{
\begin{array}{ccc}
-\widetilde{\mathcal{L}}_{u} u= \mu  \in C_0^\infty(\ern)&\mbox{in}& B_{2r}\\
u=g \in W^{s,p}(\ern)& \mbox{in}&  \mathbb R^n \backslash B_{2r}\;,
\end{array} \right.
\end{equation}
where the operator $-\widetilde{\mathcal{L}}_{u}$ has been built in \rif{abbreviated}, starting from a Kernel $ \widetilde{K}(\cdot)$ 
that depends on $u$ itself, as described in Section \ref{reduction}.
We then define the following weak comparison solution $v\in W^{s,p}(\ern)$ solving the following Dirichlet problem:
\begin{equation}
\left \{
\begin{array}{ccc}\label{dued}
-\widetilde{\mathcal{L}}_{u}  v=0 &\mbox{in}&\,\,B_r\\
v=u&\mbox{in}&\mathbb R^n \backslash B_r\,.
\end{array} \right. 
\end{equation}
For solvability see also Remark \ref{solvability}. Throughout Lemmas \ref{estim1}-\ref{lemma:comp p<2}, $u$ and $v$ will denote the solutions defined in \rif{unod} and \rif{dued}, respectively, while we define $w:=u-v$. 
\subsection{A basic inequality}
We premise two algebraic inequalities. The first is the following:
\eqn{eq:coer}
$$
\frac1{c_p} \leq \frac{(|\xi_1|^{p-2} \xi_1 {-}  |\xi_2|^{p-2} \xi_2)(\xi_1{-}\xi_2)}{(|\xi_1|+|\xi_2|)^{p-2} |\xi_1 {-} \xi_2|^2 }\leq c_p 
 \qquad \forall \ \xi_1 \neq \xi_2   \,.
$$
that holds whenever $\xi_1, \xi_2$ are real numbers such that $\xi_1 \neq \xi_2$, and where $c_p$ is a constant depending only on $p$. 
The second, which has a similar nature, is instead 
\eqn{eq:coer2}
$$c_p\int_0^1|t \xi_1+(1{-}t)\xi_2|^{p-2}\,dt \geq(|\xi_1|+|\xi_2|)^{p-2}\,.$$
We refer to \cite[Chapter 8]{giusti} for a proof. Now, the first crucial comparison lemma. 
\begin{lemma}\label{estim1} The following inequality holds for a constant $c\equiv c(n,p,\Lambda)$, whenever $\xi>1$ and $d>0$:
\begin{eqnarray} \label{eq:comp est 0} 
\nonumber && \int_{B_r}\int_{B_r} \frac{\big(| u(x){-}u(y)|  + |v(x){-}v(y)|\big)^{p-2} |w(x){-}w(y)|^2}{(d{+}|w(x)|{+}|w(y)|)^\xi}\frac{\,dx\,dy}{|x{-}y|^{n+sp}}
\\ && \hspace{85mm} \leq \frac{cd^{1-\xi} |\mu|(B_r) }{(p-1)(\xi-1)}\;. 
\end{eqnarray}
In particular, when $p \geq 2$ we have
\eqn{eq:comp est 0pp} 
$$
\int_{B_r}\int_{B_r} \frac{|w(x){-}w(y)|^p}{(d{+}|w(x)|{+}|w(y)|)^\xi}\frac{\,dx\,dy}{|x{-}y|^{n+sp}}\leq \frac{cd^{1-\xi} |\mu|(B_r) }{(p-1)(\xi-1)}\;.
$$
\end{lemma}
\begin{proof}
With $w:=u-v$ and $w_\pm :=\max\{\pm w,0\}$, we let
$$\varphi_\pm=\pm \left( d^{1-\xi}-\frac{1}{(d+w_\pm)^{\xi-1}}\right)\,.$$
Note, in particular that $\varphi_\pm = 0$ on $\partial B_r$; moreover this function is bounded and still belongs to $W^{s,p}(\er^n)$ since
it is obtained via composition of $w$ with a Lipschitz function. In the following we shall use the notation $A(t,s) := |t|^{p-2} t - |s|^{p-2}s$ for $t,s \in \er$.
We choose $\varphi_\pm$ as test functions in the weak formulations of \rif{unod2} and \rif{dued}; subtracting them and using the symmetry of the kernel $\widetilde{K}(\cdot)$ defined in \rif{newkernelu}, and the fact that $u\equiv v$ outside $B_r$, we obtain
\begin{eqnarray}
\nonumber I_1 + I_2 &:= &\int_{B_r}\int_{B_r} A(u(x){-}u(y),v(x){-}v(y))(\varphi_\pm(x){-}\varphi_\pm(y))\widetilde{K}(x,y)\,dx \, dy \\ && \qquad +
\nonumber 2 \int_{\er^n\setminus B_r}\int_{B_r} A(u(x){-}u(y),v(x){-}u(y))\varphi_\pm(x)\widetilde{K}(x,y)\, dx \, dy\\ 
&=&
 \int_{B_r} \varphi_\pm\,d\mu=:I_3 \,.\label{eq:comp weak} 
\end{eqnarray}
We preliminary notice that $I_2$ is non-negative. Indeed, by using \rif{eq:coer} we get
$
A(u(x){-}u(y),v(x){-}u(y))(u(x){-}v(x)) \geq 0, 
$
and the non-negativity of $I_2$ follows from the fact that the very definition of $\varphi_\pm$ implies that
$\varphi_\pm(x)(u(x){-}v(x)) \geq 0$; see also the formula in the next display. 
Since
\begin{equation} \  \nonumber
\varphi_\pm(x){-}\varphi_\pm(y) = \pm (w_\pm(x){-}w_\pm(y))(\xi-1)
\int_0^1(d+t w_\pm(y)+(1-t)w_\pm(x))^{-\xi}\,dt\,,
\end{equation}
we find by the mean value theorem a point $\gamma_\pm^*\equiv \gamma_\pm^*(x,y) \in (0,1)$ such that 
\eqn{phiphi}
$$
\varphi_\pm(x){-}\varphi_\pm(y)  = \pm (w_\pm(x){-}w_\pm(y))(\xi-1) (d+\gamma_\pm^* w_\pm(y)+(1-\gamma_\pm^*)w_\pm(x))^{-\xi}\,.
$$
We now evaluate the quantity 
$$\Psi_\pm(x,y) := \pm A(u(x){-}u(y),v(x){-}v(y))(w_\pm(x){-}w_\pm(y))\,, $$
and do it for $w_+$ and $\Psi_+$, the case for $w_-$ and $\Psi_-$ being completely similar. If $w(x), w(y) \geq 0$, then $w_+(x)=w(x)$ and $w_+(y)=w(y)$, and hence, applying \rif{eq:coer}, we have
\begin{eqnarray} \  
\nonumber
\Psi_+(x,y) & = &  A(u(x){-}u(y),v(x){-}v(y)) (u(x){-}u(y)-(v(x){-}v(y))) 
\\ \nonumber & \geq &
\frac1c \big(| u(x){-}u(y)|  + |v(x){-}v(y)| \big)^{p-2}|w(x){-}w(y)|^2\,.
\end{eqnarray}
If, on the other hand, $w(x)> 0$ and $w(y)<0$, i.e. $u(x)> v(x)$ and $u(y)<v(y)$, then, using also \rif{eq:coer2} we have 
\begin{eqnarray} \  
\nonumber
\frac{\Psi_+(x,y)}{u(x){-}v(x) }  & = &  |u(x){-}u(y)|^{p-2} (u(x){-}u(y)) {-} |v(x){-}v(y)|^{p-2} (v(x){-}v(y)) 
\\ \nonumber  \,& =&\, (p-1) \int_0^1[t (u(x){-}u(y))+(1-t)(v(x){-}v(y))]^{p-2}\,dt\\
&& \nonumber \qquad \qquad \cdot(u(x){-}u(y){-}v(x){+}v(y))
\\ \nonumber  \,& \geq &\, \frac{(p-1)}{c_p}\big(| u(x){-}u(y)|  + |v(x){-}v(y)|\big)^{p-2}(u(x){-}v(x)) \,.
\end{eqnarray}
The case $w(x) = 0$ and $w(y)< 0$ is trivial since then $\Psi_+(x,y) = 0$. 
Thus, for all $p>1$ we conclude with
\begin{eqnarray*}
\nonumber
c\Psi_+(x,y)& \geq &(p-1)(|u(x){-}u(y)| + |v(x){-}v(y)|)^{p-2}(w_+(x))^2  \\
& = &(p-1)(|u(x){-}u(y)| + |v(x){-}v(y)|)^{p-2} (w_+(x){-}w_+(y))^2 \,.
\end{eqnarray*}
The case $w(x)<0$ and $w(y)\geq 0$ is treated in a completely similar manner. Indeed, we have
\begin{eqnarray} \  
\nonumber
\frac{\Psi_+(x,y)}{v(y){-}u(y) }  & =&\, (p-1) \int_0^1[t (u(x){-}u(y))+(1-t)(v(x){-}v(y))]^{p-2}\,dt\\
&& \nonumber \qquad \qquad \cdot(u(y){-}u(x){-}v(y)+v(x))
\\ \nonumber  \,& \geq &\, \frac{(p-1)}{c_p}\big(| u(x){-}u(y)|  + |v(x){-}v(y)|\big)^{p-2}(v(y){-}u(y)) \,.
\end{eqnarray}
so that, once again we conclude with 
\begin{eqnarray*}
\nonumber
c\Psi_+(x,y)& \geq &(p-1) (|u(x){-}u(y)| + |v(x){-}v(y)|)^{p-2}(-w_+(y))^2 \\
& = &(p-1) (|u(x){-}u(y)| + |v(x){-}v(y)|)^{p-2}(w_+(x){-}w_+(y))^2 \;.
\end{eqnarray*}
In any case we finally end up with
\[
c\Psi_{\pm}(x,y) \geq (p-1)  (|u(x){-}u(y)| + |v(x){-}v(y)|)^{p-2} (w_\pm(x){-}w_\pm(y))^2
\]
with $c$ being a constant depending only on $p$. 
Plugging this last inequality into ~\eqref{eq:comp weak}, and using \rif{phiphi} together with the following elementary estimate:
\[
(d+\gamma_\pm^* w_\pm(y)+(1{-}\gamma_\pm^*)w_\pm(x))^{-\xi} \geq (d{+}|w(x)|{+}|w(y)|)^{-\xi}\,,
\]
we finally get (recalling that $I_2 \geq 0$)
\begin{eqnarray*}
 &&I_{\pm}:=\int_{B_r}\int_{B_r} \frac{\big(| u(x){-}u(y)|  + |v(x){-}v(y)|\big)^{p-2} |w_\pm(x){-}w_\pm(y)|^2}{(d{+}|w(x)|{+}|w(y)|)^\xi}\frac{\,dx\,dy}{|x{-}y|^{n+sp}}
\\ && \qquad \qquad \leq \frac{cI_1}{(p-1)(\xi-1)} = \frac{cI_3}{(p-1)(\xi-1)} \leq \frac{cd^{1-\xi} |\mu|(B_r)}{(p-1)(\xi-1)}\,,
\end{eqnarray*}
where we have used the very definition of $\varphi_\pm$ that gives $|\varphi_\pm|\leq d^{1-\xi}$. We now split as follows:
\begin{eqnarray*}
&&\int_{B_r}\int_{B_r} \frac{\big(| u(x){-}u(y)|  + |v(x){-}v(y)|\big)^{p-2} |w(x){-}w(y)|^2}{(d{+}|w(x)|{+}|w(y)|)^\xi}\frac{\,dx\,dy}{|x{-}y|^{n+sp}} \\
&& \quad = \int_{B_r\times B_r  \cap \{w(x), w(y)\geq 0 \}} (\ldots)   
\, dx \, dy  + \int_{B_r \times B_r \cap \{w(x), w(y)\leq 0 \}} (\ldots)   
\, dx \, dy \\ && \qquad + \int_{B_r \times B_r \cap \{w(x)\geq 0, w(y)< 0 \}}(\ldots)   
\, dx \, dy + \int_{B_r \times B_r \cap \{w(x)< 0, w(y)\geq 0 \}} (\ldots)   
\, dx \, dy \\ && \quad  =: II_1 +II_2 + II_3 + II_4\,.
\end{eqnarray*}
The definition of $w_{\pm}$ implies that
$$
 II_1 +II_2  \leq I_+ + I_- \leq \frac{cd^{1-\xi} |\mu|(B_r)}{(p-1)(\xi-1)}\,.
$$
We estimate the remaining two integrals, confining to $II_3$, the estimate for $II_4$ being similar. We therefore consider points $x,y \in B_r$ such that $w(x)\geq 0$ and $w(y)< 0$. In the case $w(x) \geq  |w(y)|$ we have
$
|w(x)-w(y)|\leq 2|w(x)|=2w_{+}(x)=2|w_{+}(x)-w_{+}(y)|. 
$
In the case $w(x) \leq  |w(y)|$ we instead have that 
$
|w(x)-w(y)|\leq 2|w(y)|=2w_{-}(y)=2|w_{-}(x)-w_{-}(y)|. 
$
We therefore again conclude that
$$
 II_3 +II_4  \leq 2(I_+ + I_-) \leq \frac{cd^{1-\xi} |\mu|(B_r)}{(p-1)(\xi-1)}\,.
$$
Combining the content of the last four displays yields \rif{eq:comp est 0} and the proof of the lemma is complete. 
\end{proof}
\subsection{Basic estimates in the case $p\geq 2$}
We now proceed with the proof of the main a priori estimates for solutions, thereby introducing the following functions:
\eqn{dv0}
$$U_h(x,y)=\frac{|u(x){-}u(y)|}{|x{-}y|^{h}}\,, \
V_h(x,y)=\frac{|v(x){-}v(y)|}{|x{-}y|^{h}}\,, \ W_h(x,y)=\frac{|w(x){-}w(y)|}{|x{-}y|^{h}}$$
for $x, y \in \ern, x\not =y$, and for any $h \in (0,s]$. The first a priori estimate we are going to prove is concerned with the case $p \geq 2$ and is contained in the following:
\begin{lemma} \label{lemma:comp est grad}
Assume that $p \geq 2$, $h \in (0,s)$, and $q \in [1,\bar q)$, where 
$
\bar q $ has been defined in \trif{condizioni}. 
Let $\delta := \min\{\bar q - q,s-h\}$. Then
\begin{equation} \label{eq:comp}
\left(\int_{B_r} \mean{B_r} \frac{|w(x){-}w(y)|^q}{|x{-}y|^{n+hq}} \, dx \, dy   \right)^{1/q} 
\leq \frac{c}{r^{h}} \left[ \frac{|\mu|(B_r)}{r^{n-sp}} \right]^{1/(p-1)}
\end{equation}
holds for a constant $c\equiv c(n,s,p,\Lambda,\delta)$, that blows up when $\delta \to 0$.
\end{lemma}

\begin{proof} Before starting we observe that by the same calculation made in the proof of Lemma \ref{lemma:cacc improved 2} we have that
$$
\int_{B_r} \mean{B_r} \frac{|w(x){-}w(y)|^{q_1}}{|x{-}y|^{n+h_1q_1}} \, dx \, dy \leq \frac{c r^{(h_2-h_1)q_1}}
{(h_2-h_1)^{q_1}} \left(\int_{B_r} \mean{B_r} \frac{|w(x){-}w(y)|^{q_2}}{|x{-}y|^{n+h_2q_2}} \, dx \, dy\right)^{q_1/q_2}
$$
holds for $0 < h_1 < h_2 < s$ and $1 \leq q_1 < q_2< \bar q$. Therefore we can reduce ourselves to prove \rif{eq:comp} when $h$ is arbitrarily close to $s$ and $q$ to $\bar q$, respectively. With $d>0$ and $\xi>1$ to be chosen in a few lines, we start by rewriting
\begin{eqnarray} \  
\nonumber
W_h^q(x,y)  & = & \left( \frac{W_s^p(x,y)}{(d{+}|w(x)|{+}|w(y)|)^{\xi}}  \right)^{q/p} 
\\ \nonumber && \qquad \cdot \left[ (d{+}|w(x)|{+}|w(y)|)^\xi |x{-}y|^{(s-h) p }\right]^{q/p}\,,
\end{eqnarray}
and then apply H\"older's inequality in order to get
\begin{eqnarray} \  
\nonumber
&& \frac{1}{|B_r|}\int_{\mathcal B_r} W_h^q(x,y)\,\frac{dx\, dy}{|x{-}y|^n}  \leq  
\left[ \frac{1}{|B_r|} \int_{\mathcal B_r} \frac{W^p_s(x,y)}{(d{+}|w(x)|{+}|w(y)|)^{\xi}}  \frac{dx\, dy}{|x{-}y|^n} \right]^{q/p}
\\ \nonumber & & \qquad \qquad \qquad  \qquad\cdot  \left[\frac{1}{|B_r|} \int_{\mathcal B_r} \frac{(d{+}|w(x)|{+}|w(y)|)^{\xi q/(p-q)}}{|x{-}y|^{n-(s-h)qp/(p-q) }} \,dx \, dy \right]^{(p-q)/p} \,.
\end{eqnarray}
Notice that we have used that $q <	 p$. 
Using \rif{eq:comp est 0pp} and Fubini's theorem then yields
\begin{eqnarray} \  
\nonumber
\frac{1}{|B_r|} \int_{\mathcal B_r} W^q_h(x,y)\,\frac{dx\, dy}{|x{-}y|^n} & \leq & \left[ \frac{c d^{1-\xi}}{\xi-1}\frac{|\mu |(B_r) }{|B_r|} \right]^{q/p}
\\ \nonumber & &\quad \cdot  c r^{(s-h)q} \left[ \mean{B_r} (d+|w(x)|)^{\xi q/(p-q)}\,dx \right]^{(p-q)/p} \,.
\end{eqnarray}
We then choose 
\eqn{defid}
$$d := \left(\mean{B_r} |w|^{\xi q /(p-q)} \, dx \right)^{(p-q)/ \xi q}\,. $$
Notice that we can assume that $d>0$; otherwise \rif{eq:comp} follows trivially. 
This gives 
\begin{equation} \label{eq:comp est prel 1}
\frac{1}{|B_r|}  \int_{\mathcal B_r} W_h^q(x,y)\,\frac{dx\, dy}{|x{-}y|^n} \leq c   d^{q/p}\left( r^{(s-h)p}\frac{|\mu|(B_r) }{|B_r|}  \right)^{q/p}.
\end{equation}
We now proceed with the choice of $\xi >1$ and we distinguish two cases. The first is when $p \leq n/s$, 
the second is when $p>n/s$. In the first case we claim we can find $0<h_m\equiv h_m(q)<s$ such that
\eqn{torna}
$$
 \frac{q}{p-q} \leq  \frac{nq}{n-h_m q}
$$
so that whenever $h_m< h< s$ we can find $\xi\equiv \xi(h)>1$ such that
\eqn{torna2}
$$
 \frac{\xi q}{p-q} =  \frac{nq}{n-hq}\,.
$$
Indeed, when $p \leq n/s$ then $\bar q = n(p-1)/(n-s)$ so that $q < \bar q $ is equivalent to have $q/(p-q)< nq/(n-sq)$ and therefore the existence of $h_m< s$ such that \rif{torna} holds follows.
We shall then prove \rif{eq:comp} with $h$ such that $s > h > h_m$ and by the observation made at the beginning of the proof this is sufficient to prove the lemma. By the fractional Sobolev inequality (see \cite{DPV}), recalling the choice of $d$ and using \rif{torna2},  we have
\eqn{imbedding}
$$
d
\leq  c r^h \left( \frac{1}{|B_r|}  \int_{\mathcal B_r} W_h^q(x,y)\,\frac{dx\, dy}{|x{-}y|^n}   \right)^{1/q}\,.
$$Inserting this last inequality into~\eqref{eq:comp est prel 1}, and taking into account the formula of $\xi$, leads to
\[
\left(\frac{1}{|B_r|} \int_{\mathcal B_r} W^q_h(x,y)\,\frac{dx\, dy}{|x{-}y|^n}  \right)^{\frac{p-1}{pq}} \leq c
\left[ r^{(s-h)p} r^h \frac{|\mu|(B_r) }{|B_r|}  \right]^{1/p}
\]
with $c \equiv c(n,s,p,\delta)$ and \rif{eq:comp} follows. It remains to treat the case when $p >n/s$, that implies $\bar q=p$. In this case, we can confine ourselves to prove the statement for exponents $q > n/s$ and $h < s$ such that $p > q > n/h$. Again by the remarks made at the beginning of the proof it is sufficient to prove \rif{eq:comp} for such values of $h$ and $q$. Since now $hq>n$, again by the Sobolev embedding theorem we have that \rif{imbedding} holds whenever $\xi>1$ and with $d$ defined as in \rif{defid}, and we can conclude as in the case $p \leq n/s$. 
\end{proof}
A straightforward application of the fractional Sobolev embedding theorem (see \cite{DPV}) together with \rif{eq:comp} gives the following: 
\begin{lemma} \label{lemma:comp est sol}
Assume that $p \geq 2$, $\gamma \in [1,\gamma^*)$, where 
$$
\gamma^* :=  \begin{cases}
\dfrac{n(p-1)}{n- p s}   & p < n/s    \\
 +\infty & p \geq n/s\;.
\end{cases}
$$ Then the inequality
$$
\left(\mean{B_r} |w|^\gamma  \, dx \right)^{1/\gamma} \leq c \left[ \frac{|\mu|(B_r)}{r^{n-sp}} \right]^{1/(p-1)}
$$
holds for a constant $c\equiv c(n,s,p,\Lambda,\gamma^* - \gamma)$.
\end{lemma}
\subsection{Basic estimates in the case $2> p > 2-s/n$}
The counterpart of Lemma \ref{lemma:comp est grad} in the case $p < 2$ turns out to be more involved. 
\begin{lemma} \label{lemma:comp est grad p<2}
Assume that $2> p > p_*= 2-s/n$; for every $q \in [1,\bar q)$ (with
$
\bar q := n(p-1)/(n-s)
$) 
there exists $h(q) \in (0,s)$, such that if $h(q)<h < s$ 
and $\delta := \min\{\bar q - q,p-p_*, s-h\}$ then
\begin{eqnarray} \label{eq:comp-sotto2}
\nonumber &&\hspace{-7mm} \left(\int_{B_r}\mean{B_r} \frac{|w(x){-}w(y)|^q}{|x{-}y|^{n+hq}} \, dx \, dy  \right)^{1/q}  \leq \frac{c}{r^{h}} \left[ \frac{|\mu|(B_r)}{r^{n-sp}} \right]^{1/(p-1)}
\\ && \qquad\qquad\qquad\qquad+ \frac{c}{r^{h(p-1)}}  \left(\int_{B_r}\mean{B_r} \frac{|u(x){-}u(y)|^q}{|x{-}y|^{n+hq}} \, dx \, dy   \right)^{(2-p)/q} \left[ \frac{|\mu|(B_r)}{r^{n-sp}} \right]
\end{eqnarray}
holds whenever $ h(q)\leq h < s$, for a constant $c\equiv c(n,s,p,\Lambda,\delta)$ that blows-up when $\delta \to 0$.
\end{lemma}
\begin{proof}
This time we rewrite
\begin{eqnarray} \  
\nonumber
W_h^q(x,y) & = &   \left[ \frac{(U_s+V_s)^{p-2}(x,y)  W_s^2(x,y)}{(d{+}|w(x)|{+}|w(y)|)^{\xi}} \right]^{q/2}  (U_h+V_h)^{(2-p)q/2}(x,y)
 \\ \nonumber & &\qquad \qquad  \cdot  (d{+}|w(x)|{+}|w(y)|)^{\xi q/2} |x{-}y|^{(s-h)pq/2}\;,
\end{eqnarray}
where again $\xi>1$ is to be chosen in a few lines. H\"older's inequality yields
\begin{eqnarray} \  
\nonumber &&
\frac{1}{|B_r|}\int_{\mathcal B_r}  W_h^q(x,y)  \frac{dx\, dy}{|x-y|^n}  
\\ \nonumber && \quad  \leq 
\left( \frac{1}{|B_r|} \int_{\mathcal B_r} \frac{(U_s+V_s)^{p-2}(x,y)  W_s^2(x,y)}{(d{+}|w(x)|{+}|w(y)|)^{\xi}} \, \frac{dx\, dy}{|x-y|^n} \right)^{q/2}
 \\ \nonumber & & \qquad    \cdot  \left(\frac{1}{|B_r|} \int_{\mathcal B_r} (U_h+V_h)^{q_1 (2-p)q/2}(x,y) \, \frac{dx\, dy}{|x-y|^n}  \right)^{1/q_1}
 \\ \nonumber & & \qquad    \cdot  \left( \frac{1}{|B_r|} \int_{\mathcal B_r} (d{+}|w(x)|{+}|w(y)|)^{\xi q_2 q/2} \, \frac{dx\, dy}{|x-y|^{n-\eps}}  \right)^{1/q_2}
   \label{eq:comp p<2 interm. 1}   = :   J_1 \cdot J_2 \cdot J_3\,,
\end{eqnarray}
where
\[
\frac{q}{2} + \frac{1}{q_1} + \frac{1}{q_2} = 1 \qquad \mbox{and} \qquad \eps = \frac{(s-h)p q q_2}{2} > 0\,.
\]
We choose $q_1 := 2/(2-p)$ that gives $q_2 = 2/(p-q)$\,.
Thus we have by Lemma~\ref{estim1} that 
\eqn{j12}
$$
J_1 \cdot J_2 \leq \left( \frac{c d^{1-\xi}}{\xi-1} \frac{|\mu |(B_r)}{|B_r|}  \right)^{q/2} \left( \frac{1}{|B_r|} \int_{\mathcal B_r} (U_h+V_h)^{q}(x,y) \, \frac{dx\, dy}{|x-y|^n}  \right)^{(2-p)/2}\,.
$$
We then turn to $J_3$. To this end, since $\eps \geq \delta/c > 0$, using Fubini's theorem we have  
\[
J_3 \leq c r^{\eps/q_2} \left( \mean{B_r} (d+|w|)^{\xi q_2 q/2} \, dx   \right)^{1/q_2}\,.
\]
We define 
\eqn{implys}
$$
\xi = \frac{n(p-q)}{n-hq}   \qquad  \Longleftrightarrow \qquad  \frac{nq}{n-hq} = \frac{\xi q_2 q}{2}  = \frac{\xi q}{p-q}
$$
and set
\[
d := \left( \mean{B_r} |w|^{\xi q_2 q/2} \, dx   \right)^{2/(\xi q_2 q)}\,.
\]
We then have that
\eqn{j3}
$$
J_3 \leq c r^{\eps/q_2}  d^{\xi q/2}\,.
$$
Observe that the previous inequality holds provided $\xi >1$; this is the point where the choice of the parameter $h(q)$ mentioned in the statement comes into the play. For this, observe that a few manipulations give
\[
\xi = 1 + \frac{n}{(n-hq)(n-s \bar q)} \left[ (n-ps) (\bar q- q)  + q (s-h)(\bar q-p)\right]\,.
\]
Then, with $q < \bar q$ being fixed as in the statement of the lemma, we find $h(q)< s $ such that the expression in the square 
bracket in the above display remains positive whenever $h(q)< h < s$ (recall that $\bar q < p$). Using \rif{j3} in \rif{j12} we get 
\[
J_1 \cdot J_2 \cdot J_3 \leq c r^{\eps/q_2} \left(d \frac{|\mu |(B_r)}{|B_r|}  \right)^{q/2} \left(  \frac{1}{|B_r|}  \int_{\mathcal B_r} (U_h+V_h)^{q}(x,y) \, \frac{dx\, dy}{|x-y|^n}  \right)^{(2-p)/2} \,.
\]
By \rif{implys} we can apply the fractional Sobolev inequality, thereby getting
\[
 d \leq c r^h \left( \frac{1}{|B_r|} \int_{\mathcal{B}_r}  |W_h(x,y)|^q \, \frac{dx\, dy}{|x-y|^n}  \right)^{1/q } \,.
\]
Thus,~\eqref{eq:comp p<2 interm. 1} implies that
\begin{eqnarray} \  
\nonumber
&& \left( \frac{1}{|B_r|}\int_{\mathcal B_r}  W_h^q(x,y)  \frac{dx\, dy}{|x-y|^n}  \right)^{1/q}
\\ \nonumber && \qquad \leq c r^{h+2 \eps /(q q_2)} \left(  \frac{1}{|B_r|}  \int_{\mathcal B_r} (U_h+V_h)^{q}(x,y) \, \frac{dx\, dy}{|x-y|^n}  \right)^{(2-p)/q} \left[ \frac{|\mu |(B_r)}{|B_r|}  \right]\,.
\end{eqnarray}
Finally, since 
\[
\frac{2\eps }{q q_2} = \frac{(s-h)p q q_2}{2}\frac{2}{q q_2} = (s-h)p\,,
\]
and recalling the definitions of $W_h, U_h$ and $V_h$ we conclude with 
\begin{eqnarray*} 
\nonumber &&\hspace{-7mm} \left(\int_{B_r}\mean{B_r} \frac{|w(x){-}w(y)|^q}{|x{-}y|^{n+hq}} \, dx \, dy  \right)^{1/q}  
\\ && \leq  \frac{c}{r^{h(p-1)}}  \left(\int_{B_r}\mean{B_r} \frac{|w(x){-}w(y)|^q}{|x{-}y|^{n+hq}} \, dx \, dy   \right)^{(2-p)/q} \left[ \frac{|\mu|(B_r)}{r^{n-sp}} \right]\\
&& \qquad + \frac{c}{r^{h(p-1)}}  \left(\int_{B_r}\mean{B_r} \frac{|u(x){-}u(y)|^q}{|x{-}y|^{n+hq}} \, dx \, dy   \right)^{(2-p)/q} \left[ \frac{|\mu|(B_r)}{r^{n-sp}} \right]\,.
\end{eqnarray*}
In turn, since $p<2$ we can use Young's inequality with conjugate exponents $1/(p-1)$ and $1/(2-p)$ in order to estimate
\begin{eqnarray*} 
\nonumber &&\hspace{-7mm}   \frac{c}{r^{h(p-1)}}  \left(\int_{B_r}\mean{B_r} \frac{|w(x){-}w(y)|^q}{|x{-}y|^{n+hq}} \, dx \, dy   \right)^{(2-p)/q} \left[ \frac{|\mu|(B_r)}{r^{n-sp}} \right]\\
&&  \qquad \leq  \frac{1}{2}\left(\int_{B_r}\mean{B_r} \frac{|w(x){-}w(y)|^q}{|x{-}y|^{n+hq}} \, dx \, dy  \right)^{1/q}  +  \frac{c}{r^{h}} \left[ \frac{|\mu|(B_r)}{r^{n-sp}} \right]^{1/(p-1)}\,.
\end{eqnarray*}
Connecting the content of the last two displays and reabsorbing terms finally yields 	\rif{eq:comp-sotto2} and the proof is complete. 
\end{proof}
The subquadratic analog of Lemma \ref{lemma:comp est sol} is much more delicate and requires a completely nontrivial proof relying on Caccioppoli estimates for solutions to homogeneous problems and the comparison argument of the previous lemma. 
\begin{lemma} \label{lemma:comp p<2}
Assume that $2> p >p_*= 2-n/s$ and 
$$
1 \leq \gamma < \gamma^* =\dfrac{n(p-1)}{n- p s} \;.
$$ Then there exists a constant 
$c\equiv c(n,s,p,\Lambda,\gamma^* - \gamma, p-p_*)$ such that the following inequality holds:
\eqn{eq:comp sol p<2} 
$$
 \left(\mean{B_{r}} |w|^\gamma \, dx  \right)^{1/\gamma} \leq c\left[ \frac{|\mu|(B_{2r})}{r^{n-sp}} \right]^{1/(p-1)}+ c[E(u;x_0,2r)]^{2-p}\left[ \frac{|\mu|(B_{2r})}{r^{n-sp}} \right]\;.
$$
Moreover, for every $q \in [1,\bar q)$ 
there exists $h(q) \in (0,s)$, such that if $h(q)<h < s$ such that the inequality
\begin{eqnarray} \label{eq:comp grad p<2}
\nonumber
 \left(\int_{B_r}\mean{B_r} \frac{|w(x){-}w(y)|^q}{|x{-}y|^{n+hq}} \, dx \, dy  \right)^{1/q}& \leq& \frac{c}{r^{h}} \left[ \frac{|\mu|(B_{2r})}{r^{n-sp}} \right]^{1/(p-1)}
\\  && \quad  + \frac{c}{r^{h}}   [E(u;x_0,2r)]^{2-p}\left[ \frac{|\mu|(B_{2r})}{r^{n-sp}} \right]
\end{eqnarray} 
holds for a constant $c$ depending only on $n,s,p,\Lambda$ and $\delta$, where the meaning of $\bar q, \delta$ is specified in Lemma \ref{lemma:comp est grad p<2}. 
\end{lemma}
\begin{proof} All the balls here will have $x_0$ as centre; accordingly, we denote ${\rm Tail}(f;s)\equiv 
{\rm Tail}(f;x_0,s)$ and $E(f;s)\equiv 
E(f;x_0,s)$ for $s \in (0,2r)$. We observe that it is sufficient to prove \rif{eq:comp grad p<2}, as \rif{eq:comp sol p<2}  will then follow via fractional Sobolev inequality. Moreover, in the rest of the proof $h$ and $q$ are taken as in Lemma \ref{lemma:comp est grad p<2} and we denote
\[
F(\varphi;t) := \left(\int_{B_t}\mean{B_t} \frac{|\varphi(x){-}\varphi(y)|^q}{|x{-}y|^{n+hq}} \, dx \, dy  \right)^{1/q} 
\]
for $\varphi \in W^{h,q}(B_t)$, $t \in (0, 2r)$. For $1 \leq \sigma' < \sigma \leq 2$ we define
 $v_\sigma \in W^{s,p}(\ern)$ as the weak solution to the problem
\begin{equation}
\left \{
\begin{array}{ccc}
-\widetilde{\mathcal{L}}_{u}  v_\sigma=0&\mbox{in}&B_{\sigma r}\\
v=u&\mbox{in}&\mathbb R^n \backslash B_{\sigma r}\;.
\end{array} \right. 
\end{equation}
The operator $-\widetilde{\mathcal{L}}_{u} $ has been defined in \rif{abbreviated}. 
We start by trivially estimating 
\eqn{tric1}
$$
F(u;\sigma' r) \leq F(v_\sigma;\sigma' r) + F(u{-}v_\sigma;\sigma' r)\;. 
$$
By Lemma~\ref{lemma:comp est grad p<2} applied to $v_\sigma$, we have that 
\eqn{tric2}
$$
F(u{-}v_\sigma;\sigma r) \leq \frac{c}{r^{h}} \left[ \frac{|\mu|(B_{2r})}{r^{n-sp}} \right]^{1/(p-1)}+ \frac{cF(u;\sigma r)^{2-p}}{ r^{h(p-1)}}\left[ \frac{|\mu|(B_{2r})}{r^{n-sp}} \right] \;.
$$
Moreover, by Lemma~\ref{lemma:cacc improved 2} (applied to $v_\sigma-(v_{\sigma})_{B_{\sigma r}}$, see Remark \ref{costanti}) we obtain
\begin{eqnarray}
\nonumber && F(v_\sigma;\sigma' r)  
\\  && \qquad \leq
\frac{ c r^{-h}}{(\sigma-\sigma')^{\theta/q}} \left[  \mean{B_{\sigma r}} |v_\sigma - (v_\sigma)_{B_{\sigma r}}|\, dx+{\rm Tail}(v_\sigma-(v_\sigma)_{B_{\sigma r}};\sigma r/2)  \right]\;.\label{rils}
\end{eqnarray}
The first term on the right-hand side can be estimated as
 \begin{eqnarray*} \  
\nonumber
\mean{B_{\sigma r}} |v_\sigma - (v_\sigma)_{B_{\sigma r}}|\, dx & \leq & 2 \mean{B_{\sigma r}} |u{-}v_\sigma|\, dx + c(n) 
\mean{B_{2r}} |u{-}(u)_{B_{2r}}| \, dx\\
& \leq & 2 \mean{B_{\sigma r}} |u{-}v_\sigma|\, dx + c
E(u;2r)\,.
\end{eqnarray*}
We use the fractional Sobolev inequality (see \cite{DPV}) to estimate
\eqn{frale}
$$
 \mean{B_{\sigma r}} |u{-}v_\sigma| \, dx \leq c r^h F(u{-}v_\sigma;\sigma r)\,,
$$
and we therefore conclude with 
\eqn{unav}
$$
\mean{B_{\sigma r}} |v_\sigma - (v_\sigma)_{B_{\sigma r}}|\, dx \leq c E(u;2r)+c r^h F(u{-}v_\sigma;\sigma r)\;.
$$
For the second term appearing on the right-hand side of \rif{rils} we instead have
\begin{eqnarray} \  
\nonumber 
{\rm Tail}(v_\sigma-(v_{\sigma })_{B_{\sigma r}};\sigma r/2) 
& \leq & c {\rm Tail}(u{-}(u)_{B_{\sigma r}};\sigma r/2) \\
 \nonumber && \quad  + c {\rm Tail}(u{-}v_\sigma-(u{-}v_\sigma)_{B_{\sigma r}};\sigma r/2) \;.
\end{eqnarray}
Now, the first term can be estimated by means of Lemma \ref{lemma:basic E}(2), thereby obtaining
\begin{eqnarray} \  
\nonumber &&
{\rm Tail}(u{-}(u)_{B_{\sigma r}};\sigma r/2) 
\\ \nonumber && \qquad \leq {\rm Tail}(u{-}(u)_{B_{2r}};\sigma r/2)    + c |(u)_{B_{\sigma r}}-(u)_{B_{2r}}| 
\\ \nonumber && \qquad \leq c {\rm Tail}(u{-}(u)_{B_{2r}};\sigma r/2)   + c \mean{B_{2r}} |u{-}(u)_{B_{2r}}| \, dx 
\leq c E(u;2r)\,.
\end{eqnarray}
As for the second, again by Lemma \ref{lemma:basic E}(2) and \rif{frale}, we have: 
\begin{eqnarray*} \  
\nonumber &&
{\rm Tail}(u{-}v_\sigma-(u{-}v_\sigma)_{B_{\sigma r}};\sigma r/2) 
\\ \nonumber && \qquad \leq  c {\rm Tail}(u{-}v_\sigma;\sigma r/2)  + c |(u{-}v_\sigma)_{B_{\sigma r}}|
\\ \nonumber && \qquad \leq  c {\rm Tail}(u{-}v_\sigma;\sigma r) 
+ c \mean{B_{\sigma r}} |u{-}v_\sigma| \, dx\\ &
 & \qquad  \leq c \mean{B_{\sigma r}} |u{-}v_\sigma| \, dx\leq c r^h F(u{-}v_\sigma;\sigma r)\,,
\end{eqnarray*}
since $u = v_\sigma$ outside of $B_{\sigma r}$. Using the last three displays we get
$$
{\rm Tail}(v_\sigma-(v_{\sigma })_{B_{\sigma r}};\sigma r/2) \leq c E(u;2r)+ c r^h F(u{-}v_\sigma;\sigma r)\;.
$$
Combining the this last inequality with \rif{rils} and \rif{unav}, we finally arrive at 
\[
F(v_\sigma;\sigma' r)   \leq \frac{ c}{(\sigma-\sigma')^{\theta/q}}    \left[  r^{-h}E(u;2r) +  F(u{-}v_\sigma;\sigma r) \right]\,.
\]
In turn, this, together with \rif{tric1}-\rif{tric2}, yields 
\begin{eqnarray*}
F(u;\sigma' r)  &\leq & \frac{ c }{(\sigma-\sigma')^{\theta/q}} \left\{r^{-h}E(u;2r) 
+ \frac{1}{r^{h}}\left[\frac{|\mu|(B_{2r})}{r^{n-sp}} \right]^{1/(p-1)}\right\}\\
&&\qquad  \qquad \qquad \quad +\frac{ c }{(\sigma-\sigma')^{\theta/q}} \frac{F(u;\sigma r)^{2-p}}{ r^{h(p-1)}}\left[ \frac{|\mu|(B_{2r})}{r^{n-sp}} \right]\;.
\end{eqnarray*}
Applying Young's inequality with conjugate exponents $(1/(2-p),1/(p-1))$, we obtain  
\[
F(u;\sigma' r)   \leq \frac12 F(u;\sigma r) + \frac{ c r^{-h}}{(\sigma-\sigma')^{\frac{\theta}{q(p-1)}}}    \left\{ E(u;2r) +  
  \left[ \frac{|\mu|(B_{2r})}{r^{n-sp}} \right]^{1/(p-1)}\right\}\,, 
\]
that now holds whenever $1 \leq \sigma' < \sigma \leq 2$ and with a constant $c$ which is independent of $\sigma, \sigma'$. 
A standard iteration argument, see for example~\cite{giusti}, implies that 
\[
\left(\int_{B_r}\mean{B_r} \frac{|u(x){-}u(y)|^q}{|x{-}y|^{n+hq}} \, dx \, dy  \right)^{1/q}   \leq  \frac{c}{r^{h}}  \left\{E(u;2r) +    \left[ \frac{|\mu|(B_{2r})}{r^{n-sp}} \right]^{1/(p-1)}\right\}\,.
\]
This last inequality and \rif{eq:comp-sotto2} yields \rif{eq:comp grad p<2} after an elementary manipulation.
\end{proof}
With the same proof of Lemmas \ref{estim1}-\ref{lemma:comp est grad p<2} we have the following one, where the ball $B_r$ is replaced by any open bounded subset $\Omega$:
\begin{lemma} \label{lemma:existence comp}
Let $\Omega\subset \er^n$ be an open and bounded domain. 
Let $u, v \in W^{s,p}(\ern)$ be weak solutions to the problems
$$
\left \{
\begin{array}{ccc}
-\mathcal{L}_{\Phi} u = \mu \in C^{\infty}(\ern)&\mbox{in}&\Omega\\
u=g&\mbox{in}&\mathbb R^n \backslash \Omega
\end{array} \right. \quad \mbox{and} \quad \left \{
\begin{array}{ccc}
-\widetilde{\mathcal{L}}_{u} v=0&\mbox{in}&\Omega\\
v=g&\mbox{in}&\mathbb R^n \backslash \Omega\, ,
\end{array} \right. 
$$
respectively, for some $g \in  W^{s,p} (\ern)
$; let $w:= u-v$, $\bar q := n(p-1)/(n-s)$. Then: 
\begin{itemize}
\item When $p \geq 2$, for every $h \in (0,s)$ and $q \in [1,\bar q)$, and with  $\delta := \min\{\bar q - q,s-h\}$, the inequality 
$$ \left(\int_{\Omega}\int_{\Omega} \frac{|w(x){-}w(y)|^q}{|x{-}y|^{n+hq}} \, dx \, dy  \right)^{1/q} 
\leq  c\left[ |\mu|(\Omega)\right]^{1/(p-1)}
$$
holds for a constant $c \equiv c(n,s,p,\Lambda, \delta, \Omega)$.
\item When $2> p > p_* = 2-s/n$, for every $q \in [1,\bar q)$ 
there exists $h(q) \in (0,s)$, such that if $h(q)<h < s$ 
and $\delta := \min\{\bar q - q,p-p_*, s-h\}$ then the inequality \begin{eqnarray*} 
\nonumber && \left(\int_{\Omega}\int_{\Omega} \frac{|w(x){-}w(y)|^q}{|x{-}y|^{n+hq}} \, dx \, dy  \right)^{1/q} 
\\ && \qquad \leq c \left[ |\mu|(\Omega)\right]^{1/(p-1)}+ c \left(\int_{\Omega}\int_{\Omega} \frac{|u(x){-}u(y)|^q}{|x{-}y|^{n+hq}} \, dx \, dy   \right)^{(2-p)/q} \left[
|\mu|(\Omega) \right]
\end{eqnarray*}
holds for a constant $c \equiv c(n,s,p,\Lambda, \delta, \Omega)$\;.
\end{itemize}
\end{lemma}
\section{Proof of Theorem \ref{existence}}\label{esistenza}
In this section we show the existence of SOLA; we shall widely use the operators $-\widetilde{\mathcal{L}}_{u_j}$ defined in \rif{abbreviated} for a suitable sequence of approximating solutions $\{u_j\}$. 
\subsection{Step 1: Construction of the approximating problems.} We start with a technical lemma. 
\begin{lemma}[Construction of the approximating boundary values $g_j$] \label{lemma:g_j} Let $z \in \Omega$; there exists a sequence $\{g_j\} \subset C_0^\infty(\ern)$ 
such that for any $R>0$ 
\eqn{bound0}
$$
g_j \to g \quad \mbox{in } \;  W^{s,p}(B_R) \quad \mbox{and} \quad \int_{\ern \setminus B_{R}(z)} \frac{|g_j(y)-g(y)|^{p-1}}{ |y{-}z|^{n+sp}} \, dy \to 0
$$
as $j \to \infty$. Moreover, for every $\eps>0$ there exist a radius $\tilde R$ and an index $\tilde j$, both depending on $\eps$, such that 
\eqn{bound00}
$$
\int_{\ern \setminus B_{R}(z)} \frac{|g(y)|^{p-1} + |g_j(y)|^{p-1}}{ |y{-}z|^{n+sp}} \, dy \leq \eps 
$$
holds whenever $j\geq \tilde j$ and $R\geq \tilde R$. Finally, for every $R>0$ there exists a constant $c_R$, depending on $R$ and $g(\cdot)$, such that
\eqn{bound}
$$
\sup_j \|g_j\|_{W^{s,p}(B_R)}  \leq c_R\;.
$$
\end{lemma}
\begin{proof} First, define a sequence of balls $B^j \equiv B_{r_j}(z)$, $r_{j+1} \geq 2r_j$ and functions  $g_j \in C_0^\infty(2 B^j)$, such that
\eqn{conv1}
$$
\int_{\ern \setminus B^j} \frac{|g(y)|^{p-1}}{ |y{-}z|^{n+sp}} \, dy+ \| g_j-g \|_{L^p(2B^j)} +\| g_j-g \|_{W^{s,p}(B^j)} \leq \frac{1}{j+1}\;.
$$
This gives the first convergence in \rif{bound0} and also \rif{bound}. 
We next prove the second convergence in \rif{bound0}. Indeed, for large enough $j$ we find that $B_R \subset B^j$ and
\begin{eqnarray} \  
\nonumber &&
\int_{\ern \setminus B_{R}(z)} \frac{|g_j(y)-g(y)|^{p-1}}{ |y{-}z|^{n+sp}} \, dy
\\ \nonumber &&  \qquad  =  \int_{2B^j \setminus B_{R}(z)} \frac{|g_j(y)-g(y)|^{p-1}}{ |y{-}z|^{n+sp}} \, dy
+  \int_{\ern \setminus 2B^j} \frac{|g(y)|^{p-1}}{ |y{-}z|^{n+sp}} \, dy
\\ \nonumber &&  \qquad  \leq  \left(\int_{2B^j \setminus B_{R}(z)} \frac{dy}{|y{-}z|^{p(n+sp)}}  \right)^{1/p} \|g_j-g \|_{L^p(2B^j)}^{p-1} + \frac{1}{j+1}
\\ \nonumber &&  \qquad  \leq \frac{c }{R^{sp}(j+1)^{p-1}} +  \frac{1}{j+1}\,,
\end{eqnarray}
where we have used H\"older's inequality and \rif{conv1}; this yields the second convergence in \rif{bound0}. Triangle inequality then gives
$$
\int_{\ern \setminus B_{R}(z)} \frac{|g_j(y)|^{p-1}}{ |y{-}z|^{n-sp} }\, dy  \leq  \frac{c }{R^{sp}(j+1)^{p-1}} +  \frac{c}{j+1} + c \int_{\ern \setminus B_R(z)} \frac{|g(y)|^{p-1}}{ |y{-}z|^{n+sp}} \, dy 
$$
and the assertions concerning \rif{bound00} also follows again recalling \rif{conv1}. 
\end{proof}
We now use the sequence $\{g_j\}$ to solve Dirichlet problems and build the approximating sequence $\{u_j\}$ required in Definition \ref{soladef}. Let $\{\mu_j\} \subset C_0^\infty(\ern)$ be a sequence of functions weakly converging in the sense of measures to $\mu$ and let $\{g_j\}\subset C_0^\infty(\ern)$ be a sequence of functions 
converging to $g$ as in Lemma~\ref{lemma:g_j}. By a standard variational argument - see Remark \ref{solvability} - we can then define, for every positive integer $j$, the functions $u_j,v_j \in  W^{s,p}(\ern)$ as weak solutions to the problems 
\eqn{approximateu}
$$
\left \{
\begin{array}{ccc}
-\mathcal{L}_{\Phi} u_j = \mu_j&\mbox{in}&\Omega\\
u_j=g_j&\mbox{in}&\mathbb R^n \backslash \Omega
\end{array} \right. \qquad \mbox{and} \qquad \left \{
\begin{array}{ccc}
-\widetilde{\mathcal{L}}_{u_j} v_j=0&\mbox{in}&\Omega\\
v_j=g_j&\mbox{in}&\mathbb R^n \backslash \Omega\, ,
\end{array} \right. 
$$
respectively. We recall here that the operator $-\widetilde{\mathcal{L}}_{u_j}$ has been defined in \rif{abbreviated} and described in Section \ref{reduction}.
\begin{lemma}[Initial coercivity estimate with Tail]\label{lemma:energy} 
There exist constants $c\equiv c(n,s,p,\Lambda)$ and $\sigma \equiv \sigma(n,s,p,\Lambda)$, $\sigma \in (0,1)$, such that
if $B_R \equiv B_R(z)$ is a ball such that $\Omega \subset B_{ \sigma R}(z)$ and $z \in \Omega$, then
\eqn{coerc} 
$$
\left[ v_j \right]_{W^{s,p}(B_R)} \leq c \left[ g_j \right]_{W^{s,p}(B_R)} + c R^{-s}\left[ \|g_j\|_{L^p(B_R)}+ {\rm Tail}(g_j;z,R)\right]\;.
$$
\end{lemma}
\begin{proof}
Let us denote, as in Section \ref{reduction}, $\Phi_p(t) = |t|^{p-2} t$ and $\widetilde K \equiv \widetilde K_{u_j,\Phi,K}$. The function $v_j$ satisfies 
$$\int_{\ern} \int_{\ern} \Phi_p(v_j(x){-}v_j(y))(\varphi(x){-}\varphi(y))\widetilde K(x,y)\,dx \,dy=0$$
for any $\varphi \in C^\infty_0(\Omega)$. Testing against $\varphi = v_j-g_j$, which is possible up to a standard density argument, and recalling that $v_j=g_j$ outside $\Omega$, one gets that
\begin{eqnarray} \  
\nonumber 0 & = & \int_{B_R} \int_{B_R} \Phi_p(v_j(x){-}v_j(y))(v_j(x){-}v_j(y) - (g_j(x){-}g_j(y)) \widetilde K(x,y) \, dx \, dy
\\ \nonumber  &  & + 2 \int_{\ern\setminus B_R} \int_{B_R} \Phi_p(v_j(x){-}g_j(y))(v_j(x){-}g_j(x)) \widetilde K(x,y) \, dx \, dy
\\ 
\nonumber  & =: & I_1 + I_2\,. 
\end{eqnarray}
Using \rif{thekernel} and \rif{monophi}, Young's inequality yields
\[
I_1 \geq \frac1c \left[ v_j \right]_{W^{s,p}(B_R)}^p - c \left[ g_j \right]_{W^{s,p}(B_R)}^p
\]
for a constant $c$ depending only on $n, p$ and $\Lambda$. As for $I_2$, again using \rif{thekernel} and \rif{monophi} 
we have
\begin{eqnarray} \  
\nonumber I_2  & \geq & -c  \int_{\ern\setminus B_R} \int_{B_R} \frac{|v_j(x)|^{p-1}|v_j(x){-}g_j(x)|}{|x{-}y|^{n+sp}}\, dx \, dy 
\\ \nonumber && -c  \int_{\ern\setminus B_R} \int_{B_R} \frac{|g_j(y)|^{p-1}|v_j(x){-}g_j(x)|}{|x{-}y|^{n+sp}}\, dx \, dy 
=:  - I_{2,1} - I_{2,2}\,,
\end{eqnarray}
for a constant depending only on $p,\Lambda$. 
Now, observe that the distance from $\ern\setminus B_R$ to the support of $v_j(x){-}g_j(x)$ is larger than $R/2$ (recall we are assuming $\Omega \subset B_{R/2}(z)$); therefore we can always estimate 
\eqn{distanza}
$$
|x{-}y|^{-n-sp} \leq 8^{n+p}|y-z|^{-n-sp}\qquad \forall \  \ y \in \ern \setminus B_R\ \mbox{and} \   x \in B_R\ .
$$Therefore we have, using fractional Sobolev inequality and Young's inequality with $\eps \in (0,1)$, that
\begin{eqnarray*} 
I_{2,1}&\leq & c \int_{\ern\setminus B_R} \int_{B_R} \frac{|v_j(x)|^{p-1}|v_j(x){-}g_j(x)|}{|y{-}z|^{n+sp}}\, dx \, dy\\
&\leq & \frac{c}{R^{sp}} \int_{B_R} |v_j|^{p-1}|v_j{-}g_j|\, dx\\
&\leq & \frac{c(\eps)}{R^{sp}} \int_{B_R} |v_j|^{p}\, dx+\frac{c\eps}{R^{sp}}\int_{B_R} |v_j{-}g_j|^p\, dx
\\&\leq & c(\eps)R^{-sp}\|v_j\|_{L^{p}(B_R)}^p+c\eps\left[ v_j - g_j \right]_{W^{s,p}(B_R)}^p\\&\leq & c(\eps)R^{-sp}\|v_j\|_{L^{p}(B_R)}^p+c\eps[ v_j]_{W^{s,p}(B_R)}^p +c\eps[ g_j]_{W^{s,p}(B_R)}^p\;.
\end{eqnarray*}  
Since fractional Sobolev's inequality implies 
\begin{eqnarray} 
\notag
\|v_j\|_{L^{p}(B_R)} & \leq & \|v_j-g_j\|_{L^{p}(B_{\sigma R})} + \|g_j\|_{L^{p}(B_R)} 
\\ \nonumber& \leq & c R^{sp} \sigma^{sp} [ v_j]_{W^{s,p}(B_R)} +c R^{sp} \sigma^{sp}[ g_j]_{W^{s,p}(B_R)}+ \|g_j\|_{L^{p}(B_R)}\,,
\end{eqnarray}
we actually get
\[
I_{2,1} \leq c(\eps) \left( [ g_j]_{W^{s,p}(B_R)} + \|g_j\|_{L^{p}(B_R)} \right) + \left(c(\eps) \sigma^{sp} + c\eps \right)[ v_j]_{W^{s,p}(B_R)}\,.
\]
In a similar, actually simpler fashion, we have
\begin{eqnarray*}  
\nonumber
I_{2,2} &\leq& c \int_{\ern \setminus B_R} \frac{|g_j(y)|^{p-1}}{|y-z|^{n+sp}} \, dy  \int_{B_R} |v_j{-}g_j| \, dx\\
&\leq & c [{\rm Tail}(g_j;z,R)]^{p-1}R^{s(1-p)}\left[ v_j - g_j \right]_{W^{s,p}(B_R)}\\
&\leq & \eps \left[ v_j \right]_{W^{s,p}(B_R)}^p + \eps \left[g_j \right]_{W^{s,p}(B_R)}^p+ c(\eps)R^{-sp} [{\rm Tail}(g_j;z,R)]^{p}\;.
\end{eqnarray*}
Combining all the estimates found for the terms $I_1, I_2, I_{2,1}, I_{2,2}$ and choosing first $\eps$ small enough and then $\sigma$ small enough accordingly, both depending only on $n,s,p,\Lambda$, in order to reabsorb the terms featuring $\left[ v_j\right]_{W^{s,p}(B_R)}^p$ appearing on the right, we conclude with \rif{coerc}.
\end{proof}
\subsection{Step 2: A priori estimates for approximate solutions $\{u_j\}$.} With the sequences $\{u_j\}$ and $\{v_j\}$ being defined in \rif{approximateu},  
we further define the sequence $\{w_j\}$ by letting $w_j=u_j-v_j$, and, for $x\not=y$, and finally introduce the new functions
$$U^j_h(x,y)=\frac{|u_j(x){-}u_j(y)|}{|x{-}y|^{h}}\,, \; \; 
V^j_h(x,y)=\frac{|v_j(x){-}v_j(y)|}{|x{-}y|^{h}}\,, \; \; 
W^j_h(x,y)=\frac{|w_j(x){-}w_j(y)|}{|x{-}y|^{h}}$$
for any $h \in (0,s)$. We then have the following uniform bound:
\begin{lemma}\label{bounds}
Let $p>2$,  $h \in (0,s)$ and $q \in [1,\bar q)$, where $\bar q$ has been defined in \trif{condizioni}. 
Then there exists a constant $c$ depending only on $n,s,p,\Lambda, s-h,\bar q-q,g(\cdot),\Omega$ such that 
\begin{equation}\label{bound1}
\left(\int_{\Omega} |u_j|^q  \, dx  \right)^{1/q} +  \left(\int_{\Omega \times \Omega} (U^j_h)^q(x,y) \, \frac{dx\, dy}{|x{-}y|^n}  \right)^{1/q} \leq c
\end{equation}
holds for all $j \in \en$. In the case $2> p > p_* = 2-s/n$, for every $q \in [1,\bar q)$ 
there exists $h(q) \in (0,s)$, such that if $h(q)<h < s$ then estimate \trif{bound1} continues to hold and the constant $c$ additionally depends on $p-p_*$. 
\end{lemma}
\begin{proof} First of all, since $\Omega$ is bounded, we find a point $z \in \Omega$ such that $\Omega \subset B_{\sigma R}(z)$ with $R := 4\sigma^{-1} \diam(\Omega)$, where $\sigma$ is as in Lemma~\ref{lemma:energy}. For such $R$, using Lemma~\ref{lemma:g_j} and Lemma~\ref{lemma:energy}, we infer that
$$
\left[ v_j \right]_{W^{s,p}(\Omega)} \leq c(g(\cdot),\Omega)
$$
holds for all $j\in \en$. It then follows that 
\begin{eqnarray} \  
\nonumber && \int_{\Omega \times \Omega} (V^j_h)^q(x,y) \, \frac{dx\, dy}{|x{-}y|^n}  = \int_{\Omega \times \Omega} \frac{|v_j(x){-}v_j(y)|^q}{|x{-}y|^{s q}}\frac{1}{|x{-}y|^{(h-s)q}}\, \frac{dx\, dy}{|x{-}y|^{n}}
\\ \nonumber && \qquad \leq \left[ v_j \right]_{W^{s,p}(\Omega)}^{q}
\left(\int_{B_{R/2} \times B_{R/2}} \frac{dx\, dy}{|x{-}y|^{n-(s-h)qp/(p-q)}} \right)^{(p-q)/p}\leq c
\end{eqnarray}
for another constant $c$ depending on $g(\cdot)$, $\Omega$ and $s-h$. 
By trivially estimating 
\begin{eqnarray} 
\nonumber
\left(\int_{\Omega \times \Omega} (U^j_h)^q(x,y) \, \frac{dx\, dy}{|x{-}y|^n}  \right)^{1/q} & \leq & \left(\int_{\Omega \times \Omega} (V^j_h)^q(x,y) \, \frac{dx\, dy}{|x{-}y|^n}  \right)^{1/q} 
\\ \nonumber & & \qquad + \left(\int_{\Omega \times \Omega} (W^j_h)^q(x,y) \, \frac{dx\, dy}{|x{-}y|^n}  \right)^{1/q}
\end{eqnarray}
we see that it remains to treat the second terms, and for this we shall employ Lemma ~\ref{lemma:existence comp} for the obvious choice $u \equiv u_j$ and $v\equiv v_j$. 
We distinguish between cases $p\geq 2$ and $2-s/n< p<2$. For $p\geq 2$, the terms are uniformly bounded by the weak convergence of $\mu_j$ to $\mu$. In the case for $2-s/n < p<2$, using the content of the last display again with the estimate from Lemma \ref{lemma:existence comp} (for the specified values of $q$ and $h(q)< h < s$), we instead have
\begin{eqnarray*}
 &&\left(\int_{\Omega \times \Omega} (U^j_h)^q(x,y) \, \frac{dx\, dy}{|x{-}y|^n}  \right)^{1/q}\leq c \left(\int_{\Omega \times \Omega} (V^j_h)^q(x,y) \, \frac{dx\, dy}{|x{-}y|^n}  \right)^{1/q} \\ && \qquad  \qquad 
+c\left[ |\mu_j|(\Omega)\right]^{1/(p-1)}+ c \left(\int_{\Omega \times \Omega} (U^j_h)^q(x,y) \, \frac{dx\, dy}{|x{-}y|^n}  \right)^{(2-p)/q}  \left[
|\mu_j|(\Omega) \right]\,.  
\end{eqnarray*}
We can now use Young's inequality with conjugate exponents $(1/(2-p), 1/(p-1))$, thereby getting again a uniform 
bound - with respect to $j$ - on the quantities appearing in the left hand side. Finally, 
the uniform boundedness of $\{\|u_j\|_{L^{q}(\Omega)}\}$ follows easily by fractional Sobolev 
inequality (see \cite{DPV}) applied to $u_j-g_j$ and the proof of \rif{bound1} is complete. 
\end{proof}
\subsection{Step 3: Passage to the limit and existence of SOLA}
From the results in the previous step we can conclude that, up to non-relabelled subsequences and using a diagonal argument, there exists $u \in W^{h,q}(\Omega)$ such that $u=g$ in $\ern\setminus \Omega$ and such that 
\eqn{convergence}
$$
\left\{
\begin{array}{ccc}
u_j \rightharpoonup u &\mbox{in} &W^{h,q}(\Omega) \\
u_j \rightarrow u &\mbox{in} &L^q(\Omega)\\
u_j \rightarrow u &\mbox{a.e. in} &\ern 
\end{array}\right.
$$
hold for any given $h \in (0,s)$ and $q\in [p-1,\bar q)$. 
Notice that we have used the compact embedding of $W^{h,q}$ in $L^q$ and the fact that $W^{h_1,q}(\Omega) \subset W^{h_2, q}(\Omega)$ holds provided $h_2 \leq h_1$. By Lemma~\ref{bounds} and Fatou's lemma we then have
\begin{equation}\label{bound2}
\left(\int_{\Omega \times \Omega} (U_h)^q(x,y) \, \frac{dx\, dy}{|x-y|^n}  \right)^{1/q} + \left(\int_{\Omega \times \Omega} (U^j_h)^q(x,y) \, \frac{dx\, dy}{|x-y|^n}  \right)^{1/q} \leq c
\end{equation}
and
\begin{equation}\label{bound3}
\left(\int_{\Omega } |u(x)|^q \, dx  \right)^{1/q} + \left(\int_{\Omega} |u_j(x)|^q \, dx   \right)^{1/q} \leq c
\end{equation}
hold for a constant $c$ is as in Lemma \ref{lemma:energy}. 
We recall that $U_h(\cdot)$ has been defined in \rif{dv0}. We now pass to the limit in the weak formulations of $-\mathcal{L}_{\Phi} u_j = \mu_j$ showing that $u$ is a weak solution to $-\mathcal{L}_{\Phi} u = \mu$. To this end, let $\varphi\in C^\infty_0(\Omega)$ be such that $\text{supp}\,\varphi \subset  \Omega$. The weak formulation for $u_j$ then gives us
\begin{eqnarray}   
\nonumber 
&& \int_{\ern} \varphi \,d\mu_j \; = \;   \int_{\ern} \int_{\ern}\Phi(u(x){-}u(y))(\varphi(x){-}\varphi(y))K(x,y)\,dx\,dy
\\  && \quad +\int_{\ern} \int_{\ern}(\Phi(u_j(x){-}u_j(y))-\Phi(u(x){-}u(y)))(\varphi(x){-}\varphi(y))K(x,y)\,dx\, dy\,.\label{lasty}
\end{eqnarray}
Call $I_j$ the last integral and and notice that by \rif{bound0}, \rif{approximateu} and \rif{convergence}, it follows that 
$u\equiv g$ a.e. outside $\Omega$; 
then decompose $I_j$ as
\begin{eqnarray} \  
\nonumber
I_j  & = & \int_{\Omega} \int_{\Omega}(\Phi(u_j(x){-}u_j(y))-\Phi(u(x){-}u(y)))(\varphi(x){-}\varphi(y))K(x,y)\,dx\,dy
\\ \nonumber & & +\int_{B_R \setminus \Omega} \int_{\Omega}(\Phi(u_j(x)-g_j(y))-\Phi(u(x)-g(y))) \varphi(x) K(x,y)\,dx\,dy
\\ \nonumber & & +\int_{\ern \setminus B_R} \int_{\Omega}(\Phi(u_j(x)-g_j(y))-\Phi(u(x)-g(y))) \varphi(x) K(x,y)\,dx\,dy
\\ \nonumber & & - \int_{\Omega}\int_{B_R \setminus \Omega}(\Phi(g_j(x)-u_j(y))-\Phi(g(x)-u(y))) \varphi(y) K(x,y)\,dx\,dy
\\ \nonumber & & - \int_{\Omega}\int_{\ern \setminus B_R} (\Phi(g_j(x)-u_j(y))-\Phi(g(x)-u(y))) \varphi(y) K(x,y)\,dx\,dy
\\& =: & I_{j,1} + I_{j,2}+ I_{j,3}- I_{j,4}- I_{j,5}\,,\label{tony}
\end{eqnarray}
where $\Omega \subset B_R(z)$, $z \in {\rm supp}\, \varphi$ and we are going to choose $B_R(z)$ large enough, eventually. We concentrate our attention on the terms $I_{j,1},I_{j,2}$ and $I_{j,3}$ since the estimation of the last two terms is completely analogous  to the ones for $I_{j,4}, I_{j,5}$. 
Accordingly, let us denote the corresponding integrands by
$$\psi_j(x,y)=(\Phi(u_j(x){-}u_j(y))-\Phi(u(x){-}u(y)))(\varphi(x){-}\varphi(y))K(x,y) $$
and
$$
\widetilde \psi_j(x,y)=(\Phi(u_j(x)-g_j(y))-\Phi(u(x)-g(y))) \varphi(x) K(x,y)\;.
$$
These are both pointwise defined in the way above for $x \not = y$; we extend them to zero be on the diagonal $x=y$. We now prove that $I_{j,1}\to 0$; to this aim we start noticing that 
$\psi_j(x,y) \to 0$ a.e. with respect to the Lebesgue measure in $\er^{2n}$ by \rif{convergence}$_3$. It is then sufficient to show that the sequence
$\{\psi_j\}$
is equintegrable in $\Omega \times \Omega$. To do this we prove that there exists $\eps >0$ such that $\{\psi_j\}$ is equibounded in $L^{1+\eps}(\Omega \times \Omega)$ for some $\eps >0$. For this, using \rif{thekernel} and \rif{monophi}, observe that with $h <s$ and $\eps >0$ chosen in such a way that $\eps n +[p(s-h)+h-1](1+\eps)\leq 0$ and $q= (p-1)(1+\eps)< \bar q$, we have 
\begin{eqnarray*}
[\psi_j(x,y)] ^{1+\eps}& \leq  &\frac{c\left[(|u(x){-}u(y)|^{p-1}+|u_j(x){-}u_j(y)|^{p-1})|\varphi(x){-}\varphi(y)|\right]^{1+\eps}}{|x{-}y|^{(n+sp)(1+\eps)}}\\
& \leq  &\frac{c\left[(|u(x){-}u(y)|^{p-1}+|u_j(x){-}u_j(y)|^{p-1})\right]^{1+\eps}\|D\varphi\|_{L^\infty}^{1+\eps}}{|x{-}y|^{n+h(p-1)(1+\eps)}
|x{-}y|^{\eps n +[p(s-h)+h-1](1+\eps)}}\\
& \leq  &c(U_h)^{(p-1)(1+\eps)}(x,y)+c (U_h^j)^{(p-1)(1+\eps)}(x,y)\,,
\end{eqnarray*}
where $c$ depends also on $\|D\varphi\|_{L^\infty}$ and $\diam\, \Omega$. We conclude by \rif{bound2} that $\{\psi_j\}$ is equibounded in $L^{1+\eps}(\Omega \times \Omega)$ 
and therefore $I_{j,1}\to 0$. We now come to the proof of $I_{j,2}\to 0$ and $I_{j,4}\to 0$; observe that 
$\widetilde \psi_j(x,y)\to 0$ a.e. Again we prove that $\{\tilde \psi_j\}$ is uniformly bounded in $L^{1+\eps}(\Omega \times B_R \setminus \Omega)$ for some $\eps >0$. This time we estimate 
\eqn{estiesti}
$$
\tilde \psi_j(x,y)\leq  \frac{c\left[|u(x)|{+}|u_j(x)|{+}|g(y)|{+}|g_j(y)|\right]^{p-1}|\varphi(x)|}{|x{-}y|^{(n+sp)}}\;. 
$$ 
We then set $d_1 :=\dist \,({\rm supp}\, \varphi, \partial \Omega)>0$ and $d_2 := \diam \, \Omega$, and observe that there exists a constant $c$ depending 
only on $n,s,p$ and the ratio $d_2/d_1$ such that, for any $z \in {\rm supp}\, \varphi$, it holds that 
\eqn{disti}
$$
\frac{ |y-z|^{n+sp}}{|y-x|^{n+sp}}\leq c
$$
whenever $(x,y) \in ({\rm supp}\, \varphi)  \times B_R  \setminus \Omega$; moreover, we have that 
$|y-x|^{-n-sp}\leq d_1^{-n-sp}$. Using this last fact and \rif{estiesti} we deduce
\begin{eqnarray*}
&&\int_{B_R \setminus \Omega} \int_{\Omega}[\tilde \psi_j(x,y)]^{1+\eps}\, dx\, dy \\
&& \quad \leq  
c(d_1)\int_{\Omega} (|u|+|u_j|)^{(p-1)(1+\eps)}\, dx+c(d_1)\int_{B_R} (|g|+|g_j|)^{(p-1)(1+\eps)}\, dy\leq c \,,
\end{eqnarray*}
where the last inequality holds for a constant $c$ independent of $j$, by \rif{bound3}, provided we choose $\eps$ in order to satisfy $(p-1)(1+\eps)< \bar q$ and we are also using \rif{bound0}. We therefore conclude that $I_{j,2}\to 0$; in a completely similar way, essentially exchanging the roles of $x$ and $y$, we can also prove that $I_{j,4}\to 0$. We now estimate $I_{j,3}$ and $I_{j,5}$. More precisely 
prove that for every $\delta >0$ there exists $R$ such that 
\eqn{limmisuppi}
$$
\limsup_{j \to \infty}\, |I_{j,3}|\leq \delta\qquad \mbox{and}\qquad \limsup_{j \to \infty}\, |I_{j,5}|\leq \delta\;. 
$$
We prove the first inequality in the above display, the proof of the second being completely analogous. 
By using \rif{estiesti}-\rif{disti}, 
we have 
\begin{eqnarray*}
&&\int_{\ern \setminus B_R} \int_{\Omega}\tilde \psi_j(x,y)\, dx\, dy =   
\int_{\ern \setminus B_R}\int_{{\rm supp}\, \varphi}\tilde \psi_j(x,y)\, dx\, dy\\
&&\qquad\leq  c\int_{\ern \setminus B_R} \int_{\Omega}\frac{(|u(x)|^{p-1}+|u_j(x)|^{p-1})}{|y{-}z|^{(n+sp)}}\, dx\, dy
\\
&&\qquad \qquad +  c\int_{\ern \setminus B_R}\frac{(|g(y)|^{p-1}+|g_j(y)|^{p-1})}{|y{-}z|^{(n+sp)}}\, dx\, dy
\\
&&\qquad\leq \frac{c}{R^{sp}} \int_{\Omega}(|u|+|u_j|)^{p-1}\, dx+  c\int_{\ern \setminus B_R}
\frac{|g(y)|^{p-1}+|g_j(y)|^{p-1}}{|y{-}z|^{(n+sp)}}\, dx\, dy\;.
\end{eqnarray*}
Using this last estimate together with Lemma \ref{lemma:g_j} (see in particular the assertion concerning \rif{bound00}) we 
conclude that \rif{limmisuppi} holds provided $R\equiv R(\delta)$ is chosen sufficiently large. 
We are now ready to conclude the proof: take $\delta>0$ and determine $R>0$ such that the inequalities in \rif{limmisuppi} hold. 
Recalling \rif{tony} we infer
$$
\limsup_{j \to \infty}\, |I_{j}|\leq \limsup_{j \to \infty}\, |I_{j,3}|+ \limsup_{j \to \infty}\, |I_{j,5}|\leq 2\delta\;,
$$
so that eventually letting $\delta \to 0$ we get that $I_j \to 0$. In turn, using this information in \rif{lasty} and recalling that the weak convergence of $\mu_j$ to $\mu$ we infer \rif{veryweak} and the proof is finally complete.


\section{Proof of Theorems \ref{pot0} and \ref{pot}} \label{sec:upper}
Let $\{u_j\} \subset W^{s,p}(\ern)$ be an approximating sequence for the SOLA $u$ with measure $\mu_j$ and boundary values $g_j$, as described in Definition \ref{soladef}. For $\varrho \leq r$, we define the comparison solution $v_j \in W^{s,p}(\ern)$ as
$$
\left \{
\begin{array}{ccc}
-\widetilde{\mathcal{L}}_{u_j}v_j = 0&\mbox{in}&B_{\varrho/2}(x_0)\\[4 pt]
v_j = u_j&\mbox{in}& \ern \setminus B_{\varrho/2}(x_0)\;, 
\end{array} \right. 
$$
eventually letting $w_j = u_j-v_j$; the operator $- \widetilde{\mathcal L}_{u_j}$ has been defined in \rif{abbreviated}.  
In the following we shall omit to denote the dependence on $x_0$, thereby writing, for instance, $E(u;\varrho)\equiv E(u;x_0,\varrho)$, ${\rm Tail}(u;t)\equiv {\rm Tail}(u;x_0, t)$ and so on; we recall that $q_*=\max\{1,p-1\}$. 
Let us first collect a few preliminary estimates. Triangle inequality yields
\begin{eqnarray}  
\nonumber && \left| \left(\mean{B_t} |u_j {-}(u_j)_{B_t} |^{q_*} \, dx \right)^{1/q_*} -
\left(\mean{B_t} |v_j {-}(v_j)_{B_t} |^{q_*} \, dx \right)^{1/q_*}  \right| \\
&& \qquad \leq \left(\mean{B_{t}} |w_j|^{q_*} \, dx \right)^{1/q_*}+ |(u_j)_{B_t}-(v_j)_{B_t}|
\leq 2 \left(\mean{B_{t}} |w_j|^{q_*} \, dx \right)^{1/q_*}\label{triangletri}
\end{eqnarray}
for any $t>0$. 
Since $w_j$ vanishes outside of $B_{\varrho/2}$, we obtain
\begin{eqnarray}  
\nonumber
\left(\mean{B_{t}} |w_j|^{q_*} \, dx \right)^{1/q_*} & \leq &  \left( \frac{|B_{\varrho/2}|}{|B_t|} \right)^{1/q_*}
\left(\mean{B_{\varrho/2}} |w_j|^{q_*} \, dx \right)^{1/q_*} 
\\  & \leq & 
 c \left( \frac{\varrho}{t} \right)^{n/q_*} \left(\mean{B_{\varrho/2}} |w_j|^{q_*} \, dx \right)^{1/q_*} \,,\label{recall1}
\end{eqnarray}
again for any $t>0$. 
Now Lemma~\ref{lemma:comp est sol}, implies for $p\geq 2$ that 
\[
\left(\mean{B_{\varrho/2}} |w_j|^{p-1} \, dx \right)^{1/(p-1)}  \leq c \left[  \frac{|\mu_j|(B_\varrho)}{\varrho^{n-sp}} \right]^{1/(p-1)}\,.
\]
Instead, for $2 -s/n< p<2$, Lemma~\ref{lemma:comp p<2} and Young's inequality with conjugate exponents $(1/(2-p),1/(p-1))$, yield
\[
\mean{B_{\varrho/2}} |w_j| \, dx  \leq \delta E(u_j;\varrho) + c \delta^{(p-2)/(p-1)} \left[ \frac{|\mu_j|(B_{\varrho})}{\varrho^{n-sp}} \right]^{1/(p-1)}
\]
for any $\delta \in (0,1)$. 
The content of the last three displays then leads to
\eqn{recall2}
$$
\left(\mean{B_{t}} |w_j|^{q_*} \, dx \right)^{1/q_*}  \leq  c \left( \frac{\varrho}{t} \right)^{n/q_*} 
 \left\{ \delta E(u_j; \varrho) +  c \delta^{-\eta} \left[  \frac{|\mu_j|(B_\varrho)}{\varrho^{n-sp}} \right]^{1/(p-1)} \right\}
$$
for any $t>0$ and $\delta\in (0,1)$, where $\eta := \max\{0,(2-p)/(p-1)\}.$ All the constants $c$ involved up to now depend only on $n,s,p,\Lambda$. Our goal is now to use the above comparison estimate together with the oscillation reduction result in Theorem~\ref{lemma:osc red} for $v$. 
Again for any $t >0$, we have
$$
A(u_j;t)  \leq cA(v_j;t) + c\left(\mean{B_{t}} |w_j|^{q_*} \, dx \right)^{1/q_*}\,
$$
by \rif{triangletri}. 
As for the Tail, recalling also \rif{recall1} and that $q_* \geq p-1$, we have
\begin{eqnarray*} 
&& \nonumber \hspace{-1cm} {\rm Tail}(u_j-(u_j)_{B_{t }}; t) \leq c{\rm Tail}((u_j-v_j)_{B_{t }}; t)+c {\rm Tail}(v_j -(v_j)_{B_t};t) 
\\
& &\nonumber\qquad \qquad\qquad \qquad\qquad  + c \left( t^{sp}  \int_{ B_{\varrho/2}\setminus B_{t}} \frac{|u_j-v_j|^{p-1}} {|x{-}x_0|^{n+sp}} \, dx \right)^{1/(p-1)}
\\ \nonumber &  &\qquad  \leq c|(u_j-v_j)_{B_{t}}|+c E(v_j;t) +c \left( \frac{\varrho}{t} \right)^{n/(p-1)}
 \left( \mean{B_{\varrho/2}} |w_j|^{q_*} \, dx \right)^{1/q_*}\\
&&\qquad \leq c D(\varrho, t)
\left( \mean{B_{\varrho/2}} |w_j|^{q_*} \, dx \right)^{1/q_*}  +  
  c E(v_j;t)\;,
\end{eqnarray*}
where we have set
\eqn{ilD}
$$
D(\varrho, t):= \left[\left( \frac{\varrho}{t} \right)^{n/(p-1)} +\left( \frac{\varrho}{t} \right)^{n/q_*}\right]\;.
$$
By also using~\eqref{triangletri} and~\rif{recall2} with $t = \varrho/2$, it follows that 
$$
E(u_j;t)   \leq  c D(\varrho, t)   \left\{ \delta E(u_j;\varrho) +  c \delta^{-\eta} \left[  \frac{|\mu_j|(B_\varrho)}{\varrho^{n-sp}} \right]^{1/(p-1)} \right\}+  c E(v_j;t)
$$
 and for any $t>0$.
In a completely analogous way we also obtain
$$
\nonumber E(v_j;t)   \leq   c D(\varrho, t)  \left\{ \delta E(u_j;\varrho) +  c \delta^{-\eta} \left[\frac{|\mu_j|(B_\varrho)}{\varrho^{n-sp}} \right]^{1/(p-1)} \right\}+  c E(u_j;t)\;.
$$
For any $\sigma \in (0,1/2)$, recall the estimate from Theorem~\ref{lemma:osc red}:
$$
E(v_j;\sigma \varrho)  \leq  c   \sigma^\alpha \left( \frac{\varrho}{r} \right)^{sp/q_*} E(v_j;r) +c  \sigma^\alpha \int_\varrho^{r}  \left(  \frac{\varrho}{t}\right)^{sp/q_*} E(v_j;t)  \, \frac{dt}{t}
$$
that holds whenever $ \varrho\leq r\leq R$. 
Combining the last three displays, and recalling that $ D(\varrho, \sigma \varrho) =D(1, \sigma ) \geq 1$, leads to
\begin{eqnarray*} 
\nonumber
E(u_j;\sigma \varrho) & \leq &  c  \sigma^\alpha \left( \frac{\varrho}{r} \right)^{sp/q_*} E(u_j;r)  +c   \sigma^\alpha \int_\varrho^{r}  \left(  \frac{\varrho}{t}\right)^{sp/q_*} E(u_j;t)  \, \frac{dt}{t}
\\ &&\quad  \qquad + c D(1, \sigma )   \left\{ \delta E(u_j;\varrho) +  c \delta^{-\eta} \left[  \frac{|\mu_j|(B_\varrho)}{\varrho^{n-sp}} \right]^{1/(p-1)} \right\}\,. 
\end{eqnarray*}
By the fact that $u_j \to u$ in $L_{\rm loc}^{q_*}(\ern)$, $g_j \to g$ as in~\eqref{eq:g_j conv}, and weak convergence of $\mu_j$ to $\mu$ as $j \to \infty$, that is \rif{eq:meas conv cond}, we obtain that
\begin{eqnarray} \label{eq:decay comp appl}
\nonumber
E(u;\sigma \varrho) & \leq &  c  \sigma^\alpha \left( \frac{\varrho}{r} \right)^{sp/q_*} E(u;r)  + c   \sigma^\alpha \int_\varrho^{r}  \left(  \frac{\varrho}{t}\right)^{sp/q_*} E(u;t)  \, \frac{dt}{t}
\\ && \quad  \qquad + c D(1, \sigma )   \left\{ \delta E(u;\varrho) +  c \delta^{-\eta} \left[  \frac{|\mu|(\overline{B_\varrho})}{\varrho^{n-sp}} \right]^{1/(p-1)} \right\}
\end{eqnarray}
holds for the original SOLA $u$ considered in the theorem. Next, we multiply the previous estimate by $\varrho^{-1}$, 
and integrate the resulting inequality on $[\varrho, r]$. Renaming the variables once again yields
\begin{eqnarray} \label{eq:iter mono} 
\nonumber\int_\varrho^r E(u;\sigma t) \, \frac{dt}{t }  
&\leq &    c\sigma^\alpha  \int_\varrho^r \left( \frac{t}{r} \right)^{sp/q_*}   \, \frac{dt}{t } \,E(u;r)
\\ \nonumber  && + c \sigma^\alpha \int_\varrho^r  t^{sp/q_*} \int_t^r E(u;\tau) \, \frac{d\tau}{\tau^{1+sp/q_*}} \, \frac{dt}{t }
\\ && + c\delta D(1, \sigma)    \int_\varrho^r E(u;t) \, \frac{dt}{t }   +  \frac{cD(1, \sigma)}{\delta^{\eta}}  \int_\varrho^r \left[  \frac{|\mu|(\overline{B_t})}{t^{n-sp}} \right]^{1/(p-1)}  \, \frac{dt}{t }.
\end{eqnarray}
For the second term on the right hand side, integration by parts yields 
\begin{eqnarray} \  
\nonumber && \int_\varrho^r  t^{sp/q_*} \int_t^r E(u;\tau) \, \frac{d\tau}{\tau^{1+sp/q_*}} \, \frac{dt}{t }  
\\ \nonumber  && \qquad = 
\frac{t^{ sp/q_*}}{ sp/q_*} 
\int_t^r E(u;\tau) \, \frac{d\tau}{\tau^{1+sp/q_*}} \bigg|_{t=\varrho}^r
+ \frac{1}{ sp/q_*}  \int_\varrho^r E(u;t) \, \frac{dt}{t }  
\\ \nonumber  && \qquad \leq  \frac{2}{ sp/q_*}  \int_\varrho^r E(u;t) \, \frac{dt}{t }    \,.
\end{eqnarray}
Moreover, we notice that
\[
\int_\varrho^r \left( \frac{t}{r} \right)^{sp/q_*}   \, \frac{dt}{t } \leq \frac{1}{ sp/q_*} \,.
\]
Therefore,~\eqref{eq:iter mono} and the content of the two last displays imply also
\begin{eqnarray*}
\nonumber\int_\varrho^r E(u;\sigma t) \, \frac{dt}{t }  
&\leq &  c\sigma^\alpha  E(u;r) +
c\left(\sigma^\alpha + \delta D(1, \sigma)\right)\int_\varrho^r E(u;t) \, \frac{dt}{t } 
\\  &&\qquad\qquad\quad   +  \frac{cD(1, \sigma)}{\delta^{\eta}}  \int_\varrho^r \left[  \frac{|\mu|(\overline{B_t})}{t^{n-sp}} \right]^{1/(p-1)}  \, \frac{dt}{t }  
\end{eqnarray*}
where $c \equiv c(n,s,p,\Lambda)$. Next, changing variables gives
\begin{eqnarray} \  \nonumber
\int_\varrho^r E(u;\sigma t) \, \frac{dt}{t }   & = &  \int_{\sigma \varrho}^{\sigma r} E(u;t)  \, \frac{dt}{t } 
\\ \nonumber  & = & \int_{\sigma \varrho}^{r} E(u;t)  \, \frac{dt}{t } - \int_{\sigma r}^{r} E(u;t)  \, \frac{dt}{t }
\\ \nonumber  & \geq & \int_{\sigma \varrho}^{r} E(u;t) \, \frac{dt}{t } - c(\sigma) E(u;r) \int_{\sigma r}^{r}   \, \frac{dt}{t }
\\ \nonumber  & \equiv &  \int_{\sigma \varrho}^{r} E(u;t) \, \frac{dt}{t } - c(\sigma)  E(u;r) \,,
\end{eqnarray}
where the second-last inequality follows by Lemma~\ref{lemma:basic E}(4). 
Matching the last two inequalities leads to
\begin{eqnarray} \  
\nonumber
\int_{\sigma \varrho}^r E(u;t) \, \frac{dt}{t }   & \leq & 
\bar c(\sigma^{\alpha} + \delta D(1, \sigma) )\int_\varrho^r E(u;t) \, \frac{dt}{t } +   c_*(\sigma)  E(u;r) 
\\  \nonumber && \qquad \ + \frac{c_*(\sigma)}{\delta^\eta}   \int_\varrho^r \left[  \frac{|\mu|(\overline{B_t})}{t^{n-sp}} \right]^{1/(p-1)}  \, \frac{dt}{t }  \,,
\end{eqnarray}
where $\bar c\equiv \bar c (n,s,p,\Lambda)$ is independent of $\sigma$, while $c_*(\sigma)$ blows-up when $\sigma \to 0$. 
Choosing first $\sigma$ and then $\delta$ small enough in order to have 
\[
\bar c \sigma^{\alpha}  = \frac14 \qquad \mbox{and} \qquad 
\bar c  \delta D(1, \sigma)   = \frac14\,,
\]
we obtain, after reabsorption and eventually letting $\varrho \to 0$, that 
\eqn{eq:pot est prel 0} 
$$
\int_{0}^r E(u;t) \, \frac{dt}{t }    \leq  c  E(u;r)  + c \int_0^r \left[  \frac{|\mu|(B_t)}{t^{n -ps}} \right]^{1/(p-1)} 
 \, \frac{dt}{t} 
$$
holds for a constant $c$ depends only on $n,s,p,\Lambda$. 
This proves the part of estimate \rif{estipot0} concerning the first term on the left hand side. Notice that we have used that 
\eqn{keep in mind}
$$
\int_0^r \left[  \frac{|\mu|(\overline{B_t})}{t^{n -ps}} \right]^{1/(p-1)}\, \frac{dt}{t}  =\int_0^r \left[  \frac{|\mu|(B_t)}{t^{n -ps}} \right]^{1/(p-1)}\, \frac{dt}{t} \;.
$$
We proceed to prove that the limit in \rif{precise} exists and to complete the proof of estimate \rif{estipot0}. For this, let $0< \tilde \varrho  \leq  \varrho/2 < r/8$ and find $k \in \en$ and $\theta \in (1/4,1/2]$ such that $\tilde \varrho = \theta^k \varrho$. Then
\begin{eqnarray*}
|(u)_{B_\varrho} - (u)_{B_{\tilde \varrho}}| &\leq& \sum_{j=0}^{k-1} |(u)_{B_{\theta^j \varrho}}-(u)_{B_{\theta^{j+1} \varrho}}|\\
& \leq & \theta^{-n/q_*}\sum_{j=0}^{k-1} A(u;\theta^j \varrho)\leq \theta^{-n/q_*}\sum_{j=0}^{k-1} E(u;\theta^j \varrho)\,.
\end{eqnarray*}
Furthermore, using Lemma~\ref{lemma:basic E}(2), we have that 
\begin{eqnarray*}
\sum_{j=0}^{k-1} E(u;\theta^j \varrho) &= &\frac{1}{\log (1/\theta)} \sum_{j=0}^{k-1} 
\int_{\theta^j \varrho}^{\theta^{j-1} \varrho}E(u;\theta^j \varrho) \, \frac{dt}{t} 
\\
 &\leq &c \sum_{j=0}^{k-1} 
\int_{\theta^j \varrho}^{\theta^{j-1} \varrho}E(u;t) \, \frac{dt}{t} 
 \leq c 
\int_{\tilde  \varrho}^{ \varrho/\theta}E(u;t) \, \frac{dt}{t} \;, 
\end{eqnarray*}
so that 
, using the content of the last three displays and recalling that $\varrho \leq r/4\leq r/\theta$, it follows
\eqn{recall3}
$$
|(u)_{B_\varrho} {-} (u)_{B_{\tilde \varrho}}|  \leq c \int_{\tilde  \varrho}^{\varrho/\theta} E(u;t) \, \frac{dt}{t}  \;.
$$
In turn, recalling that $\varrho/\theta \leq r/(4\theta)\leq r$ and using \rif{eq:pot est prel 0} we have 
\eqn{recall4}
$$
|(u)_{B_\varrho} {-} (u)_{B_{\tilde \varrho}}|  \leq  c  E(u;r)  + c {\bf W}_{s,p}^\mu(x_0,r) \;.
$$
On the other hand, by~\eqref{eq:pot est prel 0} the finiteness of $
{\bf W}_{s,p}^\mu(x_0,r)
$
implies the finiteness of the right hand side in \rif{recall4} and therefore \rif{recall3} readily implies that $\{(u)_{B_\varrho}\}$ is a Cauchy net. As a consequence, the limit in \rif{precise} exists and thereby defines the pointwise precise representative of $u$ at $x_0$. Letting $\tilde \varrho \to 0$ in \rif{recall4} and taking $\varrho =r/4$ then gives 
$$
|(u)_{B_{r/4}} {-} u(x_0)|  \leq  c  E(u;r)  + c {\bf W}_{s,p}^\mu(x_0,r) \;.
$$
that holds whenever $\varrho \leq r/4$. On the other hand notice that by also using \rif{tt2} we have
$$
|(u)_{B_r} {-} (u)_{B_\varrho}| \leq 4^n A(u;r) \leq c  E(u;r) 
$$
so that the last two displays and triangle inequality finally give \rif{estipot0} (recall also \rif{eq:pot est prel 0}). This completes the proof of Theorem \ref{pot}. Finally, estimate \rif{stimawolff1} follows from \rif{estipot0} via elementary manipulations. 
\section{Proof of Theorem~\ref{thm:cont}}
The proof goes in two steps: first we prove that $u$ is locally VMO-regular in $\Omega'$, this means that if $\Omega'' \Subset \Omega'$, then 
\eqn{VMOreg}
$$
\lim_{t\to 0} \, E(u; x, t)=0
$$
holds uniformly in $x \in \Omega''$. Once this information is available, the continuity of $u$ follows easily from estimate 
\rif{estipot0} and the fact that $\Omega''$ is arbitrary.
Indeed, for fixed $r \leq \dist(\partial \Omega'', \Omega')/100$, consider the continuous function defined by $\Omega'' \ni x \mapsto (u)_{B_r(x)}$. From \rif{estipot0}, the assumption 
\rif{decay} and \rif{VMOreg} it follows that the net of functions $\{(u)_{B_r(x)}\}_r$ converges to $u(x)$ uniformly in 
$\Omega''$ and therefore the continuity of $u$ is proved. Notice that, by Theorem \ref{pot0} and the assumptions of Theorem \ref{thm:cont}, we already know that every point $x$ is a Lebesgue point for $u$. It remains to prove \rif{VMOreg}. First of all, we observe that Theorem \ref{pot0} implies that $u$ is locally bounded in $\Omega$; in particular, we have that $\|u\|_{L^{\infty}(\Omega')}<\infty$. 
Moreover, from the definition of Wolff potentials, assumption \rif{decay} in particular implies that the limit 
\eqn{unimu}
$$
\lim_{t\to 0} \,\frac{|\mu|(B_t(x))}{t^{n-sp}}=0
$$
is uniform with respect to $x \in \Omega'$. Take a ball $\mathcal B\equiv B_R(z)$ such that $ \Omega \subset \mathcal B$ and  $z \in \Omega'$. Notice that the inequality 
$
|y-x|^{-n-sp} \leq c |y-z|^{-n-sp}
$
holds whenever $x \in \Omega' $, $y \in \ern\setminus \mathcal B$, for a constant $c$ depending only on $n,s,p,$ and 
$\dist(\Omega'', \partial \mathcal B)$, and a fortiori on $\dist(\Omega'', \partial \Omega')$. 
Using also this last fact, with $r\leq \dist(\partial \Omega'', \Omega')/100$ and $x \in \Omega''$ we have
\begin{eqnarray*}
E(u;x, r) &\leq& A(u;x, r)+ c \|u\|_{L^{\infty}(B_r(x))}+  c{\rm Tail}(u;x,r)\\
&\leq &c\|u\|_{L^{\infty}(\Omega')}+ c\left( r^{sp} \int_{\ern\setminus \mathcal B} \frac{|g(y)|^{p-1}}{ |x{-}y|^{n+sp}}\,d y \right)^{1/(p-1)}\\
&& \qquad \quad + c\left( r^{sp} \int_{ \mathcal B\setminus \Omega'} \frac{|u(y)|^{p-1}}{ |x{-}y|^{n+sp}}\,d y \right)^{1/(p-1)}\\
&& \qquad \quad + c\left( r^{sp} \int_{\Omega'  \setminus B_r(x)} \frac{|u(y)|^{p-1}}{ |x{-}y|^{n+sp}}\,d y \right)^{1/(p-1)}\\
&\leq &c\|u\|_{L^{\infty}(\Omega')}+ c\left( r^{sp} \int_{\ern\setminus \mathcal B}\frac{ |g(y)|^{p-1}}{ |y{-}z|^{n+sp}}\,d y \right)^{1/(p-1)}\\
&& \qquad \quad + \frac{c}{\dist(\partial \Omega'', \Omega')^{\frac{n}{p-1}}}
\left(\int_{\mathcal B} |u(y)|^{p-1}\,d y \right)^{1/(p-1)}\\
&& \qquad \quad  + c\left( r^{sp} \int_{\ern  \setminus B_r(x)} \frac{\|u\|_{L^{\infty}(\Omega')}^{p-1}}{ |x{-}y|^{n+sp}}\,d y \right)^{1/(p-1)}\\
&\leq &c\|u\|_{L^{\infty}(\Omega')}+ c\left(\frac{r}{R}\right)^{sp}{\rm Tail}(g;z,R)\\
&& \qquad \quad + \frac{c}{\dist(\partial \Omega'', \Omega')^{\frac{n}{p-1}}}
\left(\int_{\mathcal B} |u(y)|^{p-1}\,d y \right)^{1/(p-1)}:= cH\;.
\end{eqnarray*}
Notice that $H$ is independent of $x$; indeed all the subsequent estimates will be uniform with respect to $x \in \Omega''$. 
At this point \rif{eq:decay comp appl} for $\varrho=r\leq \dist(\partial \Omega'', \Omega')/100$ yields
$$
E(u;x,\sigma r)  \leq   c  \sigma^\alpha H + c D(1, \sigma )   \left\{ \delta H +
  c \delta^{-\eta} \left[  \frac{|\mu|(B_r)}{r^{n-sp}} \right]^{1/(p-1)} \right\}
$$
whenever $x \in \Omega''$ and where $D(1, \sigma )$ has been defined in \rif{ilD}. With $\eps >0$ being fixed, we now first choose $\tilde \sigma$ and then $\delta$ in order to satisfy   
$$
c  \sigma^\alpha H \leq \frac{\eps}{4}\qquad  \mbox{and} \qquad c D(1, \sigma )  \delta H\leq \frac{\eps}{4}
$$
as soon as $\sigma \leq \tilde \sigma$; then, using \rif{unimu} we choose $\tilde r_\eps$ such that 
$$
c D(1, \sigma )\delta^{-\eta} \left[  \frac{|\mu|(B_r)}{r^{n-sp}} \right]^{1/(p-1)}\leq \frac{\eps}{4}\qquad \mbox{for every}
 \quad r \leq \tilde r_\eps\;.$$
 All in all, we have proved that for every $\eps>0$ there exists a radius $r_\eps= \tilde \sigma \tilde r_\eps$, depending also on $n,s,p,\Lambda, \dist(\partial \Omega'', \Omega')$ and $H$, but not on the point $x\in \Omega''$, such that $r \leq r_\eps$ implies $
E(u;x, r)  \leq \eps. $ This means that \rif{VMOreg} holds and the proof is complete. 
\section{Proof of Theorem \ref{thm:lower}}
Throughout the section $\mu \in \mathcal M(\ern)$ is considered to be a non-negative measure and we are assuming also that $p < n/s$. Moreover, as usual and with no loss of generality, we shall assume throughout Lemmas \ref{lemma:cacc super}-\ref{ultimolemma} that the additional conditions in \rif{specialp}. This is possible by re-writing $- \mathcal{L}_{\Phi} u \geq 0 \, (=\mu)$ as $\widetilde{\mathcal{L}}_u u\geq0 \, (=\mu)$ and the operator $\widetilde{\mathcal{L}}_u$ has been defined in \rif{abbreviated}. 
Once again, this will allows us to use the results from \cite{DKP2}. 
In fact, let us first recall two results from~\cite{DKP2}, that hold in the full range $p>1$.  
\begin{lemma}[Caccioppoli estimate \cite{DKP2}]\label{lemma:cacc super}
Let $p\in(1,\infty)$, $q \in (1,p)$, $d>0$ and let $u$ be a weak supersolution to \trif{eqvw} with $\mu=0$, such that $u\geq 0$ in $B_R(x_0)\subset\Omega$. Then, for 
$
w:= (u+d)^{1-q/p}
$
the following inequality:
\begin{eqnarray*}
&&\hspace{-0.5cm} \int_{B_r}\int_{B_r} \frac{|w(x)\varphi(x){-}w(y)\varphi(y)|^p}{|x{-}y|^{n+sp}}\, dx\, dy \nonumber\\
&&\qquad \quad \leq c\int_{B_r}\int_{B_r} (\max\{w(x),w(y)\})^p \frac{|\varphi(x){-}\varphi(y)|^p}{|x{-}y|^{n+sp}} \, dx\, dy \\
&& \qquad \quad \quad +c \,\bigg\{ \sup_{y\in \text{\rm supp}\, \varphi} \int_{\ern \setminus B_r}|x{-}y|^{-n-sp}\, dx \nonumber
\\ &&\qquad \qquad \quad\qquad \qquad  +  d^{1-p}R^{-sp}
\big[\text{\rm Tail}(u_-; x_0, R)\big]^{p-1} \bigg\}\left(\int_{B_r} (w\varphi)^p\, dx \right) 
\nonumber
\end{eqnarray*} 
holds for $B_r\equiv B_r(x_0)\subset B_{3R/4}(x_0)$ and any nonnegative $\varphi \in C^\infty_0(B_r)$, where the constant $c$ depends only on $s,p,\Lambda,q$ and 
$u_-:=\max\{-u,0\}$. 
\end{lemma}  
\begin{theorem}[Weak Harnack inequality \cite{DKP2}] \label{thm:weak Harnack}
For any $s\in (0,1)$ and any $p \in (1,n/s)$, let $u$ be a weak supersolution to \trif{eqvw} with $\mu=0$ such that $u\geq 0$ in $B_R\equiv B_R({x_0})\subset\Omega$. Then the following estimate holds for any $B_{r}\equiv B_{r}({x_0})\subset B_{R/2}(x_0)$ and for any positive number $\gamma< n(p-1)/(n-sp)$:
$$
\left( \mean{B_r} u^\gamma \right)^{1/\gamma} \leq
c \inf_{B_{2r}} u 
+ c \left(\frac{r}{R}\right)^{\frac{sp}{p-1}} \text{\rm Tail}(u_-; {x_0},R),
$$
where the constant $c$ depends only on $n,s,p, \Lambda$. 
\end{theorem}
\begin{lemma} \label{lemma:lower 1} Let $u$ be a weak solution to \trif{eqvw} with $\mu \in C^{\infty}_0(\er^n)$, such that $u \geq 0$ in $B_{4r}\equiv B_{4r}(x_0)\subset \Omega$ and $\mu \geq 0$.
 Then there exists a constant $c\equiv c(n,s,p,\Lambda)$ such that the following inequality holds:
 \begin{eqnarray}\label{eq:lower 1}
\nonumber
\frac{\mu(B_r)}{r^{n-sp}} & \leq &
 \frac{c}{r^{1-sp}} \int_{B_{3r/2}} \mean{B_{3r/2}} \frac{|u(x){-}u(y)|^{p-1}}{|x{-}y|^{n+sp-1}}   \, dx \, dy
\\  & &  \qquad \qquad + c \left[ \inf_{B_{r}} u+ \text{\rm Tail}(u _-; {x_0},4r) \right]^{p-1}\;.
\end{eqnarray}
\end{lemma}
\begin{proof}
Let $\varphi \in C_0^\infty(B_{5r/4})$ be such that $0\leq \varphi \leq 1$, $\varphi = 1$ in $B_{r}$ and $\|D\varphi\|_{L^\infty} \leq 16/r$. 
Taking such $\varphi$ in \rif{veryweak}, using \rif{thekernel} and recalling that \rif{specialp} holds, we get
\begin{eqnarray} \  
\nonumber
&&\frac{\mu(B_r)}{r^{n-sp}} \\& & \nonumber   \leq \frac{1}{r^{n-sp}} \int_{\ern} \int_{\ern}|u(x){-}u(y)|^{p-2}( u(x){-}u(y)) 
(\varphi(x) - \varphi(y)) 
K(x,y) \, dx \, dy
\\ \nonumber &  & =|B_1| r^{sp} \int_{B_{3r/2}} \mean{B_{3r/2}}
 |u(x){-}u(y)|^{p-2}( u(x){-}u(y)) (\varphi(x) - \varphi(y)) 
K(x,y) \, dx \, dy
\\ \nonumber & & \qquad+ |B_1| r^{sp} \int_{\ern \setminus B_{3r/2}} \mean{B_{3r/2}} |u(x){-}u(y)|^{p-2}(u(x){-}u(y))\varphi(x) 
K(x,y)  \, dx \, dy
\\ \nonumber & & \qquad+ |B_1| r^{sp} \mean{B_{3r/2}} \int_{\ern \setminus B_{3r/2}} |u(y){-}u(x)|^{p-2}(u(y){-}u(x))\varphi(y) 
K(y,x)  \, dx \, dy
\\ \nonumber &  & \leq c \|D\varphi\|_{L^\infty} r^{sp}   \int_{B_{3r/2}} \mean{B_{3r/2}} \frac{|u(x){-}u(y)|^{p-1}}{|x{-}y|^{n+sp-1}}   \, dx \, dy
\\  & & \quad+ 2|B_1| r^{sp} \int_{\ern \setminus B_{3r/2}} \mean{B_{3r/2}} |u(x){-}u(y)|^{p-2}(u(x){-}u(y))\varphi(x) 
K(x,y)  \, dx \, dy\,.\label{conto}
\end{eqnarray}
We turn to the estimate of the last integral; 
we shall also use the fact that $|y-x_0|\leq 16 |x-y|$ whenever 
$y \in \ern \setminus B_{3r/2}$ and $x \in {\rm supp} \,\varphi$ (since ${\rm supp}\, \varphi\subset B_{5r/4}$). 
Recalling that $u \geq 0$ in $B_{4r}$, we start splitting as follows:
\begin{eqnarray*} \  
\nonumber &&
 r^{sp} \int_{\ern \setminus B_{3r/2}} \mean{B_{3r/2}} 
|u(x){-}u(y)|^{p-2}(u(x){-}u(y))\varphi(x)
K(x,y)   \, dx \, dy
\\ \nonumber &&
\quad =  r^{sp} \int_{\ern \setminus B_{3r/2}\cap\{u(y)\geq 0\}} \mean{B_{3r/2}} (\ldots)\, dx \, dy\\ && \qquad +   r^{sp} \int_{\ern \setminus B_{4r}\cap\{u(y)< 0\}} \mean{B_{3r/2}}(\ldots)   
\, dx \, dy =: S_1 + S_2.
\end{eqnarray*}
Recalling that $0 \leq \varphi (x)\leq 1$ and \rif{thekernel}, we estimate $S_2$ as follows:
\begin{eqnarray*} \  
  S_2 &\leq & c r^{sp} \int_{\ern \setminus B_{4r}\cap\{u(y)< 0\}} \mean{B_{3r/2}} \frac{|u(x)-u(y)|^{p-1}}
  {|y{-}x_0|^{n+sp}}   
\, dx \, dy \\&\leq & c r^{sp} \int_{\ern \setminus B_{4r}} \mean{B_{3r/2}} \frac{[u(x)]^{p-1} +[u_{-}(y)]^{p-1}}
  {|y{-}x_0|^{n+sp}}   
\, dx \, dy \\
&\leq & c  \mean{B_{3r/2}} u^{p-1} \, dx + c\left[ \text{\rm Tail}(u_-; {x_0},4r)\right]^{p-1}\;.
\end{eqnarray*}
As for $S_1$, we observe that when evaluating the double integral we can restrict the domain of integration to $(\ern \setminus B_{3r/2})\times B_{3r/2}\cap \{0\leq u(y)\leq u(x)\}$ (since on the remain part of the integration domain the integrand 
is negative); on this we can simply estimate $||u(x){-}u(y)|^{p-2}(u(x){-}u(y))|\leq c[u(x)]^{p-1}$. We thereby obtain again 
$$
  S_1 \leq  cr^{sp} \int_{\ern \setminus B_{3r/2}} \mean{B_{3r/2}} \frac{[u(x)]^{p-1}}
  {|y{-}x_0|^{n+sp}}   
\, dx \, dy\leq c  \mean{B_{3r/2}} u^{p-1} \, dx\;.
$$
Connecting the estimates found for $S_1$ and $S_2$ to \rif{conto} we conclude with 
\begin{eqnarray} \  
\nonumber
\frac{\mu(B_r)}{r^{n-sp}} &\leq & \frac{c}{r^{1-sp}}    \int_{B_{3r/2}} \mean{B_{3r/2}} \frac{|u(x){-}u(y)|^{p-1}}{|x{-}y|^{n+sp-1}}   \, dx \, dy
\\ \nonumber & & \qquad+ c  \mean{B_{3r/2}} u^{p-1} \, dx + c [\text{\rm Tail}(u_-; {x_0},4r)]^{p-1}\,.
\end{eqnarray}
Since $\mu$ is non-negative, then $u$ is a non-negative weak supersolution to $-\mathcal{L}_{\Phi}u =0$ in $B_{4r}$; at 
this point a suitable application of Theorem~\ref{thm:weak Harnack} finishes the proof.
\end{proof}
We then estimate the integral appearing in \rif{eq:lower 1}. 
\begin{lemma} \label{lemma:lower 2}
Let $u$ be a weak solution to \trif{eqvw} with $\mu \in C^{\infty}_0(\er^n)$, such that $u \geq 0$ in $B_{4r}\equiv B_{4r}(x_0)\subset \Omega$ and $\mu \geq 0$.  
Let $h \in (0,s)$, $q \in (0,\bar q)$, where $\bar q$ has been defined in \trif{condizioni}. Then there exists a constant $c \equiv c(n,s,p,\Lambda,s-h,\bar q - q)$ such that 
\begin{equation} \label{eq:lower 2}
\left(\int_{B_{2r}} \mean{B_{2r}} \frac{|u(x){-}u(y)|^{q}}{|x{-}y|^{n+hq}} \, dx \, dy \right)^{1/q} \leq \frac{c}{ r^{h}}  \left[\inf_{B_{r}}  u + \text{\rm Tail}(u_-; x_0, 4r)\right]
\end{equation}
holds.
\end{lemma}

\begin{proof}
Let 
\eqn{defid2}
$$
d \equiv d_\delta := \inf_{B_{r}}  u + \text{\rm Tail}(u_-; x_0, 4r) + \delta\,,\qquad \mbox{for}\ \delta >0\;,$$
and set $$\bar u = u +d\,,\qquad w:= \bar u^{1-m/p}\,, \qquad\mbox{where}\ m \in (1,p)\;.$$
Applying Lemma~\ref{lemma:cacc super} (with $2r$ instead of $r$ and $R \equiv 4r$) with a cut-off function $\varphi \in C_0^{\infty} (B_{7r/4})$ such that $\varphi = 1$ in $B_{3r/2}$, $0\leq \varphi \leq 1$ and $|D\varphi| \leq 16/r$, we obtain
$$
\int_{B_{3r/2}} \int_{B_{3r/2}}  \frac{|w(x){-}w(y)|^p}{|x{-}y|^{n+sp}}\, dx\, dy  \leq \frac{c}{r^{sp}}  \int_{B_{2r}} w^p \, dx \;.
$$
To evaluate the integral in the left-hand side, we start assuming, without loss of generality, $u(x) > u(y)$; we have
\begin{eqnarray} \  
\nonumber
|w(x){-}w(y)| & = & |[\bar u(x)]^{1-m/p}{-}[\bar u(y)]^{1-m/p} |
\\ \nonumber & = &[ \bar u(x)]^{1-m/p} \left[1- \left(\frac{\bar u(y)}{\bar u(x)} \right)^{1-m/p} \right] \geq  \frac{p-m}{p} \frac{\bar u(x) - \bar u(y)}{[\bar u(x)]^{m/p}}\;.
\end{eqnarray}
Notice that in the last lines we have used the elementary inequality
$
(1-t^{\beta})\geq  \beta (1-t), 
$
that holds in the case $t, \beta \in (0,1]$ and that follows by mean value 
theorem. Considering in a similar way also the case $u(y) > u(x)$ and recalling the definition of $\bar u$, we conclude that
$$
\frac{|u(x){-}u(y)|}{[\bar u(x)]^{m/p}{+}[\bar u(y)]^{m/p}} \leq
 \frac{cp}{p-m}|w(x){-}w(y)|\,.
$$
On the other hand, recalling again the definitions of $w$ and $d$, Theorem~\ref{thm:weak Harnack} then implies
\[
\mean{B_{2r}} w^p \, dx  \leq c d^{p-m}
\]
since $m \in (1,p)$, where $c\equiv (n,s,p, \Lambda)$.  
Putting the last three displays above together yields, and using that $u\geq 0$ in $B_{4r}$, we have
\begin{equation} \  \nonumber
\int_{B_{3r/2}} \mean{B_{3r/2}} \frac{|u(x){-}u(y)|^p }{[\bar u(y) {+}\bar u(x)]^{m}} \, \frac{dx\, dy}{|x{-}y|^{n+sp}}
\leq \frac{cd^{p-m}}{r^{sp }}\,.
\end{equation}
Next, we get by H\"older's inequality that 
\begin{eqnarray} 
\nonumber && \int_{B_{3r/2}} \mean{B_{3r/2}} \frac{|u(x){-}u(y)|^{q}}{|x{-}y|^{n+h q}} \, d x \, dy  
\\ \nonumber && \qquad \leq \; \left(\int_{B_{3r/2}} \mean{B_{3r/2}}  \frac{|u(x){-}u(y)|^p }{[\bar u(y){+} \bar u(x)]^{m}} \, \frac{dx\, dy}{|x{-}y|^{n+sp}}   \right)^{q/p}  
\\ \nonumber && \qquad \qquad \; \cdot  
\left(\int_{B_{3r/2}} \mean{B_{3r/2}}  \frac{[\bar u(y) {+} \bar u(x)]^{mq/(p-q)}\, dx\, dy}{|x{-}y|^{n+(h-s)qp/(p-q)}}   \right)^{(p-q)/p} \,.
\end{eqnarray}
Using the definition of $d$ in \rif{defid2}, Theorem~\ref{thm:weak Harnack} then gives 
\begin{eqnarray*} 
\nonumber && \left(\int_{B_{3r/2}} \mean{B_{3r/2}}  \frac{[\bar u(y) {+} \bar u(x)]^{mq/(p-q)}\, dx\, dy}{|x{-}y|^{n+(h-s)qp/(p-q)}}   \right)^{(p-q)/p}
\\ &&\qquad  \leq \frac{c}{r^{(h-s)q}}\left( \mean{B_{3r/2}}  \bar u^{mq/(p-q)}\, dx\right)^{(p-q)/p}\leq c \left( r^{(s-h)p} d^{m} \right)^{q/p}
\end{eqnarray*}
with the last inequalities that hold provided
\[
\frac{mq}{p-q} < \frac{n(p-1)}{n-sp} \qquad \mbox{and} \qquad h < s\,.
\]
The first inequality in the above display is required in order to apply Theorem~\ref{thm:weak Harnack}. Notice that we can find $m >1$ satisfying the previous condition observing that
$$
\frac{q}{p-q} \; < \; \frac{n(p-1)}{n-sp}  \quad \Longleftrightarrow \quad q  \; < \; \frac{n(p-1)}{n-s}\;.
$$
We notice that the constant $c$ depends where the constant depends on $n,s,p,\Lambda,s-h,\bar q -q$. Hence we arrive at
$$
\mean{B_{3r/2}} \frac{|u(x){-}u(y)|^{q}}{|x{-}y|^{n+h q}} \, d x \, dy\leq   \frac{cd^{q}}{r^{hq}}\;.
$$
This finishes the proof letting $\delta\to 0$ and recalling the definition of $d$ in \rif{defid2}.
\end{proof}
\begin{lemma} \label{ultimolemma}
Let $u$ be a weak solution to \trif{eqvw} with $\mu \in C^{\infty}_0(\er^n)$, such that $u \geq 0$ in $B_{4r}\equiv B_{4r}(x_0)\subset \Omega$ and $\mu \geq 0$. Then there exists a constant $c \equiv c(n,s,p,\Lambda)$ such that the following inequality holds:
$$
\frac{\mu(B_r)}{r^{n-sp}} \leq c \left[ \inf_{B_{r}} u+ \text{\rm Tail}(u _-; {x_0},4r) \right]^{p-1}\;.
$$
\end{lemma}
\begin{proof} This follows using Lemmas~\ref{lemma:lower 1} and~\ref{lemma:lower 2} with with $q=p-1$ and $h$ such that 
$s> h>  (sp-1)/(p-1)$ (that implies $(sp-1)-h(p-1)\leq 0$), and estimating, by mean of \rif{eq:lower 2}, as
\begin{eqnarray*}
&&\frac{1}{\varrho^{1-sp}} \int_{B_{3\varrho/2}} \mean{B_{3\varrho/2}} \frac{|u(x){-}u(y)|^{p-1}}{|x{-}y|^{n+sp-1}}   \, dx \, dy\\
&&\quad = \frac{c}{\varrho^{1-sp}} \int_{B_{2\varrho}} \mean{B_{2\varrho}} \frac{|u(x){-}u(y)|^{p-1}}{|x{-}y|^{n+h(p-1)+(sp-1)-h(p-1)}}   \, dx \, dy\\
&&\quad \leq c\varrho^{h(p-1)} \int_{B_{2\varrho}} \mean{B_{2\varrho}} \frac{|u(x){-}u(y)|^{p-1}}{|x{-}y|^{n+h(p-1)}}   \, dx \, dy\leq c \left[ \inf_{B_{r}} u+ \text{\rm Tail}(u _-; {x_0},4r) \right]^{p-1}.
\end{eqnarray*}
\end{proof}

\begin{proof}[Proof of Theorem~\ref{thm:lower}]
In the following all the balls will be centred at $x_0$. 
Let $\{u_j\}$ be an approximating sequence for the SOLA $u$ as described in Definition   \ref{soladef}, with the source terms $\mu_j$ being nonnegative, as prescribed in the assumption of Theorem~\ref{thm:lower}. 
Since $-u_j$ is in particular a weak subsolution to \rif{approssimazione}, Lemma~\ref{lemma:sub bnd} then implies 
$$
\sup_{B_{r/2}} (u_j)_- \leq c \left[  \mean{B_{r}} (u_j)_- \, dx +{\rm Tail}((u_j)_-;x_0,r/2) \right]\;.
$$
By the convergence properties of $\{u_j\}$ and the nonnegativity of $u$ in $B_r$, and recalling that $(-u_j)_+=(u_j)_-$, we can pass to the limit under the sign of integral thereby getting
\begin{equation} \label{eq:inf conv}
\limsup_{j \to \infty } \sup_{B_{r/2}} (u_j)_- \leq c {\rm Tail}(u_-;x_0,r/2)\,,
\end{equation}
for a constant $c$ depending only on $n,s,p,\Lambda$. 
Now, the function 
\eqn{defitilde}
$$\tilde u_j := u_j - \inf_{B_{r/2}} u_j$$ 
is nonnegative in $B_{r/2}$. 
Denote 
\[
 m_{\varrho,j}  := \inf_{B_\varrho} \tilde u_j \qquad \mbox{and} \qquad   T_{\varrho,j} := \text{\rm Tail}((\tilde u_j-m_{\varrho, j})_-; x_0,\varrho )
\]
for $\varrho \in (0,r/2]$. Lemma~\ref{ultimolemma} now yields
\begin{equation} \label{eq:lower iter 1} 
\left[\frac{\mu_j(B_\varrho)}{\varrho^{n-sp}} \right]^{1/(p-1)} \leq  c \left( m_{\varrho, j} - m_{4\varrho, j}  +T_{4\varrho, j} \right)
\end{equation}
for any $\varrho \in (0,r/8]$ for a constant depending only on $n,s,p,\Lambda$. Now, with $M \geq 1$ to be chosen in a few lines, we have that
\begin{eqnarray} \  
\nonumber 
T_{\varrho, j}& = &  c \varrho^{sp/(p-1)} \left[ \int_{\ern \setminus B_\varrho} \frac{(\tilde u_j(x)-m_{\varrho,j})_-^{p-1}} {|x{-}x_0|^{n+sp}} \, dx \right]^{1/(p-1)}
\\  \nonumber  & \leq & c \varrho^{sp/(p-1)} \left[ \int_{\ern \setminus B_{M \varrho}} \frac{(\tilde u_j(x)-m_{\varrho,j})_-^{p-1}}{ |x{-}x_0|^{n+sp}} \, dx \right]^{1/(p-1)}
\\  \nonumber  & & + c \varrho^{sp/(p-1)} \left[ \int_{B_{M \varrho} \setminus B_\varrho} \frac{(\tilde u_j(x)-m_{\varrho,j})_-^{p-1}}{ |x{-}x_0|^{n+sp}} \, dx \right]^{1/(p-1)}
\\  \nonumber  & \leq &  c M^{-sp/(p-1)} \left( m_{\varrho,j}-m_{M\varrho,j} + T_{M\varrho, j} \right)  +  c \left( m_{\varrho,j}-m_{M\varrho,j} \right)
\\  \label{eq:lower tail 2} & \leq & c M^{-sp/(p-1)}  T_{M\varrho, j}   +  c\left( m_{\varrho,j}-m_{M\varrho,j} \right)
\end{eqnarray}
holds for $\varrho \in (0,r/2)$. Notice that we have used the elementary inequality $$(\tilde u_j(x)-m_{\varrho,j})_- \leq (\tilde u_j(x)-m_{M\varrho,j})_- + m_{\varrho, j}-m_{M\varrho,j}\;.$$ Now, with $ t \in (0,r/(8M))$, by integrating \rif{eq:lower tail 2} and changing variables we obtain
\begin{eqnarray} \  
\nonumber
&&  
\int_{t}^{r/(8M)} T_{4\varrho, j} \, \frac{d\varrho}{\varrho} =\int_{4 t}^{r/(2M)} T_{\varrho, j} \, \frac{d\varrho}{\varrho} 
\\ \nonumber && \qquad \leq  c M^{-sp/(p-1)} \int_{t}^{r/(2M)} T_{M\varrho, j}\, \frac{d\varrho}{\varrho} + c \int_{t}^{r/(2M)} (m_{\varrho, j}-m_{M\varrho,j}) \, \frac{d\varrho}{\varrho} 
\\  \nonumber& & \qquad  =  c M^{-sp/(p-1)} \int_{tM}^{r/2} T_{\varrho, j}\, \frac{d\varrho}{\varrho} + 
c\left( \int_{t}^{Mt} m_{\varrho, j}\, \frac{d\varrho}{\varrho} -  \int_{r/(2M)}^{r/2} m_{\varrho, j}\, \frac{d\varrho}{\varrho} \right)\,. 
\end{eqnarray}
Choosing $M\equiv M(s,p)$ so large that
\[
c M^{-sp/(p-1)} = \frac12 \,,
 \]
 we obtain, after easy manipulations
 \[
 \int_{t}^{r/(8M)} T_{4\varrho, j} \, \frac{d\varrho}{\varrho} \leq \int_{r/(2M)}^{r/2} T_{\varrho, j}\, \frac{d\varrho}{\varrho} +  c \left( \int_{t}^{Mt} m_{\varrho, j}\, \frac{d\varrho}{\varrho} -  \int_{r/(2M)}^{r/2} m_{\varrho, j}\, \frac{d\varrho}{\varrho} \right)\,.
 \]
Using~\eqref{eq:lower tail 2} again (this time with $M\varrho = r/2$) we obtain
$$
\int_{r/(2M)}^{r/2} T_{\varrho, j}\, \frac{d\varrho}{\varrho}  \leq c T_{r/2,j} + c \int_{r/(2M)}^{r/2} (m_{\varrho, j} - m_{r/2,j})\, \frac{d\varrho}{\varrho}\,,
$$
and hence also, by changing variables, that 
\begin{eqnarray*} 
 \nonumber \int_{t}^{r/8} T_{4\varrho, j} \, \frac{d\varrho}{\varrho}&=&  \int_{t}^{r/(8M)} T_{4\varrho, j} \, \frac{d\varrho}{\varrho}+\int_{r/(2M)}^{r/2} T_{\varrho, j} \, \frac{d\varrho}{\varrho}\\ \nonumber & \leq &c (m_{t,j}-m_{r/2, j}) + c (m_{r/(2M),j}-m_{r/2,j})+ cT_{r/2,j}\\&\leq & c m_{t,j} + cT_{r/2,j}
\end{eqnarray*}
holds whenever $t < r/(8M)$, for a constant $c \equiv c(n,s,p,\Lambda)$; we have used that the function $\varrho \to m_{\varrho,j}$ is clearly non-increasing and that $m_{r/2,j}=0$. Again integrating~\eqref{eq:lower iter 1}, and using the last inequality, we conclude with
\begin{eqnarray} \  
\nonumber
 \int_{t}^{r/8} \left[\frac{\mu_j(B_\varrho)}{\varrho^{n-sp}} \right]^{1/(p-1)} \, \frac{d\varrho}{\varrho}
 & \leq &
 c  \int_{t}^{r/8} (m_{\varrho, j}-m_{4\varrho,j}) \, \frac{d\varrho}{\varrho} +c   \int_{t}^{r/8}T_{4\varrho, j} \, \frac{d\varrho}{\varrho}
 \\  \nonumber  & \leq & c  \int_{t}^{4t} m_{\varrho,j} \, \frac{d\varrho}{\varrho}+c   \int_{t}^{r/8}T_{4\varrho, j} \, \frac{d\varrho}{\varrho}
  \\  \nonumber  & \leq & c m_{t,j} + cT_{r/2,j}\,, 
\end{eqnarray}
for a constant $c$ depending only on $n,s,p,\Lambda$. 
We now want to let $j \to \infty$ in the above estimate, using the properties of SOLA described in Definition \ref{soladef}. As for the first term $T_{r/8,j}$, by the definition in \rif{defitilde} and recalling that in particular we have that $t \leq r/2$ so that $\inf_{B_{r/2}} u_j \leq (u_j)_t$, we have 
\begin{eqnarray*}
\limsup_{j \to \infty}\, T_{r/2,j}&=& \limsup_{j \to \infty}\,\text{\rm Tail}((\tilde u_j-m_{r/2,j})_-; x_0, r/2 ) \\
& = & \limsup_{j \to \infty}\,\text{\rm Tail}\left(\left(u_j - \inf_{B_{r/2}} u_j\right)_-; x_0, r/2 \right)\\ &
\leq  &c\limsup_{j \to \infty}\, (u_j)_t +  c\limsup_{j \to \infty}\,\text{\rm Tail}((u_j)_-; x_0,r/2)\\&
\leq &   c(u)_t + c{\rm Tail}(u_-; x_0,r/2)\;.
\end{eqnarray*}
Using also~\eqref{eq:inf conv}, we obtain
\begin{eqnarray*}
 \limsup_{j \to \infty}\, m_{t,j} &=& \limsup_{j \to \infty}\, \inf_{B_t}\left( u_j - \inf_{B_{r/2}}  u_j\right)\\& \leq &
 \limsup_{j \to \infty}\,\,(u_j)_{B_t} +  \limsup_{j \to \infty}\, \sup_{B_{r/2}}  (u_j)_-\\ &\leq &
 c \limsup_{j \to \infty}\,\, (u_j)_{B_t} + c\text{\rm Tail}(u_-; x_0,r/2) \\&\leq& c (u)_{B_t} + c\text{\rm Tail}(u_-; x_0,r/2) \;.
\end{eqnarray*}
Moreover,  by the weak convergence of $\mu_j$ to $\mu$ and Lebesgue dominated convergence theorem, we obtain that 
\[
\int_{t}^{r/8} \left[\frac{\mu_j(B_\varrho)}{\varrho^{n-sp}} \right]^{1/(p-1)} \, \frac{d\varrho}{\varrho} \to \int_{t}^{r/8} \left[\frac{\mu(B_\varrho)}{\varrho^{n-sp}} \right]^{1/(p-1)} \, \frac{d\varrho}{\varrho}
\]
as $j \to \infty$ (keep \rif{keep in mind} in mind). Thus we arrive at
\[
 \int_{t}^{r/8} \left[\frac{\mu(B_\varrho)}{\varrho^{n-sp}} \right]^{1/(p-1)} \, \frac{d\varrho}{\varrho} \leq c (u)_{B_t} + 
c \text{\rm Tail}(u_-; x_0,r/2)
\]
for any $t \in (0,r/(8M))$ and where $c \equiv c(n,s,p,\Lambda)$. Now, if ${\bf W}_{s,p}^\mu(x_0,r/8)$ is finite, then by Theorem \ref{pot0} there exists a the precise representative of $u$ at $x_0$, as described in \rif{precise}. This finishes the proof of \rif{stimawolff2} after letting $t \to 0$ in the above display, 
when ${\bf W}_{s,p}^\mu(x_0,r/8)< \infty$. In a similar way, if ${\bf W}_{s,p}^\mu(x_0,r/8)= \infty$, then it \rif{essinf} follows and in any case $x_0$ is a Lebesgue point for $u$.\end{proof}

\end{document}